\documentclass[reqno]{amsart}
\usepackage{hyperref, bbm, amssymb, mathtools, url}
\usepackage{amsmath}
\usepackage{amsthm}
\usepackage{tikz}
\usepackage{tikz-cd}
\usepackage[scr]{rsfso}
\usepackage{comment}
\usepackage{bm}
\usepackage{amssymb}
\usepackage{enumitem}
\usetikzlibrary{matrix,arrows,decorations.pathmorphing, cd}
\usepackage[capitalise]{cleveref}

\def\twocell[#1]{\arrow[#1, dash, phantom, "\Rightarrow"{scale=1.125, yshift=-.4pt, description, allow upside down, sloped, inner sep=0pt}]}

\pgfdeclarelayer{bg}
\pgfsetlayers{bg,main}
\usetikzlibrary{calc}

\newtheorem{theorem}{Theorem}[section]

\newtheorem{corollary}[theorem]{Corollary}
\newtheorem{lemma}[theorem]{Lemma}
\newtheorem{proposition}[theorem]{Proposition}
\theoremstyle{definition}
\newtheorem{construction}[theorem]{Construction}
\newtheorem{convention}[theorem]{Convention}
\newtheorem{definition}[theorem]{Definition}
\newtheorem*{definition*}{Definition}

\newtheorem{example}[theorem]{Example}
\newtheorem{non-ex}[theorem]{Non-example}

\newtheorem{notation}[theorem]{Notation}

\newtheorem{remark}[theorem]{Remark}
\newtheorem*{claim*}{Claim}
\newtheorem{warning}[theorem]{Warning}
\newtheorem{introthm}{Theorem}

\newcommand{\notehelper}[3]{\textcolor{#3}{$\blacksquare$}\marginpar{\ifodd\thepage\raggedright\else\raggedleft\fi\color{#3}\tiny \textbf{#2:} #1}}

\newcommand{\qednow}{\pushQED{\qed}\qedhere\popQED}

\DeclareMathOperator{\Cc}{\mathcal{C}}
\DeclareMathOperator{\Dd}{\mathcal{D}}
\DeclareMathOperator{\Ee}{\mathcal{E}}

\newcommand{\Z}         {{\mathbb{Z}}}

\renewcommand{\phi}{\varphi}
\renewcommand{\epsilon}{\varepsilon}

\renewcommand{\S}{{{\mathscr S}}}
\newcommand{\Sp}{\textup{Sp}}
\DeclareMathOperator{\Spc}{Spc}
\DeclareMathOperator{\SPC}{SPC}
\DeclareMathOperator{\CMON}{CMON}
\DeclareMathOperator{\Cat}{Cat}

\DeclareMathOperator{\Grpd}{Grpd}
\DeclareMathOperator{\Set}{Set}

\newcommand{\PrL}{\textup{Pr}^{\textup{L}}}

\DeclareMathOperator{\essim}{ess\,im}
\DeclareMathOperator{\Hom}{Hom}
\DeclareMathOperator{\Fun}{Fun}
\DeclareMathOperator{\PSh}{PSh}

\DeclareMathOperator{\CMon}{CMon}
\DeclareMathOperator{\RelCat}{RelCat}

\newcommand{\tcat}{{\mathcal T}}
\newcommand{\cat}[1]{\textbf{\textup{#1}}}

\newcommand{\maps}{\textup{Hom}}
\newcommand{\catop}{^{\mathrm{op}}}
\newcommand{\op}{{\textup{op}}}

\renewcommand{\smallint}{{\textstyle\int}}

\newcommand{\groU}{{\mathfrak U}}
\newcommand{\groV}{{\mathfrak V}}
\newcommand{\groW}{{\mathfrak W}}

\DeclareMathOperator{\colim}{colim}
\DeclareMathOperator{\const}{const}
\DeclareMathOperator{\id}{id}
\DeclareMathOperator{\pr}{pr}

\DeclareMathOperator{\BC}{BC}

\DeclareMathOperator{\Orb}{Orb}

\DeclareMathOperator{\Glo}{Glo}

\DeclareMathOperator{\res}{res}
\DeclareMathOperator{\ind}{ind}
\DeclareMathOperator{\coind}{coind}

\DeclareMathOperator{\Nm}{Nm}

\newcommand{\ulhelper}[2]{\underline{\setbox0=\hbox{$#1#2$}\dp0=1pt \box0\relax}}
\newcommand{\ul}[1]{{\mathpalette\ulhelper{#1}}\hbox{\rule[-2pt]{0pt}{0pt}}}
\newcommand{\ulFun}{\ul{\Fun}}

\newcommand{\finSets}{\mathbb{F}}

\newcommand{\ulfinPsets}{\ul{\finSets}_T^{\Orbital}}

\newcommand{\ulfinptdPsets}{\ul{\finSets}_{T,*}^{\Orbital}}

\newcommand{\ulPCMon}{\ul{\CMon}^{\Orbital}}

\newcommand{\Orbital}{P}
\newcommand{\bbU}{\mathbf{U}}
\newcommand{\bbV}{{\mathbf{V}}}

\newcommand{\blank}{{\textup{--}}}

\newcommand{\tcatUn}[1]{\mathop{\hfuzz=10pt\hbox to 0pt{$\textstyle\bm\int$}\kern.3pt\raise.2pt\hbox to 0pt{$\textstyle\bm\int$}\lower.2pt\hbox to 0pt{$\textstyle\bm\int$}\kern.3pt\hbox to 0pt{$\textstyle\bm\int$}\kern-.1pt\raise.1pt\hbox{\color{white}$\textstyle\int$}}}
\newcommand{\GammaS}{{\Gamma\kern-1.5pt\mathscr S}}
\newcommand{\mySp}{{\mathscr S\kern-1.75ptp}}
\newcommand{\mathscrGr}{{\mathscr G\kern-1.25ptr}}

\DeclareMathOperator{\fib}{fib}

\newcommand{\ev}{\mathrm{ev}}
\newcommand{\coev}{\mathrm{coev}}
\let\smashp=\wedge

\newcommand{\iso}{\xrightarrow{\;\smash{\raisebox{-0.5ex}{\ensuremath{\scriptstyle\sim}}}\;}}

\tikzcdset{pullback/.style = {phantom, "\lrcorner"{very near start}}}

\newcommand{\qquadtext}[1]{\qquad\textrm{#1}\qquad}

\newcommand\noloc{%
	\nobreak
	\mspace{6mu plus 1mu}
	{:}
	\nonscript\mkern-\thinmuskip
	\mathpunct{}
	\mspace{2mu}
}

\makeatletter
\def\nop@r@m{nop@r@m}
\newif\ifbr@cketed
\def\br@cketed{}
\let\@uldcite=\cite
\def\st@rcite#1#2{\@uldcite[#2]{#1}}
\def\@mscite#1{\@ifstar{\st@rcite{#1}}{\ifbr@cketed\@uldcite[\br@cketed]{#1}\else\@uldcite{#1}\fi}}
\renewcommand{\cite}[1][\nop@r@m]{\def\temp{#1}\ifx#1\nop@r@m\relax\br@cketedfalse\else\br@cketedtrue\fi\gdef\br@cketed{#1}\@mscite}
\makeatother

\begin{document}
	\title[The universal property of equivariant spectra]{Partial parametrized presentability and\\ the universal property of equivariant spectra}
	\author{Bastiaan Cnossen}
	\address{B.C.: Fakultät für Mathematik, Universität Regensburg, 93040 Regensburg, Germany (\textit{current address}) \&
	Mathematisches Institut, Rheinische Friedrich-Wilhelms-Universit\"at Bonn, Endenicher Allee 60, 53115 Bonn, Germany}
	\author{Tobias Lenz}
	\address{T.L.: Mathematisches Insitut, Rheinische Friedrich-Wilhelms-Universit\"at Bonn, Endenicher Allee 60, 53115 Bonn, Germany (\textit{current address}) \&
	{\hskip0pt minus 1pt} Mathematical{\hskip0pt minus 1pt} Institute,{\hskip0pt minus 1pt} University{\hskip0pt minus 1pt} of{\hskip0pt minus 1pt} Utrecht,{\hskip0pt minus 1pt} Budapestlaan{\hskip0pt minus 1pt} 6,{\hskip0pt minus 1pt} 3584{\hskip0pt minus 1pt} CD{\hskip0pt minus 1pt} Utrecht, The Netherlands}
	\author{Sil Linskens}
	\address{S.L.: Fakultät für Mathematik, Universität Regensburg, 93040 Regensburg, Germany (\textit{current address}) \&
	Mathematisches Institut, Rheinische Friedrich-Wilhelms-Universit\"at Bonn, Endenicher Allee 60, 53115 Bonn, Germany}

	\begin{abstract}
		We introduce a notion of \emph{partial presentability} in parametrized higher category theory and investigate its interaction with the concepts of parametrized semiadditivity and stability from \cite{CLL_Global}. In particular, we construct the free partially presentable $T$-categories in the unstable, semiadditive, and stable contexts and explain how to exhibit them as full subcategories of their fully presentable analogues.

		Specializing our results to the setting of (global) equivariant homotopy theory, we obtain a notion of \emph{equivariant presentability} for the global categories of \cite{CLL_Global}, and we show that the global category of genuine equivariant spectra is the free global category that is both equivariantly presentable and equivariantly stable. As a consequence, we deduce the analogous result about the $G$-category of genuine $G$-spectra for any finite group $G$, previously formulated by \cite{nardin2017thesis}.
	\end{abstract}

	\maketitle

	\begingroup\parskip=.5\parskip
	\tableofcontents
	\endgroup

	\section{Introduction}
	The term \emph{equivariant mathematics} was coined by Balmer and dell'Ambrogio \cite{balmerAmbrogio_Mackey} to refer in a unified way to the study of objects with group actions across a wide range of mathematical disciplines, for example in representation theory or equivariant homotopy theory. Given a group homomorphism $\alpha\colon H \to G$, any $G$-action on an object $X$ can naturally be restricted to an $H$-action, and accordingly most notions of `equivariant objects' give rise to \textit{global categories}: collections of \hbox{($\infty$-)}categories\footnote{We work in the context of higher category theory throughout, and so we will refer to $\infty$-categories simply as `categories.'} $\Cc(G)$ for every finite group $G$ equipped with suitably coherent restriction functors $\alpha^*\colon \Cc(G) \to \Cc(H)$, or more precisely~categories parametrized over the $2$-category $\Glo$ of finite connected groupoids.

	Many fundamental concepts of (higher) category theory have analogues in the world of global categories, leading for instance to notions of \textit{presentability}, \textit{equivariant semiadditivity}, and \textit{equivariant stability}. These properties were introduced and studied by the present authors in the previous article \cite{CLL_Global}, where we in particular showed that the universal presentable, presentable equivariantly semiadditive, and presentable equivariantly stable global categories all admit explicit models in terms of \emph{global homotopy theory} in the sense of \cite{schwede2018global,hausmann-global,g-global}.

	The presentability condition on a global category $\Cc$ used in these results is quite strong: in particular, it demands the existence of left adjoints to \emph{all} restriction functors $\alpha^*\colon \Cc(G) \to \Cc(H)$. This is in fact too strong for certain applications: several interesting examples, like the global category sending $G$ to the category of \emph{genuine $G$-spectra}, only admit such adjoints for \emph{injective} homomorphisms.

	In this article, we will therefore introduce and study a weaker notion of presentability for global categories called \textit{equivariant presentability}, which emphasizes the role of the subgroup inclusions among all group homomorphisms and allows one to capture these additional examples. As our main results, we will show that the universal examples of equivariantly presentable global categories in the unstable, semiadditive, and stable contexts are given by equivariant homotopy theory:

	\begin{introthm}[Universal property of equivariant spaces, \Cref{thm:orbi-S-univ-prop}]
		\label{introthm:orbi-S-univ-prop}
		The global category $\ul\S$ which associates to a finite group $G$ the category $\S_G$ of $G$-spaces is the free equivariantly presentable global category on one generator: for every equivariantly presentable global category $\Dd$, evaluation at the $1$-point space $* \in \ul\S(1)$ induces an equivalence of global categories
		\begin{equation*}
			\ul\Fun^{\textup{eq-cc}}_{\Glo}(\ul\S,\mathcal D)\iso\mathcal D
		\end{equation*}
		where the left hand side denotes a certain global category of `equivariantly cocontinuous' functors.
	\end{introthm}

	\begin{introthm}[Universal property of equivariant special $\Gamma$-spaces, \Cref{thm:gamma-univ-prop}]
		\label{introthm:gamma-univ-prop}
		The global category $\ul{\GammaS}_*^\textup{spc}$ which associates to each finite group $G$ the category of special $\Gamma$-$G$-spaces in the sense of Shimakawa \cite{shimakawa} is the free equivariantly presentable equivariantly semiadditive global category on one generator: for every equivariantly presentable equivariantly semiadditive global category $\mathcal D$ evaluation at the free commutative monoid ${\mathbb P}(*)$ provides an equivalence of global categories
		\begin{equation*}
			\ul\Fun^{\textup{eq-cc}}_{\Glo}(\ul{\GammaS}_*^\textup{spc},\mathcal D)\iso\mathcal D.
		\end{equation*}
	\end{introthm}

	\begin{introthm}[Universal property of genuine equivariant spectra, \Cref{thm:mysp-univ-prop}]
		\label{introthm:mysp-univ-prop}
		The global category $\ul\mySp$ which associates to a finite group $G$ the category $\mySp_G$ of genuine $G$-spectra is the free equivariantly presentable equivariantly stable global category on one generator: for any other such $\mathcal D$ evaluation at the sphere spectrum ${\mathbb S}$ defines an equivalence
		\[
		\Fun^\textup{eq-cc}_{\Glo}(\ul\mySp,\mathcal D)\iso\mathcal D.
		\]
	\end{introthm}

	These results suggest that the original, stronger notion of presentability from \cite{CLL_Global} can be viewed as a characteristic feature of global homotopy theory, distinguishing it from classical equivariant homotopy theory, and we will accordingly use the term \emph{global presentability} for it below.

	From these results one may immediately deduce the analogous universal properties for the \emph{$G$-categories} of $G$-spaces, $\Gamma$-$G$-spaces and $G$-spectra for every fixed finite group $G$; we spell this out in the final section of the article.

	\subsection*{Partial presentability in parametrized higher category theory}
	The above notions of equivariant presentability, semiadditivity, and stability are in fact instances of more general notions defined in the setting of \emph{parametrized higher category theory} as introduced in \cite{exposeI}. Such parametrized notions usually come in various degrees of `parametrized refinement': in particular, \cite{CLL_Global} studied various levels of semiadditivity and stability that can exist in a parametrized category, encoded in the choice of a so-called \emph{atomic orbital} subcategory of the parametrizing category $T$. Equivariant stability and semiadditivity of global categories correspond to the case of the wide subcategory $\Orb\subset\Glo$ of faithful functors.

	To study the analogous situation for presentability of parametrized categories, we introduce \emph{clefts} $S\subset T$ in the present article and associate to each of them a notion of presentability, interpolating between na\"ive, or `fiberwise,' presentability and the full parametrized presentability considered e.g.~in \cite{martiniwolf2022presentable,hilman2022parametrised,CLL_Global}. The aforementioned atomic orbital subcategories are examples of clefts, and equivariant presentability of global categories is again recovered from the case $\Orb\subset \Glo$.

	\begin{definition*}[\Cref{def:S-Presentability}, \Cref{lem:CharacterizationMCocompleteness}]
		A $T$-category $\Cc\colon T\catop \to \Cat$ is said to be \emph{$S$-presentable} if the following conditions are satisfied:
		\begin{enumerate}
			\item $\Cc$ is fiberwise presentable, i.e.~it factors through the non-full subcategory $\PrL\subset \Cat$ of presentable categories and colimit-preserving functors.
			\item For every morphism $f\colon A \to B$ in $S$, the restriction $f^*\colon \Cc(B) \to \Cc(A)$ admits a left adjoint $f_!\colon \Cc(A) \to \Cc(B)$, and these left adjoints satisfy base change for pullbacks along arbitrary maps in $T$ (see \Cref{lem:CharacterizationMCocompleteness} for a precise definition).
		\end{enumerate}
	\end{definition*}

	As one of our key technical results (Theorem~\ref{thm:fracture}), we moreover show how clefts give rise to \emph{fractured $\infty$-topoi} in the sense of \cite{SAG}*{Definition~20.1.2.1}, which allows us to investigate the behavior of partial presentability under changing the parametrizing category along a cleft. Using this `change of parameter' yoga, we then establish analogues of the results from \cite{martiniwolf2021limits, CLL_Global} in the partially presentable world by constructing the free unstable, semiadditive, and stable examples of $S$-presentable $T$-categories, and relating them both to the corresponding universal $S$-presentable $S$-categories as well as $T$-presentable $T$-categories:\footnote{For brevity we only state the stable cases of these theorems here, and we refer the reader to Lemma~\ref{lemma:iota!-ext} and \Cref{cor:spc-s-t-univ-property} resp.~Theorems~\ref{thm:CMon-ST-adj} and \ref{thm:CMon-ST-univ-prop} for the unstable and semiadditive versions.}

	\begin{introthm}[\Cref{thm:sp-univ-property}]
		\label{introthm:Abstract_Universal_Properties}
		Let $S \subset T$ be a cleft, and let $P \subset T$ be an atomic orbital subcategory such that $P\subset S$. Then there exists an $S$-presentable $P$-stable $T$-category $\ul\Sp_{S\triangleright T}^P$ with the following universal property: for any $S$-presentable $P$-stable $T$-category $\mathcal D$, evaluation at a certain object $\mathbb S$ induces an equivalence \[\ul\Fun_T^\textup{$S$-cc}(\ul\Sp^P_{S\triangleright T},\mathcal D)\iso\mathcal D.\]
	\end{introthm}

	The expression `$\ul\Sp_{S\triangleright T}^P$' is to be parsed as `$\left(\ul\Sp_{S}^P\right)_{\triangleright T}$': the underlying $S$-category of $\ul\Sp_{S\triangleright T}^P$ agrees with the free $S$-presentable $P$-stable $S$-category $\ul\Sp_S^P$.

	\begin{introthm}[\Cref{thm:sp-adjunction}]\label{introthm:Abstract_subcat}
		Let $P\subset S\subset T$ be as above and consider the unique $S$-cocontinuous $T$-functor
		\begin{equation*}
			\iota_!\colon\ul\Sp^P_{S\triangleright T}\to\ul\Sp^P_{T}
		\end{equation*}
		sending $\mathbb S$ to $\mathbb S$. Then $\iota_!$ is fully faithful, and its underlying $S$-functor sits in a sequence of $S$-adjoints $\iota_! \dashv \iota^* \dashv \iota_*$.
	\end{introthm}

	This then allows us to deduce Theorems \ref{introthm:orbi-S-univ-prop}, \ref{introthm:gamma-univ-prop}, and \ref{introthm:mysp-univ-prop} from their global analogues established in \cite{CLL_Global}: building on the model categorical results of \cite{g-global}, we show that the global categories $\ul\S$ of equivariant spaces, $\ul\GammaS_*^\text{spc}$ of equivariant special $\Gamma$-spaces, and $\ul\mySp$ of equivariant spectra likewise embed into their global counterparts, and furthermore that the images of these embeddings match up with those on the parametrized side.

	\subsection*{Outlook}
	While Theorem~\ref{introthm:Abstract_subcat} above (together with its unstable and semiadditive versions) explains how to obtain the $S$-presentable universal examples as full subcategories of their $T$-presentable analogues, it is sometimes also possible to go the other way round, and to actually reconstruct the universal fully presentable categories from the partially presentable ones: namely, as the third author shows in \cite{Linskens2023globalization}, under somewhat more restrictive conditions on the pair $S\subset T$ the forgetful functor from $T$-presentable to $S$-presentable $T$-categories admits a left adjoint, which can be explicitly computed in terms of certain partially lax limits. Furthermore this left adjoint preserves the subcategories of $P$-stable $T$-categories for $P\subset S$. Specializing to the inclusion $\Orb\subset\Glo$ again, the main results of the present paper as well as its prequel \cite{CLL_Global} then yield a description of \emph{$G$-global spectra} as a partially lax limit of $H$-equivariant spectra over all homomorphisms $H\to G$, generalizing the result for $G=1$ proven in \cite{LNP}. Note that this application manifestly relies on the universal properties as global categories, instead of their $G$-equivariant shadows for fixed groups $G$.

	\subsection*{Organization}
	We begin by recalling the necessary background on parametrized higher category theory in Section~\ref{sec:prelim}. We then introduce the notion of a cleft in Section~\ref{sec:cleft} and explain its connection to fractured $\infty$-topoi. We moreover show that any atomic orbital subcategory and any right class of a factorization system give rise to a cleft, in particular establishing our key example $\Orb\subset\Glo$.

	Section \ref{sec:Partial-presentability} explains how a cleft $S$ of $T$ yields a well-behaved theory of partial presentability for $T$-categories, and how general (co)limits behave under changing the parametrizing category along a cleft. This allows us to reinterpret and extend work of Martini and Wolf \cite{martiniwolf2021limits} on freely adding $S$-colimits, in particular identifying the free $S$-presentable $T$-category with a full subcategory of the free $T$-presentable $T$-category. In Section \ref{sec:spaces} we use this to describe the free equivariantly presentable global category as the underlying global category of a diagram of model categories of equivariant spaces, proving Theorem~\ref{introthm:orbi-S-univ-prop}.

	In Section \ref{sec:semiadd} we recall the notion of $P$-semiadditivity from \cite{CLL_Global} for atomic orbital subcategories $P\subset T$. Given a cleft $S$ with $P\subset S$, we then construct the free $S$-presentable $P$-semiadditive $T$-category as an extension of the corresponding $S$-category, and we once again exhibit it as a full subcategory of the free $T$-presentable $P$-semiadditive $T$-category. Combining this with results from \cite{CLL_Global}, we then prove Theorem~\ref{introthm:gamma-univ-prop} describing the free equivariantly presentable equivariantly semiadditive global category in terms of equivariant $\Gamma$-spaces in Section~\ref{sec:gamma-orbi}.

	The next two sections are then devoted to the stable case: In Section \ref{sec:stable} we construct the free $S$-presentable $P$-stable $T$-category, and relate it to the corresponding presentable $S$- and $T$-categories, proving Theorems~\ref{introthm:Abstract_Universal_Properties} and~\ref{introthm:Abstract_subcat}. From this we then deduce Theorem~\ref{introthm:mysp-univ-prop} in Section~\ref{sec:spectra}, giving an explicit model of the free equivariantly presentable equivariantly stable global category via equivariant stable homotopy theory.

	In the final section we restrict the global universal properties of Theorems~\ref{introthm:orbi-S-univ-prop},~\ref{introthm:gamma-univ-prop} and~\ref{introthm:mysp-univ-prop} to universal properties for the restricted $G$-categories, and make concrete the results in this context.

	In \Cref{sec:Calculus_Of_Mates} we recall some facts from the calculus of mates used throughout the paper.

	\subsection*{Conventions}
	We work in the context of higher category theory throughout, and refer to $\infty$-categories as `categories.' We fix a chain of Grothendieck universes $\groU\in\groV\in\groW$, and we will use the terms `small category,' `(large) category,' and `very large category' to refer to $\groU$-small, $\groV$-small, and $\groW$-small categories, respectively. A `locally small category' will mean a $\groV$-small category such that all its mapping spaces have $\groU$-small homotopy groups.

	\subsection*{Acknowledgements}
	At the time of writing, B.C.\ and S.L.\ were associate members of the Hausdorff Center for Mathematics at the University of Bonn. B.C.\ was supported by the Max Planck Institute for Mathematics in Bonn. S.L.\ was supported by the DFG Schwerpunktprogramm 1786 ``Homotopy Theory and Algebraic Geometry'' (project ID SCHW 860/1-1).

	\section{Preliminaries on parametrized higher categories}\label{sec:prelim}
	We begin by recalling the necessary background on parametrized higher category theory, as developed in \cite{exposeI, nardin2016exposeIV, shah2021parametrized} and, from the perspective of categories internal to $\infty$-topoi, in \cite{martini2021yoneda, martiniwolf2021limits, martiniwolf2022presentable}. Throughout this section, let us fix a small category $T$.

	\begin{definition}
		A \textit{$T$-category} is a functor $\Cc\colon T\catop \to \Cat$ into the (very large) category of categories. If $\Cc$ and $\Dd$ are $T$-categories, then a \textit{$T$-functor} $F\colon \Cc \to \Dd$ is a natural transformation from $\Cc$ to $\Dd$. The category $\Cat_T$ of $T$-categories is defined as the functor category $\Cat_T \coloneqq \Fun(T\catop, \Cat)$.
	\end{definition}

	\begin{example}\label{ex:Glo}
		\label{ex:globalCategory}
		Define $\Glo$ as the $(2,1)$-category of finite groups, group homomorphisms, and conjugations, i.e.~a 2-morphism $h\colon f \Rightarrow f'$ in $\Glo$ between group homomorphisms $f,f'\colon G \to H$ is an element $h \in H$ such that $f'(g) = hf(g)h^{-1}$ for all $g \in G$. In particular, $\Glo$ comes with a fully faithful functor $B\colon \Glo\hookrightarrow \mathrm{Grpd}$ into the $(2,1)$-category of groupoids which sends a finite group $G$ to the corresponding $1$-object groupoid $BG$. We will use the term \emph{global category} for a $\Glo$-category, \textit{global functor} for a $\Glo$-functor, etc.
	\end{example}

	\begin{example}\label{ex:G-cat}
		For a finite group $G$, let $T = \Orb_G$ be the \emph{orbit category of $G$}, the full subcategory of the $1$-category of $G$-sets spanned by the transitive $G$-sets. Following \cite{exposeI}, we will refer to $\Orb_G$-categories as \textit{$G$-categories}.
	\end{example}

	Let us mention some common examples of $T$-categories:

	\begin{example}\label{ex:T-cat-presheaves}
		Every presheaf $X$ on $T$ gives rise to a $T$-category $\ul{X}\colon T\catop \to \Cat$ by postcomposing with the inclusion $\Spc \hookrightarrow \Cat$ of spaces into categories. In particular, every object $A\in T$ yields a $T$-category $\ul A$ via the Yoneda embedding.
	\end{example}

	\begin{example}\label{ex:T-objects}
		Every category $\Ee$ gives rise to a \textit{$T$-category of $T$-objects} $\ul{\Ee}_T$, given by $\ul{\Ee}_T(B) = \Fun((T_{/B})\catop,\Ee)$ where the functoriality of $T_{/B}$ is given by straightening the cocartesian target fibration $T^{[1]}\to T$.
	\end{example}

	\begin{example}\label{ex:const-global-sections}
		Any category $\Ee$ gives rise a \textit{constant} $T$-category $\const_{\Ee}\colon A\mapsto\mathcal E$. The construction $\Ee \mapsto \const_{\Ee}$ is left adjoint to the \textit{underlying category functor} $\Gamma\colon \Cat_T \to \Cat$ which sends $\Cc$ to $\Gamma(\Cc) \coloneqq \lim_{B \in T\catop} \Cc(B)$.
	\end{example}

	\begin{convention}[{cf.\ \cite[Convention~2.1.15]{CLL_Global}}]\label{conv:limit-ext}
		Any $T$-category $\Cc\colon T\catop \to \Cat$ admits a unique extension to a limit-preserving functor $\PSh(T)\catop \to \Cat$, which we will abusively denote by $\Cc$ again. By convention, all limits and colimits of objects in $T$ are taken in the presheaf category $\PSh(T)$.
	\end{convention}

	\begin{example}
		Viewing $\mathcal C$ as a functor $\PSh(T)\catop \to \Cat$ as above, its value at the terminal presheaf $1$ is given by the underlying category $\Gamma(\mathcal C)$ of $\Cc$, in the sense of Example~\ref{ex:const-global-sections}.
	\end{example}

	\begin{example}
		The category $\Cat_T$ is cartesian closed, i.e.~given $T$-categories $\Cc$ and $\Dd$, there is a $T$-category $\ulFun_T(\Cc,\Dd)$ of $T$-functors, characterized by the property that there is a natural equivalence
		\begin{align*}
			\maps(\Ee \times \Cc, \Dd) \simeq \maps(\Ee,\ulFun_T(\Cc,\Dd))
		\end{align*}
		for every third $T$-category $\Ee$. We let
		\begin{align*}
			\Fun_T(\Cc,\Dd) \coloneqq \Gamma(\ulFun_T(\Cc,\Dd))
		\end{align*}
		denote the underlying category of $\ulFun_T(\Cc,\Dd)$. By adjunction, its objects can be identified with $T$-functors $\Cc\simeq\Cc\times\const_{[0]} \to \Dd$, while its morphisms are \emph{natural transformations} of $T$-functors, i.e.~functors $\Cc\times\const_{[1]}\to\Dd$.
	\end{example}

	To describe these functor categories more explicitly, we will use:

	\begin{lemma}[Categorical Yoneda lemma, {\cite[Corollary~2.2.8]{CLL_Global}}]
		\label{lem:YonedaLemma}
		For every presheaf $B \in \PSh(T)$ and every $T$-category $\Cc$, there is an equivalence of categories
		\[
		\Fun_T(\ul{B},\Cc) \iso \Cc(B),
		\]
		natural in both variables, determined by the fact that for $B \in T$ it is given by evaluation at the identity $\id_B \in \maps_T(B,B) = \ul{B}(B)$.\qed
	\end{lemma}

	Combining this with the (internal) adjunction equivalence for $\ul\Fun_T$ we immediately get:

	\begin{corollary}[cf.~\cite{CLL_Global}*{Corollary~2.2.9}]
		Let $\Cc,\Dd\in\Cat_T$ and $X\in\PSh(T)$. There are natural equivalences
		\begin{equation*}
			\ul\Fun_T(\Cc,\Dd)(X)\simeq\Fun_T(\Cc\times\ul X,\Dd)\simeq\Fun_T\!\big(\Cc,\ul\Fun_T(\ul X,\Dd)\big).\qednow
		\end{equation*}
	\end{corollary}

	In particular, we can (and will) identify objects of $\ul\Fun_T(\Cc,\Dd)(X)$ with $T$-functors $\Cc\times\ul X\to\Dd$ or equivalently $\Cc\to\ul\Fun_T(\ul X,\Dd)$.

	\begin{example}\label{ex:spc-T}
		As $\PSh(T)$ has pullbacks, the target map $\PSh(T)^{[1]}\to\PSh(T)$ is a \emph{cartesian} fibration, so we can straighten it to a functor \[\ul\Spc_T\coloneqq\PSh(T)_{/\bullet}\colon\PSh(T)^\op\to{\Cat}.\] Explicitly, this sends $X\in\PSh(T)$ to the slice $\PSh(T)_{/X}$ and a map $f\colon Y\to X$ to the pullback functor $f^*\colon\PSh(T)_{/Y}\to\PSh(T)_{/X}$. By \cite{HTT}*{Theorem~6.1.3.9 and Proposition~6.1.3.10}, this functor preserves limits, so it defines a $T$-category via our convention.

		As the notation suggests, this can be identified with the $T$-category of $T$-objects (Example~\ref{ex:T-objects}) in $\Spc$: \cite{CLL_Global}*{Remark~2.1.16} constructs an equivalence between the two which is given in degree $A\in T$ by the colimit extension $\PSh(T_{/A})\to\PSh(T)_{/A}$ of the slice of the Yoneda embedding $T\to\PSh(T)$ over $A$.
	\end{example}

	\subsection{Adjunctions}In $\Cat_T$ there is a natural notion of (internal) \emph{adjunctions}: a $T$-functor $F\colon\mathcal C\to\mathcal D$ is left adjoint to $G\colon\mathcal D\to\mathcal C$ if there are natural transformations $\eta\colon\id\Rightarrow GF$ and $\epsilon\colon FG\Rightarrow\id$ satisfying the triangle identities up to homotopy. We will frequently rely on the following `pointwise criterion' for adjoints:

	\begin{proposition}[see \cite{martiniwolf2021limits}*{Proposition~3.2.9 and Corollary~3.2.11}]\label{prop:adjoints-BC}
		A functor $F\colon\Cc\to\Dd$ of $T$-categories admits a right adjoint if and only if the following hold:
		\begin{enumerate}
			\item For every $A\in T$ the functor $F_A\colon\Cc(A)\to\Dd(A)$ admits a right adjoint $G_A$.
			\item For every $f\colon A\to B$ in $T$ the Beck--Chevalley transformation $f^*\circ G_B\Rightarrow G_A \circ f^*$ given by the composite
			\begin{equation*}
				f^* G_B\xRightarrow{\;\eta\;} G_A F_A f^* G_B\xRightarrow{\;\sim\;} G_A f^*  F_B G_B\xRightarrow{\;\epsilon\;} G_A f^*
			\end{equation*}
			is an equivalence.
		\end{enumerate}
		Moreover, in this case the following hold:
		\begin{enumerate}
			\item[(1\/$'\!$)] For every $X\in\PSh(T)$ the functor $F_X\colon\Cc(X)\to\Dd(X)$ (cf.~Convention~\ref{conv:limit-ext}) admits a right adjoint $G_X$.
			\item[(2\/$'\!$)] For every $f\colon X\to Y$ in $\PSh(T)$ the Beck--Chevalley map $G_X f^*\Rightarrow f^* G_Y$ is an equivalence.
		\end{enumerate}
		Finally, the right adjoint $G$ is given in degree $X\in\PSh(T)$ by $G_X$ as above and the unit and counit are given pointwise by the unit and counit of $F_X\dashv G_X$.\qed
	\end{proposition}

	We refer the reader to \Cref{sec:Calculus_Of_Mates} for more information on Beck--Chevalley transformations.

	\subsection{Limits and colimits} Next, we come to parametrized notions of limits and colimits. While this can be developed `internally' using the notions of parametrized adjunctions and parametrized functor categories, we will instead take a purely `pointwise' perspective in the spirit of the previous proposition in this paper.

	\begin{remark}
		Below we will for simplicity restrict ourselves to the case of colimits; the theory of limits is then formally dual.
	\end{remark}

	\begin{definition}\label{defi:fiberwise-cc}
		A $T$-category $\Cc$ is called \emph{fiberwise cocomplete} if $\mathcal C(A)$ is cocomplete for every $A\in T$ and the restriction $f^*\colon \mathcal C(B)\to\mathcal C(A)$ is cocontinuous for every $f\colon A\to B$. A $T$-functor $F\colon\mathcal C\to\mathcal D$ is called \emph{fiberwise cocontinuous} if $F_A\colon\mathcal C(A)\to\mathcal D(A)$ is cocontinuous for every $A\in T$.
	\end{definition}

	Note that in the above situation $\Cc(X)$ is more generally cocomplete for any $X\in\PSh(T)$, and for any $f\colon X\to Y$ in $\PSh(T)$ the restriction $f^*\colon \Cc(Y)\rightarrow \Cc(X)$ is cocontinuous, see \cite{HTT}*{Corollary~5.1.2.3 and Lemma~5.4.5.5}.

	\begin{definition}\label{defi:U-cc}
		Let $\bbU\subset\ul\Spc_T$ be any $T$-subcategory. We say that a $T$-category $\Cc$ \emph{admits $\bbU$-colimits}, or that $\Cc$ is \textit{$\bbU$-cocomplete}, if the following conditions are satisfied:
		\begin{enumerate}
			\item For every $D\in\PSh(T)$ and every $(f\colon C\to D)\in\bbU(D)$ the restriction $f^*\colon\Cc(D)\to\Cc(C)$ admits a left adjoint $f_!$.
			\item For any pullback
			\begin{equation*}
				\begin{tikzcd}
					A\arrow[r, "g"]\arrow[d, "u"']\arrow[dr,phantom,"\lrcorner"{very near start}] & B\arrow[d, "t"]\\
					C\arrow[r, "f"'] & D
				\end{tikzcd}
			\end{equation*}
			in $\PSh(T)$ such that $f\in\bbU(D)$ (and hence $g\in\bbU(B)$ as $\bbU$ is a $T$-subcategory), the Beck--Chevalley transformation $g_!u^*\Rightarrow t^*f_!$ is an equivalence.
		\end{enumerate}
		If $\Dd$ is another $\bbU$-cocomplete $T$-category, then a $T$-functor $F\colon\Cc\to\Dd$ is called \emph{$\bbU$-cocontinuous} if for every $(f\colon C\to D)\in\bbU(D)$ the Beck--Chevalley map $f_! F_C\Rightarrow F_D f_!$ is an equivalence.
	\end{definition}

	\begin{remark}\label{rk:U-cc-restriction-to-repr}
		In the definition of a $\bbU$-cocomplete $T$-category, it suffices that the above conditions are satisfied whenever $B$ and $D$ are representable, see~\cite{CLL_Global}*{Remark~2.3.15}, and likewise for $\bbU$-cocontinuity.
	\end{remark}

	\begin{definition}
		\label{def:T-cc}
		A $T$-category $\Cc$ is called \emph{$T$-cocomplete} if it is fiberwise cocomplete (Definition~\ref{defi:fiberwise-cc}) and $\ul\Spc_T$-cocomplete (Definition~\ref{defi:U-cc}).

		Similarly, a $T$-functor $F\colon\Cc\to\Dd$ between $T$-cocomplete $T$-categories is called \emph{$T$-cocontinuous} if it is fiberwise cocontinuous and $\ul\Spc_T$-cocontinuous.
	\end{definition}

	\begin{example}\label{ex:T-spc-cc}
		The $T$-category $\ul\Spc_T$ is $T$-cocomplete, see \cite{martiniwolf2021limits}*{Example~5.2.8}.
	\end{example}

	\begin{example}
		If $\Dd$ is $\bbU$-cocomplete for some $\bbU\subset\ul\Spc_T$, and $\Cc$ is any $T$-\hskip0ptcategory, then $\ul\Fun_T(\Cc,\Dd)$ is again $\bbU$-cocomplete, see~\cite{CLL_Global}*{Corollary~2.3.25}.
	\end{example}

	\begin{example}\label{ex:la-are-cc}
		Any left adjoint $F\colon\Cc\to\Dd$ of $T$-cocomplete $T$-categories is $T$-cocontinuous: it is fiberwise cocontinuous since each functor $F_A\colon \Cc(A) \to \Dd(A)$ is a left adjoint, and for a morphism $f\colon A \to B$ the Beck--Chevalley map from Definition~\ref{defi:U-cc} is simply the total mate of the Beck--Chevalley map from Proposition~\ref{prop:adjoints-BC}. Conversely, a functor of $T$-cocomplete categories is a $T$-left adjoint if and only if it is $T$-cocontinuous and admits a pointwise right adjoint.
	\end{example}

	\begin{definition}\label{def:cocont-functor-cat}
		Let $\bbU\subset\ul\Spc_T$ be any $T$-subcategory. For any $\bbU$-cocomplete $T$-categories $\Cc,\Dd$ and any $X\in\PSh(T)$ we write $\ul\Fun_T^\textup{$\bbU$-cc}(\Cc,\Dd)(X)\subset\ul\Fun_T(\Cc,\Dd)(X)$ for the full subcategory spanned by those functors $F\colon\Cc\to\ul\Fun_T(\ul X,\Dd)$ that are $\bbU$-cocontinuous.

		Similarly, we define $\ul\Fun_T^\textup{$T$-cc}(\Cc,\Dd)(X)$ whenever $\Cc$ and $\Dd$ are $T$-cocomplete.
	\end{definition}

	By \cite{martiniwolf2021limits}*{Remark~4.2.1} the above define $T$-subcategories of $\ul\Fun_T(\Cc,\Dd)$.

	\begin{remark}
		The articles \cite{martiniwolf2021limits} and \cite{CLL_Global} use an a priori different definition of $\ul\Fun_T^\textup{$\bbU$-cc}$ and $\ul\Fun_T^\text{$T$-cc}$, see \cite{CLL_Global}*{Proposition~
			2.3.26 and Remark~2.3.27} for the equivalence to the above.
	\end{remark}

	With this terminology at hand, we can now formulate the universal property of $T$-spaces:

	\begin{theorem}[\cite{martiniwolf2021limits}*{Theorem~7.1.1}]\label{thm:univ-prop-T-spc}
		For any $T$-cocomplete $\mathcal D$, evaluation at the terminal object defines an equivalence of $T$-categories
		\begin{equation*}
			\ul\Fun_T^\textup{$T$-cc}(\ul\Spc_T,\Dd)\iso\Dd.\qednow
		\end{equation*}
	\end{theorem}

	\subsection{Presentability} Finally, we come to the notion of \emph{presentability} for $T$-\hskip0ptcategories from \cite{martiniwolf2022presentable}:

	\begin{definition}\label{defi:fib-pres}
		A $T$-category $\Cc\colon T^\op\to\Cat$ is called \emph{fiberwise presentable} if it factors through the non-full subcategory $\PrL\subset\Cat$ of presentable categories and left adjoint functors.
	\end{definition}

	In this case the limit extension again factors through $\PrL$, i.e.~for any $X\in\PSh(T)$ the category $\Cc(X)$ is presentable, and for any map $f\colon X\to Y$ of presheaves the restriction $f^*\colon\Cc(Y)\to\Cc(X)$ is a left adjoint, see~\cite{HTT}*{Proposition~5.5.3.13}.

	\begin{definition}\label{defi:T-pres}
		A $T$-category is called \emph{$T$-presentable} if it is fiberwise presentable and $T$-cocomplete.
	\end{definition}

	\begin{remark}\label{rk:T-pres-implies-limits}
		Any $T$-presentable category is also $T$-complete, see \cite{martiniwolf2022presentable}*{Corollary~6.2.5}.
	\end{remark}

	\begin{example}
		The $T$-category $\ul\Spc_T$ of $T$-spaces is $T$-presentable: clearly each $\PSh(T)_{/X}$ is presentable, each $f^*\colon\PSh(T)_{/Y}\to\PSh(T)_{/X}$ is a left adjoint by local cartesian closedness, and finally $\ul\Spc_T$ is $T$-cocomplete by Example~\ref{ex:T-spc-cc}.
	\end{example}

	\begin{example}
	\label{ex:Functor_Cat_Presentable}
		If $\Cc$ is small and $\Dd$ is $T$-presentable, then $\ul\Fun_T(\Cc,\Dd)$ is again $T$-presentable, see \cite{martiniwolf2022presentable}*{Corollary~6.2.6}.
	\end{example}

	\begin{remark}
		Let $\Cc$ be $T$-presentable and $\Dd$ be locally small and $T$-cocomplete. Combining Example~\ref{ex:la-are-cc} with the usual non-parametrized Special Adjoint Functor Theorem \cite{HTT}*{Corollary 5.5.2.9(1)}, we see that a $T$-functor $\Cc\to\Dd$ is a left adjoint if and only if it is $T$-cocontinuous.
	\end{remark}

	\section{Cleft categories}\label{sec:cleft}
	Let $T$ be a small category and let $S\subset T$ be a (wide) subcategory. Associated to the inclusion $\iota\colon S\hookrightarrow T$ we have a natural restriction functor $\iota^*\colon\Cat_T\to\Cat_S$, which admits both a left adjoint $\iota_!$ as well as a right adjoint $\iota_*$, given by left and right Kan extension, respectively. One of the central questions of the present paper is under which conditions the adjunction $\iota^*\dashv\iota_*$ interacts nicely with parametrized concepts, and in particular with the notions of parametrized colimits for $T$-categories and $S$-categories discussed above.

 	To address this question, we make use of a more `geometric' description of the adjunction $\iota^* \dashv \iota_*$. By identifying $T$-categories with limit-preserving functors on $\PSh(T)\catop$ as in Convention~\ref{conv:limit-ext}, we see that precomposition with any colimit-preserving functor $f\colon\PSh(S)\to\PSh(T)$ determines a functor $f^*\colon \Cat_T \to \Cat_S$. The left Kan extension functor $\iota_!\colon \PSh(S) \to \PSh(T)$ satisfies $\iota_!A= A$, so applying this construction to $f=\iota_!$ recovers $\iota^*\colon \Cat_T \to \Cat_S$, and consequently the right adjoint $\iota_*\colon \Cat_S \to \Cat_T$ of $\iota^*$ is obtained by precomposition with $\iota^*\colon \PSh(T) \to \PSh(S)$, with the unit and counit of the adjunction $\iota^*\colon\Cat_T\rightleftarrows\Cat_S \noloc\iota_*$ given by plugging in the unit and counit of the adjunction $\iota_!\colon\PSh(S)\rightleftarrows\PSh(T)\noloc\iota^*$.

    The above description suggests that we can understand the category theoretic behavior of the adjunction $\iota^*\colon\Cat_T\rightleftarrows\Cat_S \noloc\iota_*$ in terms of the geometric, or topos-theoretic, behavior of the adjunction $\iota_!\colon\PSh(S)\rightleftarrows\PSh(T)\noloc\iota^*$. As a concrete example, consider the question of whether $\iota^*\colon\Cat_T\to\Cat_S$ preserves cocompleteness. If $\Cc \in \Cat_T$ is $T$-cocomplete, it is easy to see that the $S$-category $\iota^*\Cc$ is fiberwise cocomplete and that its restriction functors admit pointwise left adjoints, without any restrictions on $\iota$. However, the Beck--Chevalley condition for these adjoints does not always hold: it translates to the requirement that $\iota_!\colon \PSh(S) \to \PSh(T)$ preserves pullbacks. Similarly, one can translate cocontinuity of the unit $\Cc\to\iota_*\iota^*\Cc$ into a pullback condition on the \emph{counit} of $\iota_!\dashv\iota^*$. Upon closer inspection, it turns out that all the required conditions we will need for a well-behaved theory can be nicely summed up in Lurie's notion of a \emph{fractured $\infty$-topos} \cite{SAG}*{Definition~20.1.2.1}:

	\begin{definition}
		\label{defi:fracture-subcategory}
		Let $\mathcal{X}$ be an $\infty$-topos. A functor $j_!\colon \mathcal{Y} \to \mathcal{X}$ is called a \textit{fracture subcategory} if the following conditions are satisfied:
		\begin{enumerate}[label=(F\arabic*), start=0]
			\item The functor $j_!$ is a monomorphism of categories, i.e.\ it is faithful and the induced functor on groupoid cores is even fully faithful.
			\item \label{it:F1_Preserves_Pullbacks} The functor $j_!\colon \mathcal{Y} \to \mathcal{X}$ preserves pullbacks.
			\item The functor $j_!\colon \mathcal{Y} \to \mathcal{X}$ admits a right adjoint $j^*\colon \mathcal{X} \to \mathcal{Y}$ which is conservative and preserves colimits.
			\item \label{it:F3_Counit_Pullback_Square} For every morphism $f\colon X \to Y$ in $\mathcal{Y}$, the naturality square
			\[
			\begin{tikzcd}
				j_!j^*j_!X\arrow[r, "j_!j^*j_!f"]\arrow[d, "\epsilon j_!"'] &[1em] j_!j^*j_!Y\arrow[d, "\epsilon j_!"]\\
				j_!X\arrow[r, "j_!f"'] & j_!Y
			\end{tikzcd}
			\]
			of the counit transformation $\epsilon\colon j_!j^* \to \id$ is a pullback square in $\mathcal{X}$.
		\end{enumerate}
		An $\infty$-topos $\mathcal{X}$ equipped with a fracture subcategory $\mathcal{Y}$ is called a \textit{fractured $\infty$-topos}.
	\end{definition}

	However, these axioms are quite strong, making them somewhat hard to check directly. Accordingly, before coming to the parametrized applications of fractured $\infty$-topoi sketched above, we devote the present section to their construction from simpler, less geometric data. Namely, as in the introductory example we will be interested in the special case of functors $\PSh(S)\to\PSh(T)$ arising as left Kan extension along the inclusion $S\hookrightarrow T$ of a wide subcategory. It turns out that in this can case we can give a more explicit characterization in terms of the indexing categories $S$ and $T$:

	\begin{definition}\label{defi:cleft}
		Let $T$ be a small category. A wide subcategory $S\subset T$ is called a \emph{cleft} of $T$ if the following conditions are satisfied:
		\begin{enumerate}[label=(C\arabic*), ref=C\arabic*]
			\item \label{item:F-canc} The subcategory $S$ contains all equivalences of $T$ and is left-cancellable, i.e.~whenever $f$ and $g$ are composable maps in $T$ with $g \in S$ and $gf\in S$, then $f\in S$.
			\item\label{item:F-pb} For any map $f\colon A\to B$ in $S$ and any map $g\colon B'\to B$ in $T$ there exists a map $f'\colon X'\to B'$ in $\PSh(S)$ and a pullback square
			\begin{equation*}
				\begin{tikzcd}
					\iota_!X'\arrow[d]\arrow[r, "\iota_!f'"] &\iota_!B'\arrow[d, "g"]\\
					\iota_!A\arrow[r, "\iota_!f"']& \iota_!B
				\end{tikzcd}
			\end{equation*}
			in $\PSh(T)$, where $\iota_!\colon\PSh(S)\to\PSh(T)$ denotes left Kan extension along the inclusion $\iota\colon S\hookrightarrow T$.
			\item\label{item:F-idem} If $\alpha\colon A\to B$, $\beta\colon B\to A$ are maps in $T$ such that $\beta\alpha=\id_A$ and $\alpha\beta$ is a map in $S$, then also $\alpha$ belongs to $S$ (whence so does $\beta$ by left cancellability).
		\end{enumerate}
		We call a small category $T$ equipped with a cleft $S\subset T$ a \emph{cleft category}.
	\end{definition}

	\begin{remark}
		As Axiom (\ref{item:F-idem}) might look somewhat exotic, we record several more familiar properties that imply this axiom:
		\begin{enumerate}
			\item[(\ref*{item:F-idem}$'$)] Any idempotent $e\colon B \to B$ in $S$ is the identity.
			\makeatletter\protected@edef\@currentlabel{\ref*{item:F-idem}$'$}\makeatother \label{item:F-idem-equ}
			\item[(\ref*{item:F-idem}$''$)] The morphisms of $S$ are closed under retracts in the arrow category of $T$.
			\makeatletter\protected@edef\@currentlabel{\ref*{item:F-idem}$''$}\makeatother\label{item:F-retract}
			\item[(\ref*{item:F-idem}$'''$)] \makeatletter\protected@edef\@currentlabel{\ref*{item:F-idem}$'''$}\makeatother\label{item:F-2-of-6}The morphisms of $S$ satisfy the \emph{restricted $2$-out-of-$6$ property}: given composable $f,g,h$ in $T$ such that $hg$ and $gf$ belong to $S$, so does $f$.
		\end{enumerate}
		Indeed, to see that (\ref{item:F-idem-equ}) implies (\ref{item:F-idem}), note that the map $\alpha\beta\colon B \to B$ is an idempotent in $S$ and thus the identity. It follows that $\alpha$ and $\beta$ are (mutually inverse) equivalences, hence belong to $S$ by (\ref{item:F-canc}). In case of (\ref{item:F-retract}), it suffices to observe that the diagram
		\begin{equation*}
			\begin{tikzcd}
				A\arrow[r, "\alpha"]\arrow[d, "\alpha"'] & B\arrow[r, "\beta"] \arrow[d, "\alpha\beta"{description}] & A\arrow[d, "\alpha"]\\
				B\arrow[r, equal] & B \arrow[r, equal] & B
			\end{tikzcd}
		\end{equation*}
		expresses $\alpha$ as a retract of $\alpha\beta$. Finally, applying (\ref{item:F-2-of-6}) to the chain
		$\alpha,\beta,\alpha$ also implies (\ref{item:F-idem}).
	\end{remark}

	\begin{remark}
		\label{rmk:AdmissibilityGivescleft}
		Axiom (\ref{item:F-pb}) is a relaxation of the following more familiar condition:
		\begin{enumerate}
			\item[(\ref*{item:F-pb}$'$)] Pullbacks of maps in $S$ along maps in $T$ exist in $T$ and belong to $S$.\makeatletter\protected@edef\@currentlabel{\ref*{item:F-pb}$'$}\makeatother\label{item:F-pb-strong}
		\end{enumerate}
		Wide subcategories $S \subset T$ satisfying axioms (\ref{item:F-canc}), (\ref{item:F-pb-strong}) and (\ref{item:F-retract}) are called \textit{admissibility structures} in \cite{SAG}*{Definition~20.2.1.1}. In particular, every admissibility structure on $T$ is also a cleft in the above sense.
	\end{remark}

	Let us mention some examples of cleft categories:

	\begin{example}[Trivial clefts]
		Every category $T$ admits two extremal clefts: letting $S$ consist of all maps in $T$ constitutes the maximal cleft on $T$, while letting $S$ consist of only the equivalences of $T$ constitutes the minimal cleft on $T$.
	\end{example}

	\begin{example}[Factorization systems]
		\label{ex:factorization-cleft}
		Let $(E,M)$ be a \emph{factorization system} on $T$. We will prove in Proposition~\ref{prop:factorization-cleft} below that the right class $M$ is a cleft.
	\end{example}

	\begin{example}[Atomic orbital subcategories]
		\label{ex:atomic-orbital-cleft}
		Let $P\subset T$ be an \emph{atomic orbital subcategory} in the sense of \cite{CLL_Global}*{Definition~4.3.1}. We will prove in Proposition~\ref{prop:fission-cleft} below that $P\subset T$ is a cleft category.
	\end{example}

	\begin{example}\label{ex:Orb}
		Recall the \emph{global indexing category} $\Glo$ from Example~\ref{ex:Glo}. We define a wide subcategory $\Orb\subset\Glo$ spanned by the \emph{injective} homomorphisms. Then $\Orb\subset\Glo$ is a cleft category: this follows either from \Cref{ex:atomic-orbital-cleft} together with \cite{CLL_Global}*{Example~4.3.3} or from Example~\ref{ex:factorization-cleft} with \cite{LNP}*{Proposition~6.14}.
	\end{example}

	\subsection{Clefts vs.\ fractures} As promised, we will prove as the main result of this section:

	\begin{theorem}\label{thm:fracture}
		For a wide subcategory $S\subset T$, the following are equivalent:
		\begin{enumerate}
			\item The subcategory $S$ is a cleft (Definition~\ref{defi:cleft}).\label{item:frac-fiss}
			\item The left Kan extension functor $\iota_!\colon\PSh(S)\to\PSh(T)$ along the inclusion $\iota\colon S \hookrightarrow T$ is a fracture subcategory (Definition~\ref{defi:fracture-subcategory}).\label{item:frac-frac}
		\end{enumerate}
	\end{theorem}

	\begin{remark}
		In the special case where $S \subset T$ defines an admissibility structure on $T$, cf.\ \Cref{rmk:AdmissibilityGivescleft}, the implication $(\ref{item:frac-fiss})\Rightarrow(\ref{item:frac-frac})$ was already proved by Lurie in \cite{SAG}*{Theorem~20.2.4.1}. In the examples we care about, and in particular for the inclusion $\Orb\subset\Glo$, the stronger Axiom (\ref{item:F-pb-strong}) of an admissibility structure is not satisfied: the required pullbacks do \emph{not} exist before passing to presheaves. The above strengthening of Lurie's result will therefore be crucial for our purposes.
	\end{remark}

	The proof of Theorem~\ref{thm:fracture} will occupy this whole subsection; it is somewhat involved and may be skipped on a first reading.

	While both implications are non-trivial, most of the work will go into proving that clefts give rise to fractured $\infty$-topoi. We therefore fix a cleft category $\iota\colon S\hookrightarrow T$, and we start with some elementary consequences of the axioms.

	\begin{lemma}\label{lemma:iota!-ff}
		The functor $(\iota_!)_{/A}\colon \PSh(S)_{/A} \to \PSh(T)_{/A}$ is fully faithful for any object $A \in S$.
	\end{lemma}
	\begin{proof}
		As recalled in Example~\ref{ex:spc-T}, $(\iota_!)_{/A}$ may be identified with the functor $(\iota_{/A})_!\colon \PSh(S_{/A}) \to \PSh(T_{/A})$ given by left Kan extension along $\iota_{/A}\colon S_{/A}\to T_{/A}$. By Axiom (\ref{item:F-canc}), $S$ is left cancellable, so that $\iota_{/A}$ is fully faithful. Thus, also the Kan extension $(\iota_{/A})_!$ is fully faithful, whence so is $(\iota_!)_{/A}$.
	\end{proof}

	\begin{lemma}\label{lemma:iota!-essim}
		For every $g\colon A \to B$ in $T$, the pullback functor $g^*\colon \PSh(T)_{/B} \to \PSh(T)_{/A}$ sends the essential image of $(\iota_!)_{/B}$ to the essential image of $(\iota_!)_{/A}$.

		\begin{proof}
		 The functor $g^*\colon \PSh(T)_{/B} \to \PSh(T)_{/A}$ preserves colimits as $\PSh(T)$ is an $\infty$-topos. Since the essential image of the fully faithful left adjoint $(\iota_!)_{/A}$ is closed under colimits, it will be enough to show that $g^*$ maps any element of the form $\iota_!f\colon\iota_!X\to\iota_!B$ for $f\colon X\to B$ a map in $S$ into the essential image of $(\iota_!)_{/A}$. This is precisely Axiom (\ref{item:F-pb}), finishing the proof.
		\end{proof}
	\end{lemma}

		\begin{construction}
			\label{constr:iota^*}
			The functor $\iota^*\colon\PSh(T)\to\PSh(S)$ preserves pullbacks, so it induces a map
			\begin{equation}\label{eq:iota*-cart}
				\PSh(T)^{[1]}\to\PSh(S)^{[1]}\times_{\PSh(S)}\PSh(T)
			\end{equation}
			of cartesian fibrations over $\PSh(T)$.

			We now define the $S$-functor $\iota^*\colon \iota^*\ul\Spc_T \to \ul\Spc_S$ as the composite
			\[
			\iota^*\ul\Spc_T = \PSh(T)_{/\iota_!(\bullet)} \longrightarrow \PSh(S)_{/\iota^*\iota_!(\bullet)} \xrightarrow{\;\eta^*\,} \PSh(S)_{/\bullet} = \ul\Spc_S,
			\]
			where the first map is obtained from the straightening of $(\ref{eq:iota*-cart})$ by restricting along $\iota_!\colon \PSh(S) \to \PSh(T)$, while the second map is obtained by pullback along the unit transformation $\eta\colon \id \Rightarrow \iota^*\iota_!$.
		\end{construction}

		\begin{lemma}\label{lemma:iota!-pullback}
			The left Kan extension functor $\iota_!\colon\PSh(S)\to\PSh(T)$ preserves pullbacks.
			\begin{proof}
				For any $X\in\PSh(S)$, the above functor $\iota^*\colon\PSh(T)_{/\iota_!X}\to\PSh(S)_{/X}$ admits a left adjoint given by $(\iota_!)_{/X}$. The lemma then precisely amounts to saying that the Beck--Chevalley map $(\iota_!)_{/X}g^*\Rightarrow g^*(\iota_!)_{/Y}$ is an equivalence for any $g\colon X\to Y$ in $\PSh(S)$. By Propositions~\ref{prop:adjoints-BC} it suffices to check this in the case that $X$ and $Y$ are representable, i.e.~$g$ is a map in $S$. Since the functors $(\iota_!)_{/X}$ and $(\iota_!)_{/Y}$ are fully faithful in this case by Lemma~\ref{lemma:iota!-ff}, the Beck--Chevalley condition is equivalent to the condition that $g^*$ preserves their essential images, which is an instance of Lemma~\ref{lemma:iota!-essim}.
			\end{proof}
		\end{lemma}

		\begin{construction}
			As a consequence of the previous lemma, $\iota_!$ induces a map $\PSh(S)^{[1]}\to\PSh(T)^{[1]}\times_{\PSh(T)}\PSh(S)$ of cartesian fibrations, which we straighten to an $S$-functor $\iota_!\colon\ul\Spc_S\to\iota^*\ul\Spc_T$. For any presheaf $X$ in $\PSh(S)$, this is given by $(\iota_!)_{/X}\colon\PSh(S)_{/X}\to\PSh(T)_{/\iota_!X}$. This is fully faithful for representable $X$ by \Cref{lemma:iota!-ff}, and hence also for general $X$.
		\end{construction}

		\begin{lemma}\label{lemma:iota!-iota*-adj}
			The $S$-functor $\iota_!\colon\ul\Spc_S\hookrightarrow\iota^*\ul\Spc_T$ is left adjoint to the $S$-functor $\iota^*\colon\iota^*\ul\Spc_T \to \ul\Spc_S$ from Construction~\ref{constr:iota^*}.
			\begin{proof}
				We have already seen in the proof of Lemma~\ref{lemma:iota!-pullback} that $\iota^*$ admits a left adjoint $L$ which agrees \emph{pointwise} with $\iota_!$. In the same way, one shows that $\iota_!$ is indeed a left adjoint (with adjoint agreeing pointwise with $\iota^*)$. But then $L\simeq\iota_!$ because left adjoint functors out of $\ul\Spc_S$ are characterized by their value on the terminal presheaf by \Cref{thm:univ-prop-T-spc}. We conclude that $\iota_!$ is left adjoint to $\iota^*$ as claimed.
			\end{proof}
		\end{lemma}

		\begin{lemma}\label{lemma:iota!-ext}
			There is a unique $T$-category $\ul\Spc_{S\triangleright T}$ equipped with a fully faithful $T$-functor $\iota_!\colon \ul\Spc_{S\triangleright T} \hookrightarrow \ul\Spc_T$ whose underlying $S$-functor is $\iota_!\colon \ul\Spc_S \hookrightarrow \iota^*\ul\Spc_T$.
		\end{lemma}

		Here the expression `$\ul\Spc_{S\triangleright T}\!$' is to be parsed as `$(\ul\Spc_S)_{\triangleright T}$,' reflecting the fact that this is a (canonical) extension of the $S$-category $\ul\Spc_S$ to a $T$-category.

		\begin{proof}
			The statement is equivalent to the claim that the essential image of the inclusion $\iota_!\colon \ul\Spc_S \hookrightarrow \iota^*\ul\Spc_T$ is in fact a $T$-subcategory of $\ul\Spc_T$, which is precisely the content of Lemma~\ref{lemma:iota!-essim}.
		\end{proof}

		As an upshot, Axiom (\ref{item:F-pb}) holds without any representability assumptions on $A$, $B$, or $B'\!$.

	\begin{lemma}\label{lemma:iota!-faithful}
		For any presheaf $X\in\PSh(S)$, the unit map $\eta_X\colon X\to\iota^*\iota_!X$ is a monomorphism. Put differently, the functor $\iota_!\colon\PSh(S)\to\PSh(T)$ is faithful.
		\begin{proof}
			This works in exactly the same way as for admissibility structures \cite{SAG}*{Proposition~20.2.4.5-(a)}: By Kan's pointwise formula, the presheaf $\iota^*\iota_!X$ is given in degree $A \in S$ by $\colim_{B\in(A/\iota)^\op} X(B)$ with $A/\iota := A/T \times_T S$, and the unit map $\eta\colon X(A)\to(\iota^*\iota_!X)(A)$ corresponds under this identification with the structure map of the term $\id_A \in A/\iota$. Since this term is contained in the full subcategory $A/S \subset A/\iota$ of maps in $S$, we may factor $\eta$ as
			\[
				X(A) \to \colim_{B \in (A/S)\catop} X(B) \to \colim_{B\in(A/\iota)^\op} X(B).
			\]
			The first map is an equivalence (the object $\id_A \in A/S$ being terminal), and thus it remains to show that the second map is a monomorphism. For this, we claim that the category $A/\iota$ is a disjoint union of the full subcategory $A/S$ and its complement (consisting of maps not in $S$), i.e.~any object $t\colon A \to B$ in $A/\iota$ mapping to or from an object in $A/S$ must itself be in $A/S$. Indeed, let $s\colon A\to B'$ be any map in $S$: if there is a map $s\to t$ in $A/\iota$, then $t$ belongs to $S$ as the latter is a subcategory; on the other hand, if there is a map $t\to s$, then $t$ belongs to $S$ by left cancellability.

			It follows that $\colim_{B\in(A/\iota)^\op} X(B)$ splits as a disjoint union $\colim_{B\in(A/S)^\op} X(B) \amalg Y$, finishing the proof.
		\end{proof}
	\end{lemma}

	\begin{lemma}\label{lemma:unit-pb}
		Let $f\colon X\to Y$ be a map in $\PSh(S)$. Then the naturality square
		\begin{equation*}
			\begin{tikzcd}
				X\arrow[r, "f"]\arrow[d, "\eta"'] & Y\arrow[d,"\eta"]\\
				\iota^*\iota_!X\arrow[r, "\iota^*\iota_!f"'] & \iota^*\iota_!Y
			\end{tikzcd}
		\end{equation*}
		of the unit transformation $\eta\colon \id \Rightarrow \iota^*\iota_!$ is a pullback square.
		\begin{proof}
			Again, this is analogous to the proof for admissibility structures \cite{SAG}*{Proposition~20.2.4.5-(b)}. The proof of the previous lemma shows that after evaluating at $A\in T$ the naturality square is equivalent to a square of the form
			\begin{equation*}
				\begin{tikzcd}
					X(A)\arrow[d,hook]\arrow[r, "f(A)"] &[1.25em] Y(A)\arrow[d,hook]\\
					X(A)\amalg X'\arrow[r, "f(A)\amalg f'\,"'] & Y(A)\amalg Y'\llap,
				\end{tikzcd}
			\end{equation*}
			which is evidently a pullback.
		\end{proof}
	\end{lemma}

	Our next goal is to prove the following sharpening of \Cref{lemma:iota!-faithful}:

	\begin{proposition}\label{prop:iota!-mono}
		The faithful functor $\iota_!\colon\PSh(S)\to\PSh(T)$ is fully faithful on groupoid cores, and thus a monomorphism in the ($\infty$-)category of ($\infty$-)categories.
	\end{proposition}

	The proof of this proposition is surprisingly subtle and will require some further preparations.

	\begin{definition}\label{defi:admissible-early}
	Let $X,Y\in\PSh(S)$. We call $f\colon \iota_!X\to \iota_!Y$ in $\PSh(T)$ \emph{admissible} if it lies in the image of the inclusion $\iota_!\colon \maps_{\PSh(S)}(X,Y) \hookrightarrow \maps_{\PSh(T)}(\iota_!X,\iota_!Y)$.
	\end{definition}

	Beware that a priori this depends on the equivalence classes of $X$ and $Y$ in $\PSh(S)$, not only on the equivalence classes of their left Kan extensions in $\PSh(T)$, and only once we have proven Proposition~\ref{prop:iota!-mono} will we know that this independent of the choices of preimages.

	\begin{lemma}\label{lemma:admissible-local}
	Let $X,Y\in\PSh(S)$ and let $f\colon \iota_!X\to\iota_!Y$ be a map in $\PSh(T)$.
		\begin{enumerate}
			\item The map $f$ is admissible if and only if its adjunct $\tilde{f}\colon X \to \iota^*\iota_!Y$ factors through the monomorphism $\eta\colon Y \to \iota^*\iota_!Y$.
			\item Let $(g_i)_{i\in I}\colon \coprod_{i\in I} X_i \twoheadrightarrow X$ be an effective epimorphism in $\PSh(S)$. Then $f$ is admissible if and only if the composite $f \circ \iota_!(g_i)\colon \iota_!X_i \to \iota_!Y$ is admissible for every $i \in I$.
			\item Let $(h_i)_{i\in I}\colon\coprod_{i\in I} Y_i\twoheadrightarrow Y$ be an effective epimorphism in $\PSh(S)$. Then $f$ is admissible if and only if for every $i \in I$ there exists a pullback diagram in $\PSh(T)$ of the form
			\[
			\begin{tikzcd}
				\iota_!X_i \rar{\iota_!h_i'} \dar[swap]{f_i} \drar[pullback] & \iota_!X \dar{f} \\
				\iota_!Y_i \rar[swap]{\iota_!h_i} & \iota_!Y
			\end{tikzcd}
			\]
			such that $f_i$ is admissible.
		\end{enumerate}
	\end{lemma}

	For the definition of effective epimorphisms, see \cite[Corollary~6.2.3.5]{HTT}. In a presheaf topos $\PSh(T)$, $f\colon X\rightarrow Y$ is an effective epimorphism if for all $A\in T$, $f(A)\colon X(A)\rightarrow Y(A)$ induces an epimorphism on $\pi_0$ \cite[Corollary~7.2.15]{HTT}. Note also that effective epimorphisms are not necessarily epimorphisms in the categorical sense of the term: for example, $S^1\hookrightarrow S^2$ is an effective epimorphism in $\Spc$, but it isn't a (categorical) epimorphism, not even in the homotopy category of spaces.

	\begin{proof}
		Part (1) is immediately clear from the definitions. Using (1), we see that part (2) is equivalent to the statement that the adjunct map $\tilde{f}\colon X \to \iota^*\iota_!Y$ factors through the unit $\eta_Y\colon Y \to \iota^*\iota_!Y$ if and only if each of the composites $\tilde{f} \circ g_i\colon X_i \to \iota^*\iota_!Y$ do. But since $\eta_Y$ is a monomorphism, the functor $(\eta_Y)_!\colon \PSh(T)_{/Y} \to \PSh(T)_{/\iota^*\iota_!Y}$ is fully faithful, hence this is immediate from the assumption that $X$ is the $\Delta^\op$-shaped colimit of the \v{C}ech nerve of $(g_i)_{i\in I}$. For part (3), the `only if'-direction follows directly from Lemma~\ref{lemma:iota!-pullback}. For the `if'-direction, observe that the map $(h_i')_{i \in I}\colon \coprod_{i \in I} X_i \to X$ from part (3) is an effective epimorphism in $\PSh(S)$: by \Cref{lemma:unit-pb} it is a pullback of the effective epimorphism $\iota^*\iota_!(h_j)_{j\in J}\colon\iota^*\iota_!\coprod_{j\in J} Y_j\to\iota^*\iota_!Y$. The claim thus follows from part (2), as for every $i \in I$ the composite $f \circ \iota_!(h'_i) = \iota_! h_i \circ f_i$ is admissible by assumption.
	\end{proof}

	\begin{lemma}\label{lemma:admissible-canc}
		Let $X,Y,Z\in\PSh(S)$ and let $f\colon \iota_!X\to \iota_!Y$ and $g\colon \iota_!Y\to \iota_!Z$ be maps in $\PSh(T)$ such that $g$ and $gf$ are admissible. Then also $f$ is admissible.
		\begin{proof}
			By the previous lemma, we have to show that the composite $\iota^*(f)\eta\colon X \to \iota^*\iota_!Y$ factors through $\eta\colon Y\to\iota^*\iota_!Y$. However, by Lemma~\ref{lemma:unit-pb} and admissibility of $g$ the latter is pulled back from the unit $\eta\colon Z\to \iota^*\iota_!Z$ along $\iota^*(g)$. It therefore suffices to show that $\iota^*(g)\iota^*(f)\eta$ factors accordingly. However, this is immediate from admissibility of $gf$.
		\end{proof}
	\end{lemma}

	We are now ready for the proof of \Cref{prop:iota!-mono}:

	\begin{proof}[Proof of Proposition~\ref{prop:iota!-mono}]
		In light of the faithfulness of $\iota_!\colon \PSh(S) \to \PSh(T)$ from \Cref{lemma:iota!-faithful}, it remains to show that $\iota_!$ is full on cores. Note that it suffices to prove that for presheaves $X,Y\in\PSh(S)$ any equivalence $f\colon\iota_!X\iso\iota_!Y$ is admissible. We will prove this in two steps:

		\textit{Step 1}: We will first treat the special case where $X = A \in S$ is a representable presheaf. Consider the image $f_A(\id_A) \in \pi_0\big((\iota_!Y)(A)\big)$ of the identity map $\id_A \in \pi_0\big(\iota_!(A)(A)\big)$. Kan's pointwise formula shows $(\iota_!Y)(A) \simeq \colim_{B\in (A/\iota)^\op} Y(B)$, whence $\pi_0((\iota_!Y)(A))\cong \colim_{B\in (A/\iota)^\op} \pi_0Y(B)$ since $\pi_0\colon\Spc\to\Set$ is a left adjoint. Thus, we can identify $f(\id_A)$ (as a class in $\pi_0$) with a class $[\alpha,y]$ for some morphism $\alpha\colon A \to B$ in $T$ and some object $y \in Y(B)$. By \Cref{lemma:admissible-local}, we have to prove that this class $[\alpha,y]$ lies in the image of the monomorphism
		\[
		\eta_A\colon Y(A) \simeq \colim_{B\in (A/S)^\op} Y(B) \hookrightarrow \colim_{B\in (A/\iota)^\op} Y(B) = (\iota^*\iota_!Y)(A)
		\]
		(or rather, its effect on $\pi_0$) induced by the disjoint summand inclusion $A/S \hookrightarrow A/\iota$ (see the proof of Lemma~\ref{lemma:iota!-faithful}). In other words, we have to show that $\alpha$ is a morphism in $S$.

		Since the map $f_B\colon \maps_{T}(B,A) = (\iota_!A)(B) \iso (\iota_!Y)(B)$ is an equivalence, there exists a map $\beta\colon B\to A$ in $T$ satisfying $f_B(\beta) = [\id_B,y]$ in $\pi_0(\iota_!Y(B))$. We thus have
		\[
			f_A(\beta\alpha)=\alpha^*f_B(\beta)=[\alpha,y]=f_A(\id),
		\]
		and since also $f_A$ is an equivalence we deduce that $\beta\alpha=\id$. On the other hand, we have $[\id_B,y] = f_B(\beta)=\beta^*f_A(\id_A)=\beta^*[\alpha,y] = [\alpha\beta,y]$, and since $A/S\hookrightarrow A/\iota$ is a disjoint summand inclusion we see that $\alpha\beta$ belongs to $S$. It follows from Axiom (\ref{item:F-idem}) that also $\alpha$ belongs to $S$, finishing Step 1.

		\textit{Step 2:} We will now deduce the statement for an arbitrary presheaf $X \in \PSh(S)$. Pick an effective epimorphism $(h_j)_{j\in J}\colon\coprod_{j\in J} Y_j \twoheadrightarrow Y$ in $\PSh(S)$ for representable $Y_j$, and choose for each $j\in J$ a pullback
		\begin{equation*}
			\begin{tikzcd}
				\iota_!P_j\arrow[r, "\iota_!h_j'"]\arrow[dr,phantom,"\lrcorner"{very near start}]\arrow[d, "f_j"'] & \iota_!X\arrow[d, "f"]\\
				\iota_!Y_j\arrow[r, "\iota_!h_j"'] & \iota_!Y
			\end{tikzcd}
		\end{equation*}
		in $\PSh(T)$ using Axiom (\ref{item:F-pb}). As $f$ is an equivalence, so is each $f_j$. As $Y_j$ is representable, it follows from Step 1 that $f^{-1}_j\colon \iota_!Y_j \iso \iota_!P_j$ is admissible, and thus by \Cref{lemma:admissible-canc} also $f_j$ is admissible. It thus follows from Lemma~\ref{lemma:admissible-local} that also $f$ is admissible, completing the proof of the proposition.
	\end{proof}

	\begin{remark}\label{rk:F-idem-necessary}
		Axiom (\ref{item:F-idem}) is necessary for the previous proposition: every wide subcategory $\iota\colon S\hookrightarrow T$ for which the left Kan extension functor $\iota_!\colon \PSh(S) \to \PSh(T)$ is a monomorphism of categories automatically satisfies (\ref{item:F-idem}). To see this, consider morphisms $\alpha,\beta$ as in Axiom (\ref{item:F-idem}), and define $X \in \PSh(S)$ to be the colimit of the diagram
		\begin{equation*}
			A\xrightarrow{\alpha\beta}A\xrightarrow{\alpha\beta}\cdots.
		\end{equation*}
		Since $\iota_!$ preserves colimits, it follows that $\iota_!X$ is the colimit of the analogous diagram in $\PSh(T)$. But since $\alpha$ and $\beta$ are maps in $T$, the maps $\alpha\colon A \to B$ exhibit $B$ as another colimit of this diagram, and thus we get an equivalence $\iota_!X\simeq\iota_! B$ in $\PSh(T)$ compatible with the colimit structure maps. Assuming that $\iota_!$ is a monomorphism, it follows that $X \simeq B$ is a representable presheaf on $S$, and thus the map $\alpha\colon A \to B$ in $T$ agrees up to equivalence in $T$ with the structure maps $A\to X$, which belong to $S$ by construction. As $S$ contains all equivalences, this shows that also $\alpha$ belongs to $S$, finishing the argument.
	\end{remark}

	Note moreover that (\ref{item:F-idem}) is not implied by the remaining two axioms as the following example shows:

	\begin{example}
		Let $R\not=0$ be a commutative ring. We let $T=\text{Perf}(R)$ be the stable category of perfect $R$-chain complexes, and we let $S$ consist of those $f\colon X\to Y$ such that $[X]=[Y]\in K_0(R)$, or equivalently (by the defining relations of $K_0$) such that the fiber of $f$ vanishes in $K_0$.

		The first description makes it clear that $S$ is a subcategory, contains all equivalences, and even satisfies $2$-out-of-$3$, proving (\ref{item:F-canc}). On the other hand, the second description shows that $S$ is closed under pullbacks, proving (\ref{item:F-pb-strong}). However, (\ref{item:F-idem}) does not hold: $0\to R\to 0$ is the identity and $R\to0\to R$ belongs to $S$ as $[R]=[R]$, but neither $0\to R$ nor $R\to0$ are contained in $S$ as $[R]\not=0$ in $K_0(R)$.
	\end{example}

	\begin{definition}\label{defi:admissible}
		Following Lurie's notation and terminology for fractured $\infty$-topoi, we let $\PSh(T)^\textup{corp} \subset \PSh(T)$ denote the (non-full) essential image of the left Kan extension functor $\iota_!\colon\PSh(S)\to\PSh(T)$. A presheaf on $T$ is called \textit{corporeal} if it is an object of $\PSh(T)^\textup{corp}$, and a morphism between two corporeal presheaves on $T$ is called \textit{admissible} if it is a morphism in $\PSh(T)^\textup{corp}$. Beware that despite the terminology and notation both of these notions depend on $S$.
	\end{definition}

	Note that for two $X,Y\in\PSh(S)$ a map $f\colon\iota_!X \to\iota_!Y$ is admissible in the sense of \Cref{defi:admissible} if and only if it is admissible in the sense of \Cref{defi:admissible-early} above.

	\begin{lemma}\label{lemma:admissible-basic}
		Let $X,Y,Z\in\PSh(T)$ be corporeal presheaves.
		\begin{enumerate}
			\item Let $f\colon X\to Z$ be an admissible morphism, and let $g\colon Y\to Z$ be arbitrary. Then the base change $g^*(f)\colon g^*(X)\to Y$ of $f$ along $g$ is again an admissible morphism of corporeal presheaves.\label{item:ab-pb}
			\item Let $f\colon X\twoheadrightarrow Y$ be an effective epimorphism, and let $g\colon Y\to Z$ be arbitrary. Assume that $f$ and $gf$ are admissible. Then also $g$ is admissible.\label{item:ab-cover}
			\item Let $f\colon X\to Y$, $g\colon Y\to Z$ be maps such that $g$ and $gf$ are admissible. Then $f$ is admissible.\label{item:ab-canc}
		\end{enumerate}
		\begin{proof}
			The first statement is a consequence of Lemma~\ref{lemma:iota!-ext}, while the second statement follows from Lemma~\ref{lemma:admissible-local}. Finally, the third statement follows from Lemma~\ref{lemma:admissible-canc}.
		\end{proof}
	\end{lemma}

	We now come to the final missing ingredient of the proof of Theorem~\ref{thm:fracture}:

	\begin{proposition}\label{prop:counit-pb}
		Let $f\colon X\to Y$ be a map in $\PSh(T)^\textup{corp}$. Then the naturality square
		\begin{equation*}
			\begin{tikzcd}
				\iota_!\iota^*X\arrow[d, "\epsilon"']\arrow[r, "\iota_!\iota^*f"] & \iota_!\iota^*Y\arrow[d, "\epsilon"]\\
				X\arrow[r, "f"'] & Y
			\end{tikzcd}
		\end{equation*}
		is a pullback in $\PSh(T)$.
	\end{proposition}

	For the proof we will use:

	\begin{lemma}
		Let $Y\in\PSh(T)$ be an arbitrary presheaf. Then the composite
		\begin{equation}\label{eq:epsilon-push-forward}
			\PSh(S)_{/\iota^*Y}\xrightarrow{\iota_!}\PSh(T)_{/\iota_!\iota^*Y}\xrightarrow{\PSh(T)_{/\epsilon}}\PSh(T)_{/Y}
		\end{equation}
		induces an equivalence onto the non-full subcategory $(\PSh(T)_{/Y})^\textup{corp}$ whose objects are those $X\to Y$ where $X$ is corporeal (but there is no condition on the map to $Y$) and whose morphisms are the admissible maps in $\PSh(T)$.
		\begin{proof}
			It is clear that $(\ref{eq:epsilon-push-forward})$ factors through $(\PSh(T)_{Y})^\text{corp}$, so it only remains to show that the induced functor is essentially surjective and fully faithful. For this we observe that since $\iota_!$ and $\iota^*$ are adjoint, the map $\maps(X,\iota^*Y)\to\maps(\iota_!X,Y), g\mapsto \epsilon\circ\iota_!(g)$ is an equivalence for any $X\in\PSh(S)$. This immediately implies essential surjectivity, while for full faithfulness we observe that for objects $X,X' \in \PSh(S)_{/\iota^*Y}$ the induced map on mapping spaces fits in the following diagram of fiber sequences:
            \[
            \begin{tikzcd}
                \Hom_{\PSh(S)_{/\iota^*Y}}(X,X') \dar[hookrightarrow] \rar & \Hom_{\PSh(S)}(X,X') \rar \dar[hookrightarrow]{\iota_!} & \Hom_{\PSh(S)}(X,\iota^*Y) \dar{\simeq} \\
                \Hom_{\PSh(T)_{/Y}}(\iota_!X,\iota_!X') \rar & \Hom_{\PSh(T)}(\iota_!X,\iota_!X') \rar & \Hom_{\PSh(T)}(\iota_!X,Y).
            \end{tikzcd}
            \] We now simply note that the middle map is a monomorphism by \Cref{lemma:iota!-faithful}, with image the admissible maps.
		\end{proof}
	\end{lemma}

	\begin{proposition}
		The $T$-functor $\iota^*\colon\PSh(T)_{/\bullet}\to\PSh(S)_{/\iota^*(\bullet)}$ admits an $S$-left adjoint (that is, the underlying $S$-functor admits a parametrized left adjoint) given pointwise by the composites $(\ref{eq:epsilon-push-forward})$.
		\begin{proof}
			It is clear that the composites $(\ref{eq:epsilon-push-forward})$ yield a pointwise left adjoint, so it only remains to check the Beck--Chevalley condition. By the previous lemma, this amounts to saying that the adjunction $\PSh(T)_{/f}\colon\PSh(T)_{/\iota_!X}\rightleftarrows\PSh(T)_{/\iota_!Y}\noloc f^*$ restricts to an adjunction $(\PSh(T)_{/\iota_!X})^\text{corp}\rightleftarrows(\PSh(T)_{/\iota_!Y})^\text{corp}$ for any admissible $f\colon\iota_!X\to\iota_!Y$, i.e.
			\begin{enumerate}
				\item The right adjoint $f^*$ restricts to $(\PSh(T)_{/\iota_!Y})^\text{corp}\to(\PSh(T)_{/\iota_!X})^\text{corp}$.
				\item For each $Z\in(\PSh(T)_{/\iota_!Y})^\text{corp}$ the counit $\PSh(T)_{/f}f^*Z\to Z$ is admissible.
				\item For each $W\in(\PSh(T)_{/\iota_!X})^\text{corp}$ the unit $W\to f^*\PSh(T)_{/f}W$ is admissible.
			\end{enumerate}
			For this, let $g\colon Z\to Z'$ be a map in $(\PSh(T)_{/\iota_!Y})^\text{corp}$ and consider the coherent cube
			\begin{equation*}
				\begin{tikzcd}[column sep=tiny, row sep=small]
					& f^*Z'\arrow[dd]\arrow[ddrr,phantom,"\lrcorner"{very near start,xshift=-5pt,yshift=4pt}]\arrow[rr, "\epsilon"{near end}] && Z'\arrow[dd]\\
					f^*Z\arrow[ur, "f^*g"]\arrow[ddrr,phantom,"\lrcorner"{very near start,xshift=-5pt,yshift=4pt}]\arrow[rr, crossing over, "\epsilon"{near end}]\arrow[dd] && Z\arrow[ru,"g"]\\
					& \iota_!X\arrow[rr, "f"{description, near end}] && \iota_!Y\\
					\iota_!X\arrow[rr, "f"'{near end}]\arrow[ur,equal] && \iota_!Y\arrow[ur,equal]\arrow[from=uu,crossing over]
				\end{tikzcd}
			\end{equation*}
			Lemma~\ref{lemma:admissible-basic}-(\ref{item:ab-pb}) then shows that the objects $f^*Z$ and $f^*Z'$ are corporeal and that the maps $\epsilon\colon f^*Z\to Z$ and $\epsilon \colon f^*Z'\to Z'$ are admissible, proving the second claim and one half of the first claim. Together with Lemma~\ref{lemma:admissible-basic}-(\ref{item:ab-canc}) we then conclude that $f^*g$ is again admissible, proving the remaining half of of the first claim.

			Finally, if $W\in(\PSh(T)_{/\iota_!X})^\text{corp}$, then as a morphism in $\PSh(T)$ the unit $\eta\colon W\to f^*\PSh(T)_{/f}W$ is right inverse to the counit $\epsilon$. Thus, $\eta$ is admissible by another application of Lemma~\ref{lemma:admissible-basic}-(\ref{item:ab-canc}).
		\end{proof}
	\end{proposition}

	\begin{proof}[Proof of Proposition~\ref{prop:counit-pb}]
		We may assume without loss of generality that $f$ is of the form $\iota_!f'$ for some $f'\colon X'\to Y'$ in $\PSh(S)$. In this case, the previous proposition shows that the Beck--Chevalley transformation
		\begin{equation*}
			\begin{tikzcd}
				\PSh(S)_{/\iota^*Y}\arrow[d, "\PSh(T)_{/\epsilon}\circ\iota_!"']\arrow[r, "(\iota^*f)^*"] &[1.5em] \PSh(S)_{/\iota^*X}\arrow[d, "\PSh(T)_{/\epsilon}\circ\iota_!"]\arrow[dl,Rightarrow]\\
				\PSh(T)_{/Y}\arrow[r, "f^*"'] & \PSh(T)_{/X}
			\end{tikzcd}
		\end{equation*}
		is an equivalence. Chasing through the identity of $\iota^*Y$ precisely yields the claim.
	\end{proof}

	We are now ready to prove Theorem~\ref{thm:fracture}.
	\begin{proof}[Proof of Theorem~\ref{thm:fracture}]
		If $\iota\colon S\hookrightarrow T$ is a cleft category, then $\iota_!\colon \PSh(S) \to \PSh(T)$ is a fracture subcategory:
		\begin{enumerate}[label=(F\arabic*),start=0]
			\item The functor $\iota_!$ is a monomorphism by Proposition~\ref{prop:iota!-mono}.
			\item The functor $\iota_!$ preserves pullbacks by Lemma~\ref{lemma:iota!-pullback}.
			\item The right adjoint $\iota^*$ of $\iota_!$ is clearly cocontinuous, and it is conservative as $S$ contains all objects of $T$.
			\item The pullback condition for the counit was verified in Proposition~\ref{prop:counit-pb}.
		\end{enumerate}

		Conversely, assume that $\iota_!\colon\PSh(S)\to\PSh(T)$ is a fracture subcategory. Then $(\iota_!)_{/X}\colon\PSh(S)_{/X}\to\PSh(T)_{/\iota_!X}$ is fully faithful for any $X$ by \cite{SAG}*{Proposition~20.1.3.1}; specializing to $X=A\in S$, we see that left Kan extension along $\iota_{/A}\colon S_{/A}\hookrightarrow T_{/A}$ is fully faithful, whence so is $\iota_{/A}$ itself by the Yoneda Lemma. Letting $A$ vary, this precisely amounts to saying that $S$ is left cancellable, proving (\ref{item:F-canc}).

		For (\ref{item:F-pb}), consider a map $f\colon X\to Y$ in $\PSh(S)$ and a map $g\colon \iota_!Y'\to\iota_!Y$ in $\PSh(T)$. Write $\tilde g\colon Y'\to \iota^*\iota_!Y$ for the adjunct of $g$, and define $X'$ via the following pullback square in $\PSh(S)$:
		\[
		\begin{tikzcd}
			X' \dar[swap]{f'} \rar \arrow[dr,phantom,"\lrcorner"{very near start}] & \iota^*\iota_!X \dar{\iota^*\iota_!f} \\
			Y' \rar[swap]{\tilde{g}} & \iota^*\iota_!Y.
		\end{tikzcd}
		\]
		In the diagram
		\begin{equation*}
			\begin{tikzcd}
				\iota_!(X')\arrow[d, "\iota_!(f')"']\arrow[r] & \iota_!\iota^*\iota_!(X)\arrow[d,"\iota_!\iota^*\iota_!f"{description}] \arrow[r, "\epsilon"] & \iota_!(X)\arrow[d, "\iota_!(f)"]\\
				\iota_!(Y')\arrow[r, "\iota_!(\tilde g)"'] & \iota_!\iota^*\iota_!(Y)\arrow[r, "\epsilon"'] & \iota_!(Y),
			\end{tikzcd}
		\end{equation*}
		the left-hand square is a pullback as $\iota_!$ preserves pullbacks by \ref{it:F1_Preserves_Pullbacks}, while the right-hand square is a pullback square by \ref{it:F3_Counit_Pullback_Square}. Thus, the total square expresses $\iota_!(f')$ as a pullback of $\iota_!(f)$ along $g$, showing (\ref{item:F-pb}).

		Finally, Axiom (\ref{item:F-idem}) holds because $\iota_!$ is a monomorphism, see Remark~\ref{rk:F-idem-necessary}.
	\end{proof}

	\subsection{Examples}\label{subsec:Examples-clefts}  We close this section by establishing our two key examples of cleft categories. We begin with Example~\ref{ex:factorization-cleft}, for which we recall:

	\begin{definition}
		A \emph{factorization system} on an category $T$ consists of two wide subcategories $E,M\subset T$ satisfying the following conditions:
		\begin{enumerate}
			\item Both $E$ and $M$ contain all equivalences.
			\item Every morphism in $E$ is \emph{left orthogonal} to every morphism in $M$ in the following sense: for every pair of morphisms $e\colon A \to B$ in $E$ and $m\colon X \to Y$ in $M$ and every solid square
			\[\begin{tikzcd}
				A & X \\
				B & Y,
				\arrow["e"', from=1-1, to=2-1]
				\arrow[from=1-1, to=1-2]
				\arrow["m", from=1-2, to=2-2]
				\arrow[from=2-1, to=2-2]
				\arrow[dashed, from=2-1, to=1-2]
			\end{tikzcd}\] there is a contractible space of dotted lifts making both triangles commute, i.e.~the square
			\[\begin{tikzcd}
				{\maps_T(B,X)} &[.5em] {\maps_T(B,Y)} \\
				{\maps_T(A,X)} & {\maps_T(A,Y)}
				\arrow["{- \circ e}"', from=1-1, to=2-1]
				\arrow["{m \circ -}", from=1-1, to=1-2]
				\arrow["{- \circ e}", from=1-2, to=2-2]
				\arrow["{m \circ -}"', from=2-1, to=2-2]
			\end{tikzcd}\] is a pullback square in the category of spaces;
			\item Every morphism $f\in T$ admits a factorization $f= me$, with $e$ in $E$ and $m$ in $M$.
		\end{enumerate}
	\end{definition}

	\begin{remark}
		The above definition follows \cite{abfj-left-exact}*{Definition~3.1.6}. By {Lemma 3.1.9} of \emph{op.~cit.}, the class $E$ in a factorization system is \emph{exactly} the class of morphisms in $T$ which are left orthogonal to all morphisms in $M$, and vice-versa. In particular, this implies that both $E$ and $M$ are closed under retracts, so that the above is equivalent to~\cite{HTT}*{Definition 5.2.8.8} (where this condition is assumed a priori).
	\end{remark}

	\begin{proposition}\label{prop:factorization-cleft}
		Let $(E,M)$ be a factorization system on $T$. Then the right class $M$ is a cleft of $T$.
	\end{proposition}

	If $T$ has pullbacks, this proposition appears (with a rather different proof) as \cite{SAG}*{Proposition~20.2.2.1}.

	\begin{proof}
		By assumption $M\subset T$ is a wide subcategory containing all equivalences, and it is left cancellable by \cite{HTT}*{Proposition 5.2.8.6-(3)}, proving (\ref{item:F-canc}). Moreover, Axiom (\ref{item:F-retract}) was noted in the previous remark.

		It remains to verify (\ref{item:F-pb}), i.e.~that for every $f\colon A\to B$ in $T$ the pullback functor $f^*\colon\PSh(T)_{/B}\to\PSh(T)_{/A}$ maps the image of $M_{/B}$ into $\PSh(M)_{/A}$.	We will prove this more generally for $\PSh(M)_{/B}$. For this we observe that the diagram
		\begin{equation}\label{diag:slice-in-out}
			\begin{tikzcd}
				\PSh(M_{/A})\arrow[r, "\sim"]\arrow[d, "(\iota_{/A})_!"'] & \PSh(M)_{/A}\arrow[d, "(\iota_!)_{/A}"]\\
				\PSh(T_{/A})\arrow[r, "\sim"'] & \PSh(T)_{/A}
			\end{tikzcd}
		\end{equation}
		with the horizontal equivalences as in Example~\ref{ex:spc-T} commutes up to equivalence since both paths are cocontinuous and agree on the Yoneda image. Arguing in the same way for $B$, it then suffices to show: \emph{if $X\in\PSh(T_{/B})$ is left Kan extended from $\PSh(M_{/B})$, then its restriction to $\PSh(T_{/A})$ is left Kan extended from $\PSh(M_{/A})$}.

		To this end, let $g\colon X\to A$ be any map in $T$, and fix a factorization
		\begin{equation*}
			\begin{tikzcd}[column sep=small]
				X\arrow[rr, "e"]\arrow[dr, bend right=15pt, "g"'] && Y\arrow[dl, bend left=15pt, "m"]\\
				& A
			\end{tikzcd}
		\end{equation*}
		with $e$ in $E$ and $m$ in $M$. Viewing this as a map in $T_{/A}$, \cite{HTT}*{Remark~5.2.8.3} shows that for every other other $t\in T_{/A}$ the map $e^*\colon\maps(m,t)\to\maps(g,t)$ is an equivalence. It follows that $\iota_{/A}\colon M_{/A} \hookrightarrow T_{/A}$ admits a left adjoint $\lambda_A\colon T_{/A} \to M_{/A}$ sending $g$ to $m$ with unit $g\to m$ given by the above triangle. In particular, all units live in the subcategory $T_{/A}^E\coloneqq E\times_T T_{/A}$; conversely, an easy $2$-out-of-$3$ argument shows that $\lambda_A$ inverts all maps in $T_{/A}^E$. By abstract nonsense about Bousfield localizations, it follows that $\lambda_A$ is a localization at $T_{/A}^E$, so that $X\in\PSh(T_{/A})$ is left Kan extended if and only if it inverts $T_{/A}^E$. Arguing in the same way for $B$, the proposition follows as the postcomposition functor $T_{/f}\colon T_{/A}\to T_{/B}$, $g\mapsto f\circ g$ obviously restricts to $T_{/A}^E\to T_{/B}^E$.
	\end{proof}

	Next, we recall atomic orbital subcategories from \cite{CLL_Global}*{Definition~4.3.1}:

	\begin{definition}\label{def:atomic-orbital}
		A wide subcategory $P\subset T$ containing all equivalences is called \emph{atomic orbital} if the following conditions are satisfied:
		\begin{enumerate}
			\item For every $p\colon C\to D$ in $P$ and $t\colon B\to D$ in $T$ there exists a pullback
			\begin{equation*}
				\begin{tikzcd}
					\coprod_{i=1}^n A_i\arrow[r, "(p_i)_{i=1,\dots,n}"]\arrow[dr,phantom,"\lrcorner"{very near start,xshift=-7pt,yshift=3pt}] \arrow[d] &[2.5em] B\arrow[d, "t"]\\
					C\arrow[r, "p"'] & D
				\end{tikzcd}
			\end{equation*}
			in $\PSh(T)$ such that each $p_i\colon A_i\to B$ belongs to $P$.
			\item\label{item:ao-disj-sum} For every $p\colon A\to B$ in $P$ the diagonal $A\to A\times_BA$ is a disjoint summand inclusion in $\PSh(T)$, i.e.\ it is equivalent to an inclusion of the form $A \hookrightarrow A \amalg C$ for some $C \in \PSh(T)$.
		\end{enumerate}
	\end{definition}

	\begin{remark}
		By \cite{CLL_Global}*{Lemma~4.3.2} we can equivalently replace (\ref{item:ao-disj-sum}) by the following axiom:
		\begin{enumerate}
			\item[(\ref*{item:ao-disj-sum}$'$)] Every map in $P$ that admits a section in $T$ is an equivalence.
		\end{enumerate}
	\end{remark}

	\begin{example}
		For a finite group $G$, the orbit category $\Orb_G$ from \Cref{ex:G-cat} is atomic orbital as a subcategory of itself. Similarly the subcategory $\Orb \subset \Glo$ from \Cref{ex:Orb} is atomic orbital \cite{CLL_Global}*{Example~4.3.3}.
	\end{example}

	Atomic orbital subcategories were introduced in \cite{CLL_Global} to encode different degrees of `parametrized semiadditivity,' and we will revisit them from this perspective in Section~\ref{sec:semiadd}. For now we are interested in them as examples of clefts:

	\begin{proposition}\label{prop:fission-cleft}
		Any atomic orbital subcategory $P\subset T$ is a cleft.
	\end{proposition}

	For the proof we will use:

	\begin{lemma}
		Let $P$ be atomic orbital (say, as a subcategory of itself) and consider an object $A\in P$. Then any endomorphism in $P_{/A}$ is invertible.
		\begin{proof}
			Let $B\in P_{/A}$ and fix a decomposition $B\times_AB=\coprod_{i=1}^nX_i$ into representables. We introduce the following terminology:
			\begin{enumerate}
				\item[(A)] Given any map $g\colon C\to B\times_AB$ from a representable, it factors through a unique $X_i$, and we call $i\eqqcolon\text{idx}(g)$ the \emph{index} of $g$.
				\item[(B)] An index $i\in\{1,\dots,n\}$ is called \emph{good} if the projection $\textup{pr}_2\colon B\times_AB\to B$ to the second factor restricts to an equivalence $X_i\to B$.
			\end{enumerate}
			Now let $f$ be an endomorphism of $B$, inducing a map $(1,f)\colon B\to B\times_AB$. We claim that $\text{idx}(1,f)$ is good, which will then imply the lemma as $(1,f)$ induces an equivalence onto $X_{\text{idx}(1,f)}$, being a section to the map $\pr_1\colon X_{\text{idx}(1,f)}\to B$ in $P$.

			To prove the claim, we make the following basic observations:
			\begin{enumerate}
				\item\label{item:ex-good} Given any endomorphism $g$ of $B$, the index $\text{idx}(g,1)$ is good (arguing as above using that $\textup{pr}_2(g,1)=1)$.
				\item\label{item:right-stability} Given any map $\alpha\colon X\to Y$ of representables and a map $\beta\colon Y\to B\times_AB$, we have $\text{idx}(\beta\alpha)=\text{idx}(\beta)$.
				\item\label{item:left-stability} If $\alpha,\beta\colon X\rightrightarrows B\times_AB$ are maps from a representable with $\text{idx}(\alpha)=\text{idx}(\beta)$ and $\gamma$ is any endomorphism of $B\times_AB$, then $\text{idx}(\gamma\alpha)=\text{idx}(\gamma\beta)$.
			\end{enumerate}
			By (\ref{item:right-stability}), we have
			\begin{equation*}
				\text{idx}(1,f)=\text{idx}(f^k,f^{k+1})
			\end{equation*}
			for any $k\ge0$. Now by the pigeonhole principle we find $\ell>k\ge0$ with $\text{idx}(f^k,1)=\text{idx}(f^\ell,1)$ and hence also
			\begin{equation*}
				\text{idx}(f^k,f^{k+1})=\text{idx}(f^\ell,f^{k+1})
			\end{equation*}
			by $(\ref{item:left-stability})$ applied to $1\times_Af^{k+1}$. However, by construction $\ell\ge k+1$, whence
			\begin{equation*}
				\text{idx}(f^\ell,f^{k+1})=\text{idx}(f^{\ell-k-1},1)
			\end{equation*}
			by another application of $(\ref{item:right-stability})$. Altogether we therefore get
			\begin{equation*}
				\text{idx}(1,f)=\text{idx}(f^k,f^{k+1})=\text{idx}(f^\ell,f^{k+1})=\text{idx}(f^{\ell-k-1},1)
			\end{equation*}
			and the right hand side is good by $(\ref{item:ex-good})$, finishing the proof.
		\end{proof}
	\end{lemma}

	\begin{proof}[Proof of Proposition~\ref{prop:fission-cleft}]
		Axiom (\ref{item:F-canc}) follows from \cite{CLL_Global}*{Lemma~4.3.5}, while (\ref{item:F-pb}) is immediate from \Cref{def:atomic-orbital}-(1). To prove (\ref{item:F-idem-equ}), we note that any idempotent $e\colon A\to A$ defines an endomorphism of itself considered as an object of $P_{/A}$. By the previous lemma, we conclude that $e$ is an invertible idempotent, hence homotopic to the identity.
	\end{proof}

	\section{Partial presentability}
	\label{sec:Partial-presentability}
	Given a small category $T$, there is a natural notion of \textit{$T$-presentability} for a $T$-category, recalled in \Cref{defi:T-pres}. This is quite a strong condition on $\Cc$: it in particular requires that the restriction functors $f^*\colon \Cc(B) \to \Cc(A)$ admit left adjoints for \textit{all} morphisms $f\colon A \to B$, which is unfortunately not satisfied in several naturally occurring examples, see for example Warning~\ref{warn:sp-no-fiberwise-limits} about the global category of equivariant spectra.

	The goal of this section is to introduce and study relaxations of the notion of presentability for a $T$-category $\Cc$. While we still demand that $\Cc$ be fiberwise presentable, we will weaken the cocompleteness assumption: more precisely, for any cleft $S\subset T$, we will introduce notions of \emph{$S$-cocompleteness} and \emph{$S$-presentability}, see Subsection~\ref{subsec:SPresentability}. In Subsection~\ref{subsec:ColimitsInKanExtensions} we discuss the relation between $S$-presentable $T$-categories and $S$-presentable $S$-categories. We end this section in Subsection \ref{subsec:S-cocompletion} with a discussion of the \emph{$S$-cocompletion} of a small $T$-category and the relation to the $S$-cocompletion of its underlying $S$-category.

	\subsection{$\bm S$-(co)limits and $\bm S$-presentability}
	\label{subsec:SPresentability}
	We fix a cleft category $S\subset T$ and we write $\iota\colon S\hookrightarrow T$ for the inclusion. In this subsection we study what it means for a $T$-category $\mathcal C$ to be \emph{$S$-(co)complete} or \emph{$S$-presentable}.

	\begin{definition}
		\label{def:bbU_S}
		We define the $T$-subcategory $\bbU_S \subset \ul\Spc_T$ as the essential image of the fully faithful $T$-functor $\iota_!\colon \ul\Spc_{S\triangleright T}\hookrightarrow \ul\Spc_T$ from \Cref{lemma:iota!-ext}: for an object $A \in T$, the subcategory $\bbU_S(A) \subset \ul\Spc_T(A) = \PSh(T)_{/A}$ is the full subcategory spanned by the admissible maps.
	\end{definition}

	\begin{definition}[$S$-(co)completeness]
		A $T$-category $\Cc$ is called \emph{$S$-cocomplete} if it is fiberwise cocomplete and admits all $\bbU_S$-colimits in the sense of Definition~\ref{defi:U-cc}. Dually, $\Cc$ is called \emph{$S$-complete} if it is fiberwise complete and admits all $\bbU_S$-limits.
	\end{definition}

	\begin{definition}[$S$-presentability]
		\label{def:S-Presentability}
		A $T$-category $\Cc$ is called \textit{$S$-presentable} if it is $S$-cocomplete and fiberwise presentable (Definition~\ref{defi:fib-pres}).
	\end{definition}

	\begin{warning}
		As recalled in Remark~\ref{rk:T-pres-implies-limits} any $T$-presentable $T$-category is also $T$-complete. In contrast, there are interesting examples of $T$-categories that are $S$-presentable in the above sense, but \textit{not} $S$-complete, see Warnings~\ref{warn:sp-no-finite-limits} and~\ref{warn:sp-no-fiberwise-limits}.
	\end{warning}

	We will now provide a description of $S$-(co)completeness in terms of pointwise conditions. For this we first introduce:

	\begin{definition}
		\label{defi:admissible-general}
		A morphism $f\colon X \to Y$ in $\PSh(T)$ is called \textit{admissible} if it defines an object in $\bbU_S(Y)\subset\ul\Spc_T$, i.e.~for every $A\in T$ and $t\colon A\to Y$ in $\PSh(T)$ the pulled back map $t^*
		(f)\colon t^*(X)\to A$ is an admissible map of corporeal objects in the sense of Definition~\ref{defi:admissible}.
	\end{definition}

	\begin{remark}
		Note that for a corporeal object $Y$ this recovers the previous definition by Lemma~\ref{lemma:admissible-basic}.
	\end{remark}

	By the pasting law, the admissible maps are closed under composition, and it is clear that every equivalence is admissible; in particular, the admissible maps define a wide subcategory $\PSh(T)^\text{ad}\subset\PSh(T)$. By another application of the pasting law, this is closed under pulling back along arbitrary maps in $\PSh(T)$.

	\begin{lemma}\label{lemma:S-cc-adj-general}
		Let $\mathcal C$ be a fiberwise cocomplete $T$-category. Then the following are equivalent:
		\begin{enumerate}
			\item For every $m\colon A\to B$ in $S$, the functor $m^*\colon\mathcal C(B)\to\mathcal C(A)$ admits a left adjoint $m_!$.
			\item For every $B'\in T$ and any admissible $n\colon A'\to B'$ the functor $n^*\colon \Cc(B') \to \Cc(A')$ admits a left adjoint $n_!$.
		\end{enumerate}
		\begin{proof}
			It is immediate that (2) implies (1). Conversely, let $B'\in T$ and consider an object $n\colon A'\rightarrow B'$ in $\bbU_S(B')$. Decomposing a preimage in $\PSh(S)_{/B'}$ into representables, we get an equivalence $(k_i)_{i\in I}\colon\colim_{i\in I}A'_i\simeq A'$ for a functor $A'_\bullet\colon I\to T$ such that for every $i\in I$ the composite $n_i=nk_i\colon A_i'\to A'\to B'$ lies in $S$. Then $n^*:\Cc(B')\to\Cc(A')$ agrees up to equivalence with the functor $\mathcal C(B')\to\lim_{i\in I}\mathcal C(A'_i)$ induced by the $n_i^*$. Now each of these $n_i^*$ admits a left adjoint by assumption and moreover $\mathcal C(B')$ is cocomplete; thus, also $n^*$ admits a left adjoint by \cite{descent-lim}*{Theorem B$^\op$}.
		\end{proof}
	\end{lemma}

	\begin{remark}
	For later use, we make the construction of the left adjoint given in \emph{loc.~cit.~}semi-explicit, keeping the notation from the previous proof:
		\begin{enumerate}
			\item For $X\in\Cc(A')$, $n_!X$ is the colimit of a suitable $I^\op$-diagram with $i\mapsto n_{i!}k_i^*(X)$.
			\item The counit $n_!n^*X=\colim_{i\in I^\op} n_!k_i^*n^*X=\colim_{i\in I^\op} n_{i!}n_i^*X\to X$ is induced by a cocone given at $i\in I^\op$ by the counit of $n_{i!}\dashv n_i^*$.
			\item The unit $Y\to n^*n_!Y$ is given after restricting along $k_i$ by the composite $k_i^*Y\to n_i^*n_{i!}k_i^*Y\to n_i^*\colim_{j\in J}n_{j!}k_j^*Y$ of the unit and the structure map of the colimit.
		\end{enumerate}
	\end{remark}

	Using this we can now prove:

	\begin{lemma}
		\label{lem:CharacterizationMCocompleteness}
		Let $\mathcal C$ be a $T$-category. Then $\mathcal C$ is $S$-cocomplete if and only if the following conditions are satisfied:
		\begin{enumerate}
			\item The $T$-category $\Cc$ is fiberwise cocomplete,
			\item For every morphism $m\colon A \to B$ in $S$, the restriction $m^*\colon \Cc(B) \to \Cc(A)$ admits a left adjoint $m_!$,
			\item For every pullback square
			\begin{equation*}
				\begin{tikzcd}
					A\arrow[r, "m"]\arrow[d, "t"'] \arrow[dr,phantom,"\lrcorner"{very near start}] & B\arrow[d, "u"]\\
					A'\arrow[r, "n"'] & B'
				\end{tikzcd}
			\end{equation*}
			in $\PSh(T)$ where $n$ belongs to $S$ and $u$ is a map in $T$, the Beck--Chevalley map $m_!t^*\to u^*n_!$ is an equivalence (note that $m_!$ exists by Lemma~\ref{lemma:S-cc-adj-general}).\label{item:cmcc-BC}
		\end{enumerate}
		The dual characterization for $S$-completeness also holds.
		\begin{proof}
			By definition, $S$-cocompleteness implies all of the above conditions. Conversely, if these three conditions are satisfied, it only remains by the previous lemma together with Remark~\ref{rk:U-cc-restriction-to-repr} to show that the Beck--Chevalley condition (\ref{item:cmcc-BC}) actually holds without representability assumption on $A'$.

			For this we fix a decomposition $(k_i)_{i\in I}\colon\colim_{i\in I} A'_i\simeq A'$ in $\PSh(T)$ into representables as before. We now pull back each individual $n_i=nk_i$ along $u$ to an $m_i$, and then appeal to universality of colimits to obtain a pullback
			\begin{equation*}
				\begin{tikzcd}
					\colim_{i\in I}(A'_i\times_{B'}B)\arrow[dr,phantom,"\lrcorner"{very near start},xshift=-5pt]\arrow[d, "t=\colim t_i"']\arrow[r, "m=(m_i)"]&[1.75em] B\arrow[d, "u"]\\
					\colim_{i\in I} A'_i\arrow[r, "(n_i)"'] & B'\llap.
				\end{tikzcd}
			\end{equation*}

			It then follows from cocontinuity of $u^*$ and the above description of unit and counit, that the Beck--Chevalley map $m_!t^*X\to u^*n_!X$ is given for any $X\in\Cc(\colim_{i\in I}A_i)$ as a colimit (over $I^\op$) of the Beck--Chevalley maps
			\begin{equation*}
				m_{i!}t_i^*k_i^*X\to u^*n_{i!}k_i^*X,
			\end{equation*}
			each of which is an equivalence by assumption.
		\end{proof}
	\end{lemma}

	\begin{warning}\label{warn:S-cocomplete-vs-und-S-cocomplete}
		Even for a fiberwise cocomplete $T$-category, being $S$-cocomplete is not just a property of the underlying $S$-category: the former includes more Beck--Chevalley conditions.
	\end{warning}

	\begin{lemma}
		Let $F\colon\mathcal C\to\mathcal D$ be a $T$-functor of $S$-cocomplete $T$-categories. Then the following are equivalent:
		\begin{enumerate}
			\item The $T$-functor $F$ preserves fiberwise colimits and $\bbU_S$-colimits.
			\item The $T$-functor $F$ preserves fiberwise colimits and for every map $m$ in $S$ the Beck--Chevalley map $m_!F\to Fm_!$ is an equivalence.
			\item The $S$-functor $\iota^*F$ is $S$-cocontinuous.
		\end{enumerate}
		The dual statement for $S$-complete categories also holds.
	\end{lemma}
	\begin{proof}
		The equivalence between (1) and (2) follows just as in \Cref{lem:CharacterizationMCocompleteness}. Since the conditions in (2) only depend on the underlying $S$-functor $\iota^*F$, the equivalence between (2) and (3) is clear.
	\end{proof}

	\begin{definition}
		An $S$-functor $F$ satisfying the above equivalent conditions is called \emph{$S$-cocontinuous}. We write $\Cat_T^{S\textup{-cc}}\subset\Cat_T$ for the very large category of $S$-cocomplete $T$-categories and $S$-cocontinuous functors, and $\Pr_T^S\subset\Cat_T^{S\textup{-cc}}$ for the full subcategory spanned by the $S$-presentable $T$-categories.
	\end{definition}

	\begin{lemma}\label{lemma:fun-S-cocomplete}
		Let $\Cc,\Dd\in\Cat_T$ such that $\Dd$ is $S$-cocomplete. Then $\ul\Fun_T(\Cc,\Dd)$ is again $S$-cocomplete. Moreover, for any $F\colon\Cc\to\Cc'$ the restriction $\ul\Fun_T(\Cc',\Dd)\to\ul\Fun_T(\Cc,\Dd)$ is $S$-cocontinuous.
		\begin{proof}
			This is a special case of \cite{CLL_Global}*{Corollary~2.3.25}.
		\end{proof}
	\end{lemma}

	\begin{definition}
		We write $\ul\Fun_T^{S\text{-cc}}(\mathcal C,\mathcal D)\subset\ul\Fun_T(\mathcal C,\mathcal D)$ for the full subfunctor spanned in degree $X\in\PSh(T)$ by the $S$-cocontinuous functors $\mathcal C\to\ul\Fun_T(\ul X,\Dd)$.
	\end{definition}

	\begin{lemma}
		$\ul\Fun_T^{S\textup{-cc}}(\mathcal C,\mathcal D)$ defines a $T$-subcategory of $\ul\Fun_T(\Cc,\Dd)$.
		\begin{proof}
			If $X\to Y$ is any map in $\PSh(T)$, then Lemma~\ref{lemma:fun-S-cocomplete} shows that composing with the restriction $\ul\Fun_T(\ul Y,\Dd)\to\ul\Fun_T(\ul X ,\Dd)$ preserves $S$-cocontinuous functors. To see that this subfunctor is limit preserving, it suffices to observe that the functors $\ul\Fun_T(\ul Y,\Dd)\to\ul\Fun_T(\ul A,\Dd)$ for all $A\to Y$ with $A$ representable are jointly conservative and hence detect $S$-cocontinuity, cf.~the proof of \cite{CLL_Global}*{Proposition~2.3.28}.
		\end{proof}
	\end{lemma}

	\subsection{Colimits in Kan extensions}
	\label{subsec:ColimitsInKanExtensions}
	Recall that for any functor $\alpha\colon S\to T$ the restriction $\alpha^*\colon\Cat_T\to\Cat_S$ admits a right adjoint $\alpha_*$, which can be computed via restriction along $\alpha^*\colon\PSh(T)\to\PSh(S)$. We will now study the interplay of these adjoints with parametrized colimits and limits in the case that $\alpha=\iota$ is a cleft category.

	\begin{convention}
		For the rest of this subsection let us fix a cleft category $\iota\colon S\hookrightarrow T$ and a $T$-subcategory $\bbV^{(T)}\subset\bbU_S\subset\ul\Spc_T$. We will write $\bbV^{(S)}$ for the $S$-subcategory defined as the preimage of $\iota^*\bbV^{(T)}$ along the inclusion $\ul\Spc_S \hookrightarrow \iota^*\ul\Spc_T$.
	\end{convention}

	\begin{lemma}\label{lemma:V-restr}
		Let $A\in T$. Then $\iota^*\colon\PSh(T)_{/A}\to\PSh(S)_{/\iota^*A}$ restricts to a map $\bbV^{(T)}(A)\to\bbV^{(S)}(\iota^*A)$.
		\begin{proof}
			Let $(u\colon X\to A)\in\bbV^{(T)}(A)$ arbitrary. By assumption on $\bbV^{(T)}$, $u$ is admissible, so we have a pullback
			\begin{equation}\label{diag:counit-pullback-V}
				\begin{tikzcd}
					\iota_!\iota^*X\arrow[d, "\iota_!\iota^*u"']\arrow[dr,phantom,"\lrcorner"{very near start}] \arrow[r, "\epsilon"] & X\arrow[d, "u"]\\
					\iota_!\iota^*A\arrow[r, "\epsilon"'] & A
				\end{tikzcd}
			\end{equation}
			in $\PSh(T)$ by Proposition~\ref{prop:counit-pb}; in particular $\iota_!\iota^*u \in\bbV^{(T)}(\iota_!\iota^*A)$ as $\bbV^{(T)}$ is a $T$-subcategory of $\ul\Spc_T$. But then $\iota^*u\in\bbV^{(S)}(\iota^*A)$ as desired.
		\end{proof}
	\end{lemma}

	From now on we will conflate $\bbV^{(S)}$ and $\bbV^{(T)}$ and simply write $\bbV$ for both of them.

	\begin{theorem}\label{thm:iota*-cc}
		The adjunction $\iota^*\colon\Cat_T\rightleftarrows\Cat_S\noloc\iota_*$ restricts to an adjunction $\Cat_T^{\bbV\textup{-cc}}\rightleftarrows\Cat_S^{\bbV\textup{-cc}}$ between the categories of $\bbV$-cocomplete $T$- and $S$-categories, respectively, and $\bbV$-cocontinuous functors.
	\end{theorem}

	For the proof of the theorem we will use:

	\begin{lemma}[See \cite{CLL_Global}*{Lemma~2.3.17}]\label{lemma:restriction-cc}
		Let $f\colon\PSh(S)\to\PSh(T)$ be a left adjoint functor that preserves pullbacks, let $\bbV'\subset\ul\Spc_S$, and let $\bbV\subset\ul\Spc_T$ such that for every $A\in S$ and every $v\in\bbV(A)$ also $f(v)\in \bbV'(f(A))$. Then $f^*\colon\Cat_T\to\Cat_S$ restricts to $\Cat_T^{\bbV\textup{-cc}}\to\Cat_S^{\bbV'\textup{-cc}}$.\qed
	\end{lemma}

	\begin{proof}[Proof of Theorem~\ref{thm:iota*-cc}]
		The functor $\iota_!\colon\PSh(S)\to\PSh(T)$ preserves pullbacks by Lemma~\ref{lemma:iota!-pullback}, so the previous lemma shows that $\iota^*\colon\Cat_T\to\Cat_S$ preserves $\bbV$-cocomplete categories and $\bbV$-cocontinuous functors. In the same way, we deduce from Lemma~\ref{lemma:V-restr} that $\iota_*$ restricts accordingly. It remains that the unit and counit maps are $\bbV$-cocontinuous.

		For a $\bbV$-cocomplete $T$-category $\mathcal C$, the unit $\eta\colon \mathcal C\to\iota_*\iota^*\mathcal C$ is given by restriction along the \emph{co}unit $\epsilon\colon \iota_!\iota^*\Rightarrow\id$ of the adjunction $\iota_!\colon\PSh(S)\rightleftarrows\PSh(T)\noloc\iota^*$. Thus, if $A\in T$ and $(u\colon X\to A)\in\bbV(A)$ are arbitrary, then the Beck--Chevalley map $u_!\eta\to\eta u_!$ associated to the commutative diagram
		\[
		\begin{tikzcd}
			\Cc(A) \rar{\eta} \dar[swap]{u^*} & \iota_*\iota^*\Cc(A) \dar{u^*} \\
			\Cc(X) \rar[swap]{\eta} & \iota_*\iota^*\Cc(X)
		\end{tikzcd}
		\qquad = \qquad
		\begin{tikzcd}
		\Cc(A) \rar{\epsilon^*} \dar[swap]{u^*} & \Cc(\iota_!\iota^*A) \dar{(\iota_!\iota^*u)^*} \\
		\Cc(X) \rar[swap]{\epsilon^*} & \Cc(\iota_!\iota^*X)
		\end{tikzcd}
		\]
		is given by the Beck--Chevalley map $(\iota_!\iota^*u)_!\epsilon^*\to\eta^*u_!$ associated to the pullback $(\ref{diag:counit-pullback-V})$ and hence is an equivalence by $\bbV$-cocompleteness of $\mathcal C$.

		Similarly, if $\mathcal D$ is a $\bbV$-cocomplete $S$-category, then the counit $\epsilon\colon\iota^*\iota_*\mathcal D\to\mathcal D$ is given by restricting along the unit of the adjunction $\PSh(S)\rightleftarrows\PSh(T)$, and the Beck--Chevalley transformation $u_!\epsilon\to \epsilon u_!$ for $(u\colon X\to Y)\in\bbV(Y)$ is simply the Beck--Chevalley transformation for the pullback
		\begin{equation*}
			\begin{tikzcd}
				X\arrow[r, "\eta"]\arrow[d, "u"']\arrow[dr, phantom, "\lrcorner"{very near start}] & \iota^*\iota_!X\arrow[d, "\iota^*\iota_!u"]\\
				Y\arrow[r, "\eta"'] & \iota^*\iota_!Y
			\end{tikzcd}
		\end{equation*}
		in $\PSh(S)$ from Lemma~\ref{lemma:unit-pb}, hence an equivalence as claimed.
	\end{proof}

	\begin{corollary}\label{cor:S-cocomplete-adjunction}
		The adjunction $\iota^*\colon\Cat_T\rightleftarrows\Cat_S\noloc\iota_*$ restricts to adjunctions $\Cat_T^{S\textup{-cc}}\rightleftarrows\Cat_S^{S\textup{-cc}}$ and $\Pr_T^S\rightleftarrows\Pr^S_S$.
		\begin{proof}
			Clearly, $\iota^*$ and $\iota_*$ preserve fiberwise cocompleteness and cocontinuity; moreover, the unit $\mathcal C\to\iota_*\iota^*\mathcal C$ and counit $\iota^*\iota_*\mathcal D\to\mathcal D$ are simply given by restricting along suitable maps in $\PSh(T)$ or $\PSh(S)$ respectively, hence fiberwise cocontinuous.

			The first claim now follows from the special case $\bbV=\bbU_S$ of the previous theorem. For the second one it then only remains to observe that $\iota^*$ and $\iota_*$ preserve fiberwise presentability by the same reasoning as for fiberwise cocompleteness.
		\end{proof}
	\end{corollary}

	We close this discussion by giving an `internal' version of the above adjunction, for which we introduce:

	\begin{construction}\label{constr:no-pun-intended}
		For any $\mathcal C,\mathcal D\in\Cat_T$ we get a natural map $\iota^*\hskip0pt minus 1.5pt\colon\hskip0pt minus .5pt\iota^*\ul\Fun_T\hskip0pt minus .5pt(\mathcal C,\hskip0pt minus .5pt\mathcal D)\hskip0pt minus 1pt\to\hskip0pt minus 1pt\ul\Fun_S(\iota^*\mathcal C,\iota^*\mathcal D)$ as the composite
		\begin{align*}
			\iota^*\ul\Fun_T(\mathcal C,\mathcal D)&\xrightarrow[\hphantom{\iota^*\ev}]{\coev}
			\ul\Fun_S(\iota^*\mathcal C,\iota^*\mathcal C\times \iota^*\ul\Fun_T(\mathcal C,\mathcal D))\\
			&\iso \ul\Fun_S(\iota^*\mathcal C,\iota^*(\mathcal C\times\ul\Fun_T(\mathcal C,\mathcal D)))\\
			&\xrightarrow{\iota^*\ev}
			\ul\Fun_T(\iota^*\mathcal C,\iota^*\mathcal D),
		\end{align*}
		where the unlabelled equivalence is the canonical one, and $\ev\colon \Cc \times \ulFun_T(\Cc,\Dd) \to \Dd$ and $\coev\colon \Ee \to \ulFun_S(\iota^*\Cc, \iota^*\Cc \times \Ee)$ are the evaluation and coevaluation maps, adjunct to the respective identity maps. Put differently, for any fixed $\mathcal C$, this is the mate of the canonical natural equivalence filling the square
		\begin{equation}\label{diag:product-equivalence}
			\begin{tikzcd}
				\Cat_T\arrow[d, "\iota^*"']\arrow[r, "\mathcal C\times\blank"] &[1em] \Cat_T \arrow[d,"\iota^*"]\\
				\Cat_S\arrow[r, "\iota^*\Cc\times\blank"'] & \Cat_S \rlap.
			\end{tikzcd}
		\end{equation}
		Explicitly, this sends an object in degree $A\in T$ corresponding to $F\colon\ul{\iota(A)}\times\Cc\to\Dd$ to the composite
		\begin{equation*}
			\ul A\times\iota^*\Cc\xrightarrow{\;\eta\;}\iota^*\ul{\iota(A)}\times\iota^*\Cc\simeq\iota^*(\ul{\iota(A)}\times\Cc)\xrightarrow{\iota^*F}\iota^*D,
		\end{equation*}
		where $\eta$ refers to the adjunction $\iota_!\dashv \iota^*$.

		Passing to mates once more, we also obtain an equivalence $\Phi\colon\ul\Fun_T(\mathcal C,\iota_*\mathcal D)\simeq\iota_*\ul\Fun_S(\iota^*\mathcal C,\mathcal D)$ natural in $\mathcal C\in\Cat_T$ and $\mathcal D\in\Cat_S$, given for any $A\in T$ by sending a functor $F\colon\mathcal C\to \ul\Fun_T(\ul A,\iota_*\mathcal D)$ to the composite
		\begin{equation*}
			\iota^*\mathcal C\xrightarrow{\iota^*F}\iota^*\ul\Fun_T(\ul A,\iota_*\mathcal D)\xrightarrow{\iota^*}\ul\Fun_S(\iota^*\ul A,\iota^*\iota_*\mathcal D)\xrightarrow{\epsilon\circ - }\ul\Fun_S(\iota^*\ul A,\mathcal D).
		\end{equation*}
		The composition of the two rightmost arrows agrees with $\epsilon\circ\iota^* \Phi$ by the triangle identity, i.e.~$\Phi(F)$ is the adjunct of the composite
		\begin{equation*}
			\mathcal C\xrightarrow{\;F\;}\ul\Fun_T(\ul A,\iota_*\mathcal D)\mathrel{\smash{\xrightarrow[\raise12pt\hbox{$\scriptstyle\sim$}]{\,\;\Phi\,\;}}}\iota_*\ul\Fun_S(\iota^*\ul A,\mathcal D).
		\end{equation*}
		The equivalence $\Phi$ can accordingly be viewed as an `internal' version of the adjunction equivalence for $\iota^*\colon\Cat_T\rightleftarrows\Cat_S\noloc\iota_*$.
	\end{construction}

	\begin{corollary}\label{cor:iota**-cc-internal}
		Let $\Cc$ be an $S$-cocomplete $T$-category and $\Dd$ an $S$-cocomplete $S$-category. Then the previous construction restricts to an equivalence
		\begin{equation*}
			\ul\Fun_T^{S\textup{-cc}}(\Cc,\iota_*\Dd)\iso\iota_*\ul\Fun_S^{S\textup{-cc}}(\iota^*\Cc,\Dd).
		\end{equation*}
		\begin{proof}
			It only remains to show that $F\colon\Cc\to\ul\Fun_T(\ul X,\iota_*\Dd)$ is $S$-cocontinuous if and only if $\Phi(F)$ is so. However, by the above explicit description of $\Phi(F)$, this is precisely the statement of Corollary~\ref{cor:S-cocomplete-adjunction}.
		\end{proof}
	\end{corollary}

	\begin{remark}
		One can deduce from the previous corollary that the functor $\iota^*\colon\Cat_T^{S\textup{-cc}}\to \Cat_S^{S\textup{-cc}}$ from \Cref{cor:S-cocomplete-adjunction} is symmetric monoidal with respect to the symmetric monoidal structures defined in \cite[Section~8.2]{martiniwolf2022presentable}, applied to $\bbU_S$. It follows in particular that the subcategory $\Pr^S_T \subset \Cat_T^{S\textup{-cc}}$ is closed under tensor products, being the preimage along $\iota^*$ of the symmetric monoidal subcategory $\PrL_S \subset \Cat_S^{S\textup{-cc}}$. Since we will not make use of these symmetric monoidal structures in this paper, we will leave the details to the interested reader.
	\end{remark}

	\subsection{$\bm S$-cocompletion}
	\label{subsec:S-cocompletion}
	As an application of the above theory we can now reinterpret and extend work of Martini and Wolf on parametrized cocompletions:

	\begin{theorem}\label{thm:mixed-presheaves}
		Let $I$ be any small $T$-category. Then the unique $S$-cocontinuous \hbox{$S$-functor} $\iota_!\colon\ul\PSh_S(\iota^*I)=\ul\Fun_S(\iota^*I^\op,\ul\Spc_S)\to\iota^*\ul\Fun_T(I,\ul\Spc_T)=\iota^*\ul\PSh_T(I)$ compatible with the Yoneda embeddings is fully faithful. Moreover, its essential image is actually a $T$-subcategory, and this is the $T$-subcategory generated under $S$-colimits by the Yoneda image.
		\begin{proof}
			Write $\mathcal C$ for the full $T$-subcategory of $\ul\PSh_T(I)$ generated under $S$-colimits by the Yoneda image. Then \cite[Theorem~7.1.13]{martiniwolf2021limits} (for $\mathsf{U}$ the union of $\bbU_S$ and the constant $T$-categories) shows that restriction along the Yoneda embedding $y$ defines an equivalence
			\begin{equation*}
				\maps_{\Cat_T^{S\text{-cc}}}(\mathcal C,\mathcal D)\iso\maps_{\Cat_T}(I,\mathcal D)
			\end{equation*}
			for any $S$-cocomplete $\mathcal D$. Specializing to $\mathcal D=\iota_*\mathcal E$ for $\mathcal E\in\Cat_S^{S\text{-cc}}$ and appealing to Corollary~\ref{cor:S-cocomplete-adjunction} we see that restricting along $\iota^*(y)$ defines an equivalence
			\begin{equation*}
				\maps_{\Cat_S^{S\text{-cc}}}(\iota^*\mathcal C,\mathcal E)\iso\maps_{\Cat_S}(\iota^*I,\mathcal E).
			\end{equation*}
			However, the Yoneda embedding $\iota^* I\to\ul\PSh_S(\iota^*I)$ has the same property by \cite[Theorem~7.1.1]{martiniwolf2021limits}, so comparing corepresented functors shows that $\iota_!$ defines an equivalence $\ul\PSh_{S}(\iota^*I)\simeq\mathcal C$.
		\end{proof}
	\end{theorem}

	\begin{construction}
		We let $\iota_!\colon\ul\PSh_{S\triangleright T}(I)\to\ul\PSh_T(I)$
		denote the unique extension of $\iota_!\colon\ul\PSh_S(\iota^*I)\to\iota^*\ul\PSh_T(I)$ to a $T$-functor obtained from \Cref{thm:mixed-presheaves}. Note that for $I=1$ the terminal presheaf, this recovers the functor $\iota_!\colon\ul\Spc_{S\triangleright T}\to\ul\Spc_T$ from Lemma~\ref{lemma:iota!-ext}.

		By full faithfulness of $\iota_!$, there is then a unique lift of the $S$-parametrized Yoneda embedding $\iota^*I\to\ul\PSh_S(I)$ to a $T$-functor $y\colon I\to\ul\PSh_{S\triangleright T}(I)$ together with an equivalence between $\iota_!y$ and the $T$-parametrized Yoneda embedding $I\to\ul\PSh_T(I)$.
	\end{construction}

	\begin{corollary}
		In the above situation, $\ul\PSh_{S\triangleright T}(I)$ is $S$-presentable. For any $S$-cocomplete $T$-category $\mathcal D$, restriction along $y$ defines an equivalence
		\begin{equation*}
			\ul\Fun_T^{S\textup{-cc}}(\ul\PSh_{S\triangleright T}(I),\mathcal D)\iso\ul\Fun_T(I,\mathcal D).
		\end{equation*}
		\begin{proof}
			For $S$-presentability, we observe that $\ul\PSh_{S\triangleright T}(I)$ is $S$-cocomplete as it is equivalent to a subcategory of $\ul\Spc_T$ closed under $S$-colimits, and that for any $A\in T$ the category $\ul\PSh_{S\triangleright T}(I)(A)=\ul{\PSh}_S(\iota^*I)(A)$ is presentable by \Cref{ex:Functor_Cat_Presentable}.

			The universal property is an instance of \cite{martiniwolf2021limits}*{Theorem~7.1.13} as before.
		\end{proof}
	\end{corollary}

	\begin{corollary}\label{cor:spc-s-t-univ-property}
		The $T$-category $\ul\Spc_{S\triangleright T}$ is $S$-presentable. For any $S$-cocomplete $T$-category $\mathcal D$, evaluation at the terminal object defines an equivalence
		\begin{equation*}
			\ul\Fun_T^{S\textup{-cc}}(\ul\Spc_{S\triangleright T},\mathcal D)\iso\mathcal D.\qednow
		\end{equation*}
	\end{corollary}

	In the situation of the previous corollary we actually have further right adjoints:

	\begin{proposition}\label{prop:right-adjoints-adjunct}
		\begin{enumerate}
			\item The $S$-functor $\iota^*\colon\iota^*\ul\Spc_T\to\ul\Spc_S$ right adjoint to $\iota_!$ admits an $S$-right adjoint $\iota_*$.
			\item The adjunct $\tilde\iota^*\colon\ul\Spc_T\to\iota_*\ul\Spc_S$ of $\iota^*$ admits a $T$-right adjoint $\tilde\iota_*$.
		\end{enumerate}
		\begin{proof}
			We will prove the second statement. Corollary~\ref{cor:S-cocomplete-adjunction} then implies that $\iota^*$ is $S$-cocontinuous, so that the first statement is an instance of the Special Adjoint Functor Theorem.

			Recalling the definition (Construction~\ref{constr:iota^*}), $\iota^*$ is given by the composite
			\begin{equation*}
				\PSh(T)_{/\iota_!(\bullet)}\xrightarrow{\,\iota^*\,}\PSh(S)_{/\iota^*\iota_!(\bullet)}\xrightarrow{\eta^*}\PSh(S)_{/\bullet},
			\end{equation*}
			i.e.~it is adjunct to $(\iota^*)_{/\bullet}\colon\PSh(T)_{/\bullet}\to\PSh(S)_{/\iota^*(\bullet)}$. The latter obviously preserves $T$-colimits and has a pointwise right adjoint given by the composites
			\begin{equation*}
				\PSh(S)_{/\iota^*(\bullet)}\xrightarrow{\iota_*}\PSh(T)_{/\iota_*\iota^*(\bullet)}\xrightarrow{\eta^*}\PSh(T)_{/\bullet}.\qedhere
			\end{equation*}
		\end{proof}
	\end{proposition}

	\begin{remark}
		We close this discussion by giving a different interpretation of the functor $\iota^*\colon\iota^*\ul\Spc_T\to\ul\Spc_S$. For this we recall once more the equivalences $\ul\Spc_T\simeq\PSh(T_{/\bullet})$, $\ul\Spc_S\simeq\PSh(S_{/\bullet})$ from Example~\ref{ex:spc-T}; we claim that under these equivalences our functor $\iota^*$ is given by the restriction $F\colon\iota^*\PSh(T_{/\bullet})\to\PSh(S_{/\bullet})$ along the $S$-natural transformation $\iota^*{T_{/\bullet}}\to S_{/\bullet}$.

		While this can be carefully proven by hand, we will instead resort to a sequence of cheap tricks that avoids ever talking about coherences. Namely, by the universal property of $S$-spaces it suffices to show that $F$ admits a left adjoint and that this preserves terminal objects. Indeed, $F$ admits a pointwise left adjoint given by the left Kan extension functors $(\iota_{/A})_!\colon\PSh(S_{/A})\to\PSh(T_{/A})$, and each of these preserves terminal objects (as they are simply represented by the respective identity maps). It remains to show that for every $g\colon A\to B$ in $S$ the Beck--Chevalley transformation $(\iota_{/A})_!g^*\Rightarrow g^*(\iota_{/B})_!$ is an equivalence. By full faithfulness of $(\iota_{/A})_!$ and $(\iota_{/B})_!$, this is equivalent to demanding that $g^*$ preserves the essential images, for which it is turn enough to show that there is \emph{some} equivalence $(\iota_{/A})_!g^*\simeq g^*(\iota_{/B})_!$. This however follows simply from the equivalences $(\ref{diag:slice-in-out})$ and the fact that $\iota^*$ has a left adjoint.

		Note that this argument more generally shows that $\iota^*$ corresponds under \emph{any pair of equivalences} to the above restriction functor $\iota^*\PSh(T_{/\bullet})\to\PSh(S_{/\bullet})$.
	\end{remark}

	\section{The universal property of equivariant spaces}\label{sec:spaces}
	Building on the above, we will establish a universal property of equivariant unstable homotopy theory in this section. We begin by introducing the object of study:

	\begin{construction}\label{constr:G-SSet}
		Recall the global indexing category $\Glo$ from Example~\ref{ex:globalCategory}, which comes with an embedding $B\colon\Glo\hookrightarrow\Grpd$ into the $(2,1)$-category of $1$-groupoids, identifying it with the connected $1$-groupoids. We write $\cat{SSet}$ for the $1$-category of simplicial sets, and we define a strict $2$-functor $\cat{$\bm\bullet$-SSet}\colon\Glo^\op\to\Cat_{1}$ into the $(2,1)$-category of $1$-categories as the composite
		\begin{equation*}
			\Glo^\op\xhookrightarrow{\;B\;}\Grpd^\op\xrightarrow{\Fun(\blank,\cat{SSet})} \Cat_{1}.
		\end{equation*}
		This lifts to a functor into the $(2,1)$-category $\RelCat$ of relative categories, homotopical functors, and natural isomorphisms by equipping $\cat{$\bm G$-\hskip0pt minus 1ptSSet}\hskip0pt minus 1pt\coloneqq\hskip0pt minus 1pt\Fun(BG,\hskip0pt minus .5pt\cat{SSet})$ with the \emph{$G$-equivariant weak equivalences}, i.e.~the class of those maps $f$ such that $f^H$ is a weak equivalence for every subgroup $H\subset G$, or equivalently such that the geometric realization $|f|$ is a $G$-equivariant homotopy equivalence.

		Postcomposing with the localization functor $\RelCat\to\Cat$, we obtain a global category $\ul\S\colon\Glo^\op\to\Cat$. We call $\ul\S$ the \emph{global category of equivariant spaces}.
	\end{construction}

	Note that $\ul\S(G)\eqqcolon\mathscr S_G$ is the usual category of $G$-spaces, and for any $\alpha\colon G\to G'$ the structure map $\alpha^*\colon\S_{G'}\to\S_G$ is the usual restriction functor.

	\begin{notation}
		Recall from Example~\ref{ex:Orb} that the wide subcategory $\Orb\subset \Glo$ of injective group homomorphisms is an example of a cleft category, giving rise to notions of $\Orb$-cocompleteness and $\Orb$-presentability. To emphasize the connections to equivariant homotopy theory obtained in this article we will refer to these as \textit{equivariant cocompleteness} and \textit{equivariant presentability}.

		Similarly an $\Orb$-cocontinuous functor $F\colon \Cc\rightarrow \Dd$ between equivariantly cocomplete global categories will be called \textit{equivariantly cocontinuous}, and we will write $\ulFun^{\textup{eq-cc}}_{\Glo}(\Cc,\Dd)$ for the global category $\ulFun^{\Orb\textup{-cc}}_{\Glo}(\Cc,\Dd)$ of equivariantly cocontinuous functor.
	\end{notation}

	We can now state the main result of this section:

	\begin{theorem}\label{thm:orbi-S-univ-prop}
		The global category $\ul\S$ is equivariantly presentable. Moreover, it is the free equivariantly cocomplete global category in the following sense: for any equivariantly cocomplete global category $\mathcal D$, evaluating at the $1$-point space provides an equivalence
		\begin{equation*}
			\ul\Fun^{\textup{eq-cc}}_{\Glo}(\ul\S,\mathcal D)\iso\mathcal D.
		\end{equation*}
	\end{theorem}

	\begin{remark}
	While we will not prove this here, we remark that $\ul\S$ is in fact even globally presentable: this is a rather straightforward model categorical computation using that the left adjoints $\alpha_!\colon\cat{$\bm H$-SSet}\to\cat{$\bm G$-SSet}$ of the restrictions are again homotopical and that the Beck--Chevalley conditions hold on the pointset level by smooth and proper base change.

	However, this `extra presentability' should be considered as an anomaly for two reasons: firstly, it is something rather specific to $\Orb\subset\Glo$, and does not hold for general cleft categories $S\subset T$; secondly, it breaks down as soon as we pass to the semiadditive and stable world, cf.~Warning~\ref{warn:sp-no-fiberwise-limits}.
	\end{remark}

	In view of Corollary~\ref{cor:spc-s-t-univ-property}, the second half of the theorem can be reformulated as follows:

	\begin{theorem}\label{thm:orbi-S-equivalence}
		The essentially unique equivariantly cocontinuous global functor $\ul{\Spc}_{\Orb\triangleright\Glo}\to\ul\S$ preserving the terminal object is an equivalence.
	\end{theorem}

	In fact, our proof of these two theorems below will proceed the other way round by first establishing an equivalence $\ul\Spc_{\Orb\triangleright\Glo}\simeq\ul\S$ and then deducing all the remaining statements from this.

	To motivate what comes next, let us briefly outline how we will obtain such an equivalence. Since we have already exhibited $\ul\Spc_{\Orb\triangleright\Glo}$ as a full subcategory of $\ul\Spc_{\Glo}$, it will suffice to also exhibit $\ul\S$ as a full subcategory of $\ul\Spc_{\Glo}$ with the same essential image. To do this, we will model $\ul\Spc_{\Glo}$ concretely as the localization of a 1-category, and then build the functor explicitly before localization. This is accomplished by the theory of $G$-global spaces and the results of \cite{CLL_Global}, which we recall now.

	\begin{construction}
		We write $I$ for the category of finite sets and injections and $\mathcal I$ for the simplicial category obtained by applying the right adjoint $E\colon\cat{Set}\to\cat{SSet}$ of the evaluation functor $\ev_0\colon\cat{SSet}\to\cat{Set}$ to all hom sets. We write $\cat{$\bm{\mathcal I}$-SSet}$ for the category of enriched functors $\mathcal I\to\cat{SSet}$, and for any $G$ we denote the category of $G$-objects in $\cat{$\bm{\mathcal I}$-SSet}$ by $\cat{$\bm G$-$\bm{\mathcal I}$-SSet}$. Analogously to Construction~\ref{constr:G-SSet} these assemble into a functor $\cat{$\bm\bullet$-$\bm{\mathcal I}$-SSet}\colon\Glo^\op\to\Cat_1$.

		We can evaluate a $G$-$\mathcal I$-simplicial set $X$ at any (not necessarily finite) set $A$ via
		\begin{equation*}
			X(A)\coloneqq\mathop{\text{colim}}\limits_{B\subset A\text{ finite}} X(B),
		\end{equation*}
		and this acquires an action of the symmetric group $\Sigma_A$ via permuting the factors. In particular, if $A$ is a $G$-set, then we can equip $X(A)$ with the diagonal $G$-action, yielding a functor $\ev_A\colon\cat{$\bm{G}$-$\bm{\mathcal I}$-SSet}\to\cat{$\bm G$-SSet}$.

		We now call a map $f\colon X\to Y$ of $G$-$\mathcal I$-simplicial sets a \emph{$G$-equivariant weak equivalence} if $f(\mathcal U)$ is a $G$-equivariant weak equivalence in $\cat{$\bm G$-SSet}$ for some, hence any complete $G$-set universe $\mathcal U$ (i.e.~a countable $G$-set into which any other countable $G$-set embeds equivariantly). Finally, we call $f$ a \emph{$G$-global weak equivalence} if $\phi^*f$ is an $H$-equivariant weak equivalence of $H$-$\mathcal I$-simplicial sets for any homomorphism $\phi\colon H\to G$ from a finite group to $G$.

		Clearly, for any $\alpha\colon G\to G'$ the restriction functor $\alpha^*\colon \cat{$\bm{G'}$-$\bm{\mathcal I}$-SSet}\to\cat{$\bm{G}$-$\bm{\mathcal I}$-SSet}$, sends $G'$-global weak equivalences to $G$-global weak equivalences, lifting $\cat{$\bm\bullet$-$\bm{\mathcal I}$-SSet}$ to $\Glo^\op\to\RelCat$. Localizing, we then again get a global category, which we denote by $\ul\S^\text{gl}$.
	\end{construction}

	We write $\S^\text{gl}_G\coloneqq\ul\S^\text{gl}(G)$ and call it the \emph{category of $G$-global spaces}. Note that \cite{CLL_Global} uses the notation `$\ul\S_{\mathcal I}^\text{gl}$' for the above global category and reserves $\ul\S^\text{gl}$ for a different, but equivalent, model based on actions of a certain `universal finite group.' In the present paper, however, we will only be interested in the above approach.

	\begin{remark}
		The $G$-global weak equivalences are part of several model structures on $\cat{$\bm G$-$\bm{\mathcal I}$-SSet}$, see~\cite[Section~1.4]{g-global}. We will not need them explicitly in this section, but they will make an indirect appearance in Section~\ref{sec:gamma-orbi}.
	\end{remark}

	Our interest in $\ul\S^\text{gl}$ comes from the following `global' version of Theorem~\ref{thm:orbi-S-equivalence}:

	\begin{theorem}[See~\cite{CLL_Global}*{Theorem~3.3.1 and Corollary~3.2.5}]\label{thm:global-spaces-universal}
		The global category $\ul\S^\textup{gl}$ is globally presentable. The unique globally cocontinuous functor $\ul\Spc_{\Glo}\to\ul\S^\textup{gl}$ preserving the terminal object is an equivalence.\qed
	\end{theorem}

	On the other hand we can relate the global categories of global and equivariant spaces as follows:

	\begin{lemma}\label{lemma:triv-S-Sgl}
		There exists a global functor $\textup{const}\colon\ul\S\to\ul\S^\textup{gl}$ with the following properties:
		\begin{enumerate}
			\item $\textup{const}$ is fully faithful and sends the terminal object of $\S$ to the terminal object of $\S^\textup{gl}$,
			\item  it admits an $\Orb$-right adjoint $\cat{R}\,\ev\colon\ul\S^\textup{gl}|_{\Orb}\to\ul\S|_{\Orb}$.
		\end{enumerate}
	\end{lemma}

	Once the above two theorems have been established, we will see that this adjunction is actually uniquely described by the requirement that the left adjoint preserve the terminal object.

	\begin{proof}
		The functor $\const\colon\cat{$\bm G$-SSet}\to\cat{$\bm G$-$\bm{\mathcal I}$-SSet}$ is homotopical and strictly natural in $G$, so it induces a global functor $\const\colon\ul\S\to\ul\S^\text{gl}$. By \cite[Corollary~1.4.56]{g-global} this functor is fully faithful, and it admits a pointwise right adjoint (given by the right derived functor of $\ev_\emptyset\colon\cat{$\bm G$-$\bm{\mathcal I}$-SSet}\to\cat{$\bm G$-SSet}$).

		To complete the proof, it only remains to establish the Beck--Chevalley condition for the pointwise right adjoint, or equivalently that for any injective $\alpha\colon G\to G'$ the mate transform $\alpha_!\circ\const\Rightarrow{\const}\circ\alpha_!$ is an equivalence of functors $\S_G\to\S^\text{gl}_{G'}$. But indeed, this holds on the pointset level by direct inspection, so the claim follows as $\alpha_!\colon\cat{$\bm G$-$\bm{\mathcal I}$-SSet}\to\cat{$\bm{G'}$-$\bm{\mathcal I}$-SSet}$ is homotopical by \cite[Lemma~1.4.42]{g-global} while $\alpha_!\colon\cat{$\bm G$-SSet}\to\cat{$\bm G'$-SSet}$ is so by (a well-known special case of) Proposition~1.1.18 of \emph{op.~cit.}
	\end{proof}

	\begin{proof}[Proof of Theorems~\ref{thm:orbi-S-univ-prop} and~\ref{thm:orbi-S-equivalence}]
		Lemma~\ref{lemma:triv-S-Sgl} provides a fully faithful global functor $\const\colon\ul\S\to\ul\S^\text{gl}$. Now the right hand side is globally cocomplete, hence in particular equivariantly cocomplete. Moreover, the essential image of the functor $\const$ is closed under all equivariant colimits as it is an $\Orb$-left adjoint. Thus, also $\ul\S$ is equivariantly cocomplete.

		Appealing to Corollary~\ref{cor:spc-s-t-univ-property}, we therefore see that there is an essentially unique equivariantly cocontinuous functor $F\colon\ul{\Spc}_{\Orb\triangleright\Glo}\to\ul\S$ preserving the terminal object, and we claim that this is an equivalence. For this, we consider the diagram
		\begin{equation*}
			\begin{tikzcd}
				\ul\Spc_{\Orb\triangleright\Glo}\arrow[d, "\iota_!"'] \arrow[r, "F"] & \ul\S\arrow[d, "\const"]\\
				\ul\Spc_{\Glo}\arrow[r, "\sim"'] & \ul\S^\text{gl}
			\end{tikzcd}
		\end{equation*}
		of global functors where the lower equivalence is as in Theorem~\ref{thm:global-spaces-universal}; both paths through this diagram are equivariantly cocontinuous and preserve the terminal object, so this commutes by the universal property of $\ul{\Spc}_{\Orb\triangleright\Glo}$. Moreover, the vertical arrows are fully faithful by Theorem~\ref{thm:mixed-presheaves} and Lemma~\ref{lemma:triv-S-Sgl}, respectively. It follows that also $F$ is fully faithful.

		To see that each $F_G\colon\ul{\Spc}_{\Orb}(G)\to\S_G$ is essentially surjective, we observe that $F_G$ is a fully faithful left adjoint, so that its essential image is closed under all colimits. On the other hand, by Elmendorf's Theorem \cite{elmendorf} (or simply looking at the standard generating cofibrations), $\S_G$ is generated under colimits by the $G/H$'s for subgroups $H\subset G$, so it is enough that each $G/H$ is contained in the essential image. However, $G/H=i_!(*)$, where $i\colon H\hookrightarrow G$ denotes the inclusion, so $F_G(i_!(*))\simeq i_!F_H(*)\simeq i_!(*)\simeq G/H$ by the defining properties of $F$.

		Finally, the universal property of $\ul\S$ follows from combining the above with Corollary~\ref{cor:spc-s-t-univ-property}.
	\end{proof}

	\section{The semiadditive story}\label{sec:semiadd}
	Let $P\subset T$ be an \emph{atomic orbital subcategory} in the sense of Definition~\ref{def:atomic-orbital}. In \cite{CLL_Global} we defined what it means for a $T$-category $\Cc$ to be \textit{$P$-semiadditive}: roughly speaking this demands that for every morphism $p\colon A \to B$ in $P$ the restriction functor $p^*\colon \Cc(B) \to \Cc(A)$ admits both a left adjoint $p_!$ and a right adjoint $p_*$ which are canonically equivalent to each other. The goal of this section is to give a description of the universal $S$-presentable $P$-semiadditive $T$-category $\ul\CMon^P_{S\triangleright T}$ for a cleft $S \subset T$ satisfying $P \subset S$.

	The description we give proceeds by reduction to the fully presentable case $S = T$, which was treated in \cite{CLL_Global} and will be recalled in \Cref{subsec:P-sadd} below: the universal $T$-presentable $P$-semiadditive $T$-category $\ul\CMon^P_{T}$ is given by certain \textit{$P$-commutative monoids} $\ulfinptdPsets \to \ul{\Spc}_T$, see \Cref{def:CMon^P}. The $T$-category $\ul\CMon^P_{S\triangleright T}$ is then given as a certain full subcategory of $\ul\CMon^P_{T}$ whose underlying $S$-category is the $S$-category $\ul\CMon^P_{S}$ of $S$-parametrized $P$-commutative monoids. Proving that this is indeed the universal $S$-presentable $P$-semiadditive $T$-category is somewhat technical and will be the content of \Cref{subsec:Univ_Prop_CMon^P_ST}. In \Cref{subsec:Add_Adj} we will investigate in more detail the relationship between the $S$-categories $\ul\CMon^P_{S}$ and $\iota^*\ul\CMon^P_{T}$.

	\subsection{$\bm P$-semiadditivity and $\bm P$-commutative monoids} \label{subsec:P-sadd}
	We begin with a recollection of the relevant material from \cite{CLL_Global} on parametrized semiadditivity. Throughout, we fix an atomic orbital subcategory $P\subset T$.

	\begin{construction}
		We write $\mathbb F_T\subset\PSh(T)$ for the finite coproduct completion of $T$ and $\mathbb F^P_T$ for the finite coproduct completion of $P$, viewed as a subcategory of $\mathbb F_T$.

		We define a $T$-subcategory $\ulfinPsets\subset \ul\Spc_T$ by letting $\ulfinPsets(B)$ be the full subcategory of $\PSh(T)_{/B}$ spanned by objects of the form $(p_i)\colon \coprod_{i=1}^{n} A_i \to B$ such that each morphism $p_i\colon A_i \to B$ is in $P$; put differently, this is the slice $(\mathbb F^P_T)_{/B}$. Note that by atomic orbitality of $P$, $\ulfinPsets$ indeed forms a $T$-subcategory of $\ul{\Spc}_T$.
	\end{construction}

	\begin{definition}
		We say a $T$-category has \emph{finite $P$-products} or \emph{finite $P$-coproducts} if it has $\ulfinPsets$-limits or $\ulfinPsets$-colimits, respectively, in the sense of Definition~\ref{defi:U-cc}.
	\end{definition}

	\begin{definition}
		A $T$-category $\Cc$ is called \emph{pointed} if it factors through the non-full subcategory $\Cat_*\subset\Cat$ of categories with zero objects and functors preserving the zero object.
	\end{definition}

	\begin{construction}
		Let $\Cc$ be a pointed $T$-category which has finite $P$-coproducts, and let $\mathcal D$ be a $T$-category with finite $P$-products. For any functor $F\colon\mathcal C\to\mathcal D$ and any $p\colon A\to B$ in $\mathbb F^P_T$, \cite{CLL_Global}*{Construction 4.6.1} defines a \emph{relative norm map}
		\[
		\Nm^F_{p}\colon F_B \circ p_! \implies p_* \circ F_A.
		\]
		If $\mathcal C$ also has finite $P$-products, we write $\Nm_p\colon p_!\Rightarrow p_*$ for the relative norm map of $\id_{\mathcal C}$, and simply call it the \emph{norm map}.
	\end{construction}

	\begin{definition}\label{defi:P-semiadd}
		A $T$-category $\mathcal C$ is called \emph{$P$-semiadditive} if it is pointed, has finite $P$-products and $P$-coproducts, and they agree in the sense that for every $p$ in $\mathbb F^P_T$ the norm map $\Nm_p\colon p_!\Rightarrow p_*$ is an equivalence.
	\end{definition}

	\begin{example}
		When $P\subset T$ is the subcategory $\Orb\subset \Glo$ from \Cref{ex:Orb}, the previous definition specializes to the notion of \textit{equivariant semiadditivity} from \cite{CLL_Global}.
	\end{example}

	\begin{example}\label{ex:G_semi}
		When $P\subset T$ equals $\Orb_G\subset \Orb_G$ for a finite group $G$, the notion of semiadditivity obtained agrees with $G$-semiadditivity as defined in \cite{nardin2016exposeIV}, see \cite{CLL_Global}*{Proposition 4.6.4}.
	\end{example}

	\begin{definition}
		Let $F\colon\mathcal C\to\mathcal D$ be a functor of $T$-categories, such that $\mathcal C$ is pointed and has finite $P$-coproducts, while $\mathcal D$ has finite $P$-products. Then $F$ is called \emph{$P$-semiadditive} if it sends $P$-coproducts to $P$-products in the sense that the relative norm map $\Nm^F_p\colon Fp_!\Rightarrow p_*F$ is an equivalence for every $p$ in $\mathbb F^P_T$.
	\end{definition}

	\begin{definition}
		We write $\Cat_T^{P\text-\oplus}\subset\Cat_T$ for the non-full subcategory of $P$-semiadditive categories and $P$-semiadditive $T$-functors.
	\end{definition}

	By \cite{CLL_Global}*{Proposition~4.6.14}, the morphisms of $\Cat_T^{P\text-\oplus}$ are equivalently the $\ul{\mathbb F}^P_T$-cocontinuous or $\ul{\mathbb F}^P_T$-continuous $T$-functors.

	\begin{remark}\label{rk:P-semiadd-underlying}
		Similarly to Warning~\ref{warn:S-cocomplete-vs-und-S-cocomplete}, having finite $P$-(co)products is not just a property of the underlying $P$-category. On the other hand, if a $T$-category either has finite $P$-coproducts or finite $P$-products, then it is $P$-semiadditive if and only if its underlying $P$-category is so \cite{CLL_Global}*{Lemma~4.5.2 and Lemma~4.6.4}.
	\end{remark}

	\begin{definition}
		In the above situation, we write $\ul\Fun_T^{P\text-\oplus}(\mathcal C,\mathcal D)$ for the param\-etrized subcategory spanned in degree $X\in\PSh(T)$ by the $P$-semiadditive functors $\mathcal C\to\ul\Fun_T(\ul X,\mathcal D)$.
	\end{definition}

	Note that the above is indeed a $T$-subcategory by \cite{CLL_Global}*{Corollary~4.6.10}.

	\begin{definition}
		We define $\ulfinptdPsets$, the $T$-category of \textit{finite pointed $P$-sets}, to be the essential image of $\ulfinPsets$ under the functor $(-)_+\colon \ul\Spc_T\rightarrow \ul\Spc_{T,\ast}$ which adds a disjoint basepoint.
	\end{definition}

	\begin{definition}\label{def:CMon^P}
		Given a $T$-category $\Cc$ with $P$-products we define $\ul\CMon^P(\Cc)$, the \textit{$T$-category of $P$-commutative monoids in $\Cc$}, as $\ul\Fun_T^{P\text-\oplus}(\ul{\mathbb F}^P_{T,*},\mathcal C)$. If $\Cc = \ul\Spc_T$, we write $\ul\CMon^P_T\coloneqq\ul\CMon^P(\ul\Spc_T)$.
	\end{definition}

	This construction enjoys several universal properties. To express them we introduce:

	\begin{construction}
		Let $\mathcal C$ have finite $P$-products. Evaluation at the global section $S^0\coloneqq(\id)_+\in\ul{\mathbb F}^P_{T,*}(1)\subset(\PSh(T)_{/1})_*$ gives a forgetful functor
		\[\mathbb U\colon \ul\CMon_P(\Cc)\rightarrow \Cc.\]
	\end{construction}

	\begin{construction}
		Assume $\mathcal C$ is presentable. Then $\ul\CMon^P(\mathcal C)\hookrightarrow\ul\Fun_T(\ul{\mathbb F}^P_{T,*},\mathcal C)$ admits a left adjoint $L^{P\text-\oplus}$ by \cite{CLL_Global}*{Proposition~4.6.15}. In particular, the functor $\mathbb U$ from the previous construction has a left adjoint given by composing the left Kan extension functor $(S^0)_!\colon\mathcal C\to\ul\Fun_T(\ul{\mathbb F}^P_{T,*},\mathcal C)$ \cite{martiniwolf2021limits}*{Corollary~6.3.7} with $L^{P\text-\oplus}$.
	\end{construction}

	\begin{theorem}[See \cite{CLL_Global}*{Theorem~4.8.10}]\label{thm:Semiadd_omnibus}
		Let $\Cc$ be a $T$-category with finite $P$-products. The functor $\mathbb{U}\colon\ulPCMon(\Cc)\to \Cc$ exhibits $\ulPCMon(\Cc)$ as the $P$-semiadditive envelope of $\Cc$ in the following sense: for every $P$-semiadditive $T$-category $\tcat$ postcomposition with $\mathbb{U}$ induces an equivalence
		\begin{equation*}
			\ul\Fun^{P\text-\times}(\tcat, \mathbb{U})\colon\ul\Fun^{P\text-\oplus}(\tcat,\ulPCMon(\Cc))\iso \ul\Fun^{P\text-\times}(\tcat, \Cc).
		\end{equation*}
		Suppose now that $\Cc$ is moreover presentable. Then the left adjoint $\mathbb{P}$ of $\mathbb{U}$ exhibits $\ulPCMon(\Cc)$ as the presentable $P$-semiadditive completion of $\Cc$ in the following sense:~for any presentable $P$-semiadditive $T$-category $\tcat$ precomposition with $\mathbb{P}$ yields an equivalence
		\begin{equation*}
			\ul\Fun^\textup{$T$-cc}(\mathbb{P}, \tcat)\colon\ul\Fun^\textup{$T$-cc}(\ulPCMon(\Cc), \tcat)\iso\ul\Fun^\textup{$T$-cc}(\Cc, \tcat).\qednow
		\end{equation*}
	\end{theorem}

	\begin{theorem}[See \cite{CLL_Global}*{Theorem~4.8.11}]\label{thm:semiadd-T-spc}
		The $T$-category $\ul\CMon^P_T$ is $P$-semiadditive and $T$-presentable. Moreover, it has the following universal property: for any locally small $T$-cocomplete $P$-semiadditive $\mathcal D$, evaluation at $\mathbb P(*)\simeq L^{P\text-\oplus}y(S^0)$ induces an equivalence
		\begin{equation*}
			\ul\Fun^\textup{$T$-cc}(\ul\CMon^P_T,\mathcal D)\iso\mathcal D.
		\end{equation*}
	\end{theorem}

	Let us end this with the following result, which identifies the universal $P$-semiadditive $T$-category after restricting to slices.

	\begin{proposition}\label{prop:cmon-slice}
		Let $P\subset T$ be atomic orbital, let $A\in T$, write $\pi_A\colon T_{/A} \to T$ for the forgetful functor, and let $T_{/A}^P\coloneqq\pi_A^{-1}(P)\subset T_{/A}$ again. Then $\pi_A^*\ul\CMon^P_T$ is $T_{/A}$-cocomplete and $T_{/A}^P$-semiadditive, and the unique left adjoint $\cramped{\ul\CMon^{T^{P\rule[-1pt]{0pt}{0pt}}_{\smash{/A}}\rule[-1pt]{0pt}{0pt}}_{T_{/A}}}\to\pi_A^*\ul\CMon^P_T$ preserving $\mathbb P(*)$ is an equivalence.
		\begin{proof}
		By \cite{CLL_Global}*{Proposition~2.3.26 and Corollary~4.6.9} $\pi_A^*\dashv\pi_{A*}$ restricts to an adjunction \[\Cat_T^{P\text-\oplus, T\textup{-cc}}\rightleftarrows\Cat_{T_{/A}}^{T_{/A}^P\text-\oplus,T_{/A}\textup{-cc}},\] so the claim follows as before by comparing corepresented functors.
		\end{proof}

	\end{proposition}

	\subsection{The free $\bm P$-semiadditive $\bm S$-presentable $\bm T$-category} \label{subsec:Univ_Prop_CMon^P_ST}
	Let now $P\subset T$ be an atomic orbital subcategory satisfying $P \subset S$. As the main results of this section, we will prove the following `partially presentable' versions of the previous theorem:

	\begin{theorem}\label{thm:CMon-ST-adj}
		There is a unique $S$-cocontinuous functor $\iota_!\colon\ul\CMon^P_S\to\iota^*\ul\CMon^P_T$ sending $\mathbb P(*)$ to $\mathbb P(*)$. Moreover, $\iota_!$ is fully faithful, and it sits in a sequence of $S$-adjoints $\iota_!\dashv\iota^*\dashv\iota_*$.
	\end{theorem}

	\begin{theorem}\label{thm:CMon-ST-univ-prop}
		There exists a unique $T$-category $\ul\CMon^P_{S\triangleright T}$ equipped with a fully faithful $T$-functor $\iota_!\colon \ul\CMon^P_{S\triangleright T}\hookrightarrow\ul\CMon^P_T$ whose underlying $S$-functor is the inclusion $\iota_!\colon \ul\CMon^P_S\hookrightarrow\iota^*\ul\CMon^P_T$ from \Cref{thm:CMon-ST-adj}. Moreover, $\ul\CMon^P_{S\triangleright T}$ is $S$-presentable, $P$-semiadditive, and it has the following universal property: for any $S$-cocomplete $P$-semiadditive $T$-category $\mathcal D$ evaluation at a certain global section $\mathbb P(*)$ defines an equivalence
		\begin{equation*}
			\ul\Fun_T^{S\textup{-cc}}(\ul\CMon^P_{S\triangleright T},\mathcal D)\iso\mathcal D.
		\end{equation*}
	\end{theorem}

	\begin{remark}
		Note that in contrast to Theorem~\ref{thm:semiadd-T-spc} there is no local smallness condition on $\Dd$ here anymore; in particular, for $S=T$ this improves upon our result in \cite{CLL_Global}.
	\end{remark}

	The proof of these two theorems will occupy the rest of this section.

	\subsubsection{Construction of the universal example} As our first step, we will construct \emph{some} $P$-semiadditive $S$-cocomplete $T$-category $\mathcal C$ with the correct universal property, which must then a posteriori agree with the $T$-category $\ul\CMon^P_{S\triangleright T}$ from \Cref{thm:CMon-ST-univ-prop}. More precisely, we want to show:

	\begin{proposition}\label{prop:C-bar}
		Write ${\mathcal C}\subset\ul\CMon^P_T$ for the full subcategory generated under $S$-colimits by $\mathbb P(*)$. Then ${\mathcal C}$ is a $P$-semiadditive $S$-cocomplete $T$-category, and for any other such $\mathcal D$ evaluating at $\mathbb P(*)$ yields an equivalence $\ul\Fun^{S\textup{-cc}}_T({\mathcal C},\mathcal D)\simeq\mathcal D$.
	\end{proposition}

	The basic idea will be to deduce the universal property of $\mathcal C$ from the one for $\ul\CMon^P_T$. However, we only understand maps from the latter to \emph{$T$-cocomplete} categories, so we will have to embed a general $P$-semiadditive $S$-cocomplete $T$-category into a $T$-cocomplete one first. However, in this process some size issues pop up; to avoid any ambiguities, we will therefore for once refer back to our chosen Grothendieck universes explicitly:

	\begin{lemma}\label{lemma:semiadd-T-cocompl}
		Let $\mathcal D$ be an $S$-cocomplete $P$-semiadditive $\groV$-small $T$-category. Then there exists a $\groW$-small, locally $\groV$-small $P$-semiadditive $T$-category $\mathcal D'$ having all $\groV$-small $T$-colimits together with a fully faithful functor $j\colon\mathcal D\to\mathcal D'$ preserving all $\groU$-small $S$-colimits.
		\begin{proof}
			Write $\ul\SPC_T$ for the $\groW$-small and locally $\groV$-small $T$-category of $\groV$-small spaces. Then the Yoneda embedding $\mathcal D^\op\to\ul\Fun_T(\mathcal D,\ul\SPC_T)$ actually lands in the full subcategory $\mathcal E\coloneqq\ul\Fun_T^{P\text-\times}(\mathcal D,\ul\SPC_T)$ by \cite{martiniwolf2021limits}*{Corollary~4.4.8}. Now $\mathcal E$ is closed under all $\groV$-small $T$-limits and the Yoneda embedding preserves all $\groU$-small $S$-limits by Proposition~4.4.7 of \emph{op.~cit.} Thus, the Yoneda embedding $\mathcal D^\op\to\mathcal E$ is a fully faithful functor into a category with all $\groV$-small $T$-limits preserving $\groU$-small $S$-limits. Moreover, as $\mathcal D$ is $P$-semiadditive, so is $\mathcal E$ by \cite{CLL_Global}*{Proposition~4.6.13}, and hence so is $\mathcal E^\op$ by Lemma~4.5.4 of \emph{op.~cit}. The dual $\mathcal D\to\mathcal E^\op$ of the Yoneda embedding therefore has the required properties.
		\end{proof}
	\end{lemma}

	As the lemma requires us to pass a larger universe, it is not clear a priori whether $\ul\CMon^P_T$ still has the correct universal property (we will see a posteriori that, as a matter of fact, it does). For locally small $\Dd$, one might try to avoid this issue by considering $\ul\Fun_T^{P\text-\times}(\mathcal D,\ul\Spc_T)$ instead or even by just closing up the Yoneda image under $\groU$-small $T$-limits in there, but even in this case it is not clear whether the result is still locally small---and said local smallness was crucial in the proof of Theorem~\ref{thm:Semiadd_omnibus} given in \cite{CLL_Global}, which relied on the Special Adjoint Functor Theorem. Accordingly, we will have to consider a $\groW$-version $\ul\CMON^P_T$ of $\ul\CMon^P_T$. The crucial technical lemma to relate these two to each other will be the following:

	\begin{lemma}
		The functor $\ul\Fun_T(\ul{\mathbb F}^P_{T,*},\ul\SPC_T)\to\ul\CMON^P_T\coloneqq\ul\Fun_T^{P\text-\oplus}(\ul{\mathbb F}^P_{T,*},\ul\SPC_T)$ left adjoint to the inclusion preserves $\ul\Fun_T(\ul{\mathbb F}^P_{T,*},\ul\Spc_T)$.
	\end{lemma}

	Accordingly, it restricts to a left adjoint $\ul\Fun_T(\ul{\mathbb F}^P_{T,*},\ul\Spc_T)\to\ul\CMon^P_T$ of the inclusion, and there is no harm in denoting both the localization functor in ordinary $T$-categories and in large $T$-categories by the same symbol $L^{P\text-\oplus}$.

	\begin{proof}
		Let $A\in T$ be arbitrary. \cite{CLL_Global}*{Remark~2.2.14} provides for any $\groW$-small category $\Ee$ a natural equivalence
		\begin{equation}\label{eq:Grothendieck}
			\ul\Fun_T(\ul{\mathbb F}^P_{T,*},\ul{\Ee}_T)(A)\simeq\Fun(\smallint\,\ul{\mathbb F}^P_{T,*}\times\ul A,\Ee)
		\end{equation}
		where $\smallint$ denotes the usual Grothendieck construction over $T^\op$ and $(\blank)_T$ denotes the $T$-category of $T$-objects (Example~\ref{ex:T-objects}). On the other hand, \cite{CLL_Global}*{Remark~4.9.9} characterizes the essential image of $\ul\CMon^P(\ul{\Ee}_T)(A)$ under this---it consists precisely of the functors $F^\dagger\colon \smallint\ul{\mathbb F}^P_{T,*}\times\ul{A}\to\Ee$ satisfying the following:
		\begin{enumerate}
			\item For every $f\colon B\to A$ in $T$ the restriction of $F^\dagger$ to the non-full subcategory $\ul{\mathbb F}^P_{T,*}(B)\simeq\ul{\mathbb F}^P_{T,*}(B)\times\{f\}\subset \int\ul{\mathbb F}^P_{T,*}\times\ul A$ is semiadditive in the usual sense.
			\item For every $p\colon B\to B'$ in $P$ and $f\colon B'\to A$ in $T$ a certain natural \emph{Segal map} $F^\dagger(B',B_+,f)\to F^\dagger(B,B_+,pf)$ is an equivalence; here we as usual denote objects in the Grothendieck construction by triples $(C\in T^\op,X\in \ul{\mathbb F}^P_{T,*}(C),g\in \maps(C,A))$.
		\end{enumerate}
		Specializing to $\Ee=\SPC$ and writing $y$ for the (non-parametrized) Yoneda embedding of $\smallint\ul{\mathbb F}^P_{T,*}\times\ul A$, we see that $F$ is $P$-semiadditive if and only if $F^\dagger$ is local with respect to the set $U^\dagger$ made up of suitable maps
		\begin{enumerate}
			\item $\emptyset\to y(B,*,f)$ for every map $f\colon B\to A$ in $T$ (so that the restriction to each $\ul{\mathbb F}^P_{T,*}(B)\times\{f\}$ is pointed)
			\item $y(B,X_+,f)\amalg y(B,Y_+,f)\to y(B,X_+\vee Y_+,f)$ for all $f\colon B\to A$ in $T$ and $X_+,Y_+\in\ul{\mathbb F}^P_{T,*}(B)$ (so that each restriction sends coproducts to products)
			\item $y(B,B_+,pf)\to y(B'
			,B_+,f)$ for every $p\colon A\to B$ in $P$ and $f\colon B'\to B$ in $T$ (ensuring that the Segal maps are equivalences).
		\end{enumerate}
		Transporting $U^\dagger$ along the equivalence $(\ref{eq:Grothendieck})$, we then get a set $U$ such that $F$ is $P$-semiadditive if and only if it is $U$-local. By direct inspection, each map in $U^\dagger$ actually lives in $\Fun(\int\ul{\mathbb F}^P_{T,*}\times\ul A,\Spc)$ (as opposed to functors into $\SPC$). By naturality of $(\ref{eq:Grothendieck})$ we can therefore also take the set $U$ to consist of maps in $\ul\Fun_T(\ul{\mathbb F}^P_{T,*},\ul\Spc_T)(A)$. We now write $U_1$ for the strongly saturated class generated by $U$ in $\ul\Fun_T(\ul{\mathbb F}^P_{T,*},\ul\Spc_T)(A)$ (with respect to $\groU$-small colimits) and $U_2$ for the strongly saturated class generated in $\ul\Fun_T(\ul{\mathbb F}^P_{T,*},\ul\SPC_T)(A)$ (with respect to $\groV$-small colimits). Clearly, $U_1\subset U_2$.

		By \cite{HTT}*{Proposition~5.4.5.15}, there exists for any $F\in\ul\Fun_T(\ul{\mathbb F}^P_{T,*},\ul\Spc_T)(A)$ a map $\eta\colon F\to F'$ into a $U$-local $F'$ such that $\eta\in U_1$. But then also $\eta\in U_2$, so it qualifies as the adjunction unit in the larger category by the same result, and in particular the image of $F$ under the localization functor to $\ul\CMON^P_T(A)$ is equivalent to $F'\in\ul\Fun_T(\ul{\mathbb F}^P_{T,*},\ul\Spc_T)(A)$ as desired.
	\end{proof}

	\begin{proof}[Proof of Proposition~\ref{prop:C-bar}]
		By the previous lemma, $\ul\CMon^P_T\subset\ul\CMON^P_T$ contains $\mathbb P(*)$. We claim that it is closed under $\groU$-small $T$-colimits: indeed, fiberwise colimits in $\ul\CMON^P_T$ are formed by first computing them in $\ul\Fun_T(\ul{\mathbb F}^P_{T,*},\ul\SPC_T)$ and then reflecting via $L^{P\text-\oplus}$, so $\ul\CMon^P_T$ is closed under $\groU$-small fiberwise colimits by the lemma, and similarly the functor $f_!\colon\ul\CMON^P_T(A)\to\ul\CMON^P_T(B)$ factors for any map $f\colon A\to B$ in $T$ as
		\begin{equation*}\hskip-4.87pt\hfuzz=4.88pt
			\ul\CMON^P_T(A)\hookrightarrow\ul\Fun_T(\ul{\mathbb F}^P_{T,*},\ul\SPC_T)(A)\xrightarrow{\,f_!\,}\ul\Fun_T(\ul{\mathbb F}^P_{T,*},\ul\SPC_T)(B)\xrightarrow{L^{P\text-\oplus}}\ul\CMON^P_T(B).
		\end{equation*}
		In particular, ${\mathcal C}\subset\ul\CMon^P_T$ is also closed under all $\groU$-small $S$-colimits in $\ul\CMON^P_T$, and thus under finite $P$-coproducts. As $\ul\CMON^P_T$ is $P$-semiadditive, ${\mathcal C}$ is then also closed under finite $P$-products and moreover $P$-semiadditive itself. By \cite{CLL_Global}*{Corollary~4.7.8$^\op$} there is then a unique functor $j\colon(\ul{\mathbb F}^P_{T,*})^\op\to{\mathcal C}$ that preserves finite $P$-products and sends $S^0$ to $\mathbb P(*)$. We moreover write $k$ for the inclusion ${\mathcal C}\hookrightarrow\ul\CMON^P_T$; then $k$ preserves $\groU$-small $S$-colimits, and hence in particular finite $P$-(co)products.

		If now $\mathcal D'$ is a $\groW$-small and locally $\groV$-small $P$-semiadditive $T$-category which admits all $\groV$-small $T$-colimits, then \cite{martiniwolf2021limits}*{Theorem~6.3.5 and Corollary~6.3.7} show that the left Kan extension functors $j_!\colon\ul\Fun_T((\ul{\mathbb F}^P_{T,*})^\op,\mathcal D')\to\ul\Fun_T({\mathcal C},\mathcal D')$ and $k_!\colon\ul\Fun_T({\mathcal C},\mathcal D')\to\ul\Fun_T(\ul\CMON^P_T,\mathcal D')$ exist and that the latter is fully faithful, while \cite{CLL_Global}*{Proposition~4.8.12} shows that $k_!j_!$ restricts to an equivalence $\ul\Fun_T^{P\text-\times}(\ul{\mathbb F}^P_{T,*},\mathcal D')\simeq\ul\Fun_T^{T\text{-CC}}(\ul\CMON^P_T,\mathcal D')$, where the right hand side denotes functors preserving all $\groV$-small $T$-colimits.

		We claim that $j_!$ restricts to an equivalence $\ul\Fun_T^{P\text-\times}((\ul{\mathbb F}^P_{T,*})^\op,\mathcal D')\simeq\ul\Fun_T^{S\text{-cc}}({\mathcal C},\mathcal D')$, for which it is enough by $2$-out-of-$3$ that for any $A\in T$ and any finite $P$-product preserving $f\colon(\ul{\mathbb F}^P_{T,*})^\op\to\ul\Fun_T(\ul A,\mathcal D')$ the Kan extension $j_!f$ preserves $\groU$-small $S$-colimits, and that conversely any $S$-cocontinuous functor arises this way.

		For the first statement, it is enough to observe that $k_!j_!f\colon\ul\CMON^P_T\to\ul\Fun_T(A,\Dd')$ is in particular $S$-cocontinuous, whence so is $k^*k_!j_!f\simeq j_!f$ as $k$ is $S$-cocontinuous. Conversely, if $F\colon{\mathcal C}\to\ul \Fun_T(\ul A,\mathcal D')$ is $S$-cocontinuous, then its restriction to $(\ul{\mathbb F}^P_{T,*})^\op$ preserves finite $P$-products. Consider now the subcategory of $\mathcal C$ of all objects for which the counit $\epsilon\colon j_!j^*F\to F$ is an equivalence. Then this is closed under $\groU$-small $S$-colimits as both sides are $S$-cocontinuous, and it moreover contains $\mathbb P(*)$ as the unit $j^*F\to j^*j_!j^*F$ is an equivalence by full faithfulness of $j_!$. The claim then follows as ${\mathcal C}$ is generated by $\mathbb P(*)$ under $\groU$-small $S$-colimits by construction.

		Let now $\mathcal D$ be a $P$-semiadditive $S$-cocomplete $T$-category, and use Lemma~\ref{lemma:semiadd-T-cocompl} to obtain an $S$-cocontinuous embedding into a large $\mathcal D'$ as above. Then the Kan extension $j_!\colon\ul\Fun^{P\text-\times}_T((\ul{\mathbb F}^P_{T,*})^\op,\mathcal D')\to\ul\Fun_T^{T\text{-CC}}({\mathcal C},\mathcal D')$ restricts to $\ul\Fun^{P\text-\times}_T((\ul{\mathbb F}^P_{T,*})^\op,\mathcal D)\to\ul\Fun_T({\mathcal C},\mathcal D)$ as $\mathcal D\subset\mathcal D'$ is closed under $\groU$-small $S$-colimits and ${\mathcal C}$ is generated under them by $\mathbb P(*)$. In particular, $j^*\colon\ul\Fun^{S\text{-cc}}_T({\mathcal C},\mathcal D)\to\ul\Fun_T^{P\text-\times}((\ul{\mathbb F}^P_{T,*})^\op,\mathcal D)$ is an equivalence. The proposition follows as the right hand side is further equivalent to $\mathcal D$ via evaluation at $S^0$ by \cite{CLL_Global}*{Corollary~4.7.8$^\op$}.
	\end{proof}

	\begin{remark}
		Running the same argument in an even larger universe $\mathfrak X$, the above proof (without the penultimate paragraph) shows that $\ul\Fun^{T\text{-cc}}_{T}(\ul\CMon^P_T,\mathcal D)\simeq\mathcal D$ via evaluation at $\mathbb P(*)$ for any $\groW$-small $P$-semiadditive $T$-category $\Dd$ with $\groU$-small $S$-colimits.
	\end{remark}

	As an upshot, we can now stop thinking about universes.

	\subsubsection{Relation to $\ul\CMon^P_S$} Next, we want to understand the underlying $S$-category of the universal $P$-semiadditive $S$-cocomplete $T$-category $\mathcal C$ constructed above. As in the unstable situation this will be some formal Yoneda yoga.

	\begin{proposition}\label{prop:semiadd-cat-adj}
		The adjunction $\iota^*\colon\Cat_T\rightleftarrows\Cat_S\noloc\iota_*$ restricts to give adjunctions $\Cat_T^{P\text-\oplus}\rightleftarrows\Cat_S^{P\textup{-}\oplus}$, $\Cat_T^{P\text-\oplus,S\textup{-cc}}\rightleftarrows\Cat_S^{P\textup-\oplus, S\textup{-cc}}$, and $\Pr_T^{S,P\text-\oplus}\rightleftarrows\Pr_S^{S,P\text-\oplus}$.
		\begin{proof}
			If will suffice to prove the first statement; the second one will then follow from Corollary~\ref{cor:S-cocomplete-adjunction}.

			By Lemma~\ref{lemma:iota!-pullback}, $\iota_!\colon\PSh(S)\to\PSh(T)$ preserves pullbacks, while its right adjoint restricts to $\ul{\mathbb F}^P_T\to\iota_*\ul{\mathbb F}^P_S$ by Lemma~\ref{lemma:V-restr} for $\bbV=\ul{\mathbb F}^P_T$. \cite{CLL_Global}*{Lemma~4.6.5} therefore shows that both $\iota^*$ and $\iota_*$ restrict accordingly. Moreover, Theorem~\ref{thm:iota*-cc} shows that the unit and counit are $P$-cocontinuous and in particular $P$-semiadditive.
		\end{proof}
	\end{proposition}

	In fact, the above argument also shows slightly more generally:

	\begin{proposition}\label{prop:semiadd-cat-adj-general}
		Let $\mathcal C$ be a pointed $T$-category with finite $P$-coproducts and let $\mathcal D$ be an $S$-category with finite $P$-products. Then a $T$-functor $\mathcal C\to\iota_*\mathcal D$ is $P$-semiadditive if and only if its adjunct $\iota^*\mathcal C\to\mathcal D$ is so.\qed
	\end{proposition}

	Arguing as in Corollary~\ref{cor:iota**-cc-internal} we immediately deduce:

	\begin{corollary}\label{cor:Phi-semiadd}
		Let $\Cc$ be a pointed $T$-category with finite $P$-coproducts and let $\Dd$ be an $S$-category with finite $P$-products. Then the equivalence $\Phi\colon\ul\Fun_T(\Cc,\iota_*\Dd)\iso\iota_*\ul\Fun_S(\iota^*\Cc,\Dd)$ from Construction~\ref{constr:no-pun-intended} restricts to an equivalence
		\begin{equation*}
			\ul\Fun_T^{P\text-\oplus}(\Cc,\iota_*\Dd)\iso\iota_*\ul\Fun_S^{P\text-\oplus}(\iota^*\Cc,\Dd).\qednow
		\end{equation*}
	\end{corollary}

	On the other hand, we now easily get the following result subsuming Theorem~\ref{thm:CMon-ST-univ-prop} and one half of \ref{thm:CMon-ST-adj}:

	\begin{theorem}\label{thm:CMon-technical}
		\begin{enumerate}
			\item There is a unique $S$-cocontinuous functor $\iota_!\colon\ul\CMon^P_S\to\iota^*\ul\CMon^P_T$ sending $\mathbb P(*)$ to $\mathbb P(*)$.
			\item $\iota_!$ is fully faithful and extends uniquely to a $T$-functor $\ul\CMon^P_{S\triangleright T}\to \ul\CMon^P_T$.
			\item $\ul\CMon^P_{S\triangleright T}$ is $S$-presentable and $P$-semiadditive.
			\item Let $\mathbb P(*)\in\Gamma(\ul\CMon^P_{S\triangleright T})$ denote the preimage of the object of $\Gamma(\ul\CMon^P_T)$ of the same name. Then $\ul\CMon^P_{S\triangleright T}$ has the following universal property: for any $P$-semiadditive $S$-cocomplete $T$-category $\Dd$ evaluation at $\mathbb P(*)$ defines an equivalence $\ul\Fun_T(\ul{\CMon}^P_{S\triangleright T},\Dd)\simeq\Dd$.
		\end{enumerate}
		\begin{proof}
			Let $\mathcal C\subset\ul\CMon^P_T$ again be generated under $S$-colimits by $\mathbb P(*)$. Arguing as in the proof of Theorem~\ref{thm:mixed-presheaves}, Proposition~\ref{prop:semiadd-cat-adj} together with Proposition~\ref{prop:C-bar} shows that there is a unique $S$-cocontinuous functor $\ul\CMon^P_S\to\iota^*\mathcal C$ preserving $\mathbb P(*)$, and that this is an equivalence. Thus, $\iota_!\colon\ul\CMon^P_S\to\iota^*\ul\CMon^P_T$ extends uniquely to a fully faithful  $T$-functor $\ul\CMon^P_{S\triangleright T}\to\ul\CMon^P_T$, and this induces an equivalence onto $\mathcal C$. The universal property then follows by another application of Proposition~\ref{prop:C-bar}.

			It only remains to show that the category $\mathcal C$ (and hence also $\ul\CMon^P_{S\triangleright T})$ is $S$-presentable. But indeed, $\Cc$ is $S$-cocomplete as $\ul\CMon^P_T$ is so, and $\mathcal C(A)\simeq\ul\CMon^P_S(A)$ is presentable for any $A\in T$.
		\end{proof}
	\end{theorem}

	\subsection{An additional adjoint}\label{subsec:Add_Adj}
	Our goal in this subsection will be to understand the right adjoint $\iota^*$ of the above $S$-functor $\iota_!\colon\ul\CMon^P_S\to\iota^*\ul\CMon^P_T$ better, and to use this to show that it in turn admits another right adjoint $\iota_*$, finishing the proof of Theorem~\ref{thm:CMon-ST-adj}. We begin with the following observation:

	\begin{lemma}\label{lemma:iota*-over-iota*}
		The diagram
		\begin{equation*}
			\begin{tikzcd}
				\iota^*\ul\CMon^P_T\arrow[d, "\iota^*\mathbb U"']\arrow[r, "\iota^*"] & \ul\CMon^P_S\arrow[d, "\mathbb U"]\\
				\iota^*\ul\Spc_T\arrow[r, "\iota^*"'] & \ul\Spc_S\rlap.
			\end{tikzcd}
		\end{equation*}
		commutes up to natural equivalence.
	\end{lemma}

	Note that $\iota^*$ is $P$-semiadditive as it is right adjoint; by the universal property of $\ul\CMon^P_S$ from Theorem~\ref{thm:Semiadd_omnibus} the above then actually characterizes $\iota^*$ completely.

	\begin{proof}
		All functors in the diagram
		\begin{equation}\label{diag:iota*/iota*-adj}
			\begin{tikzcd}
				\ul\Spc_S\arrow[r, "\iota_!"]\arrow[d, "\mathbb P"'] & \iota^*\ul\Spc_T\arrow[d, "\iota^*\mathbb P"]\\
				\ul\CMon^P_S\arrow[r, "\iota_!"'] & \iota^*\ul\CMon^P_T
			\end{tikzcd}
		\end{equation}
		are left adjoints, and both paths through this diagram send the terminal object to the same object by the defining property of the horizontal maps. Thus, the universal property of $\ul\Spc_S$ shows that $(\ref{diag:iota*/iota*-adj})$ commutes up to equivalence. The claim follows by passing to total mates.
	\end{proof}

	This suggests a natural strategy to get a more explicit description of $\iota^*$ and to prove that it has a right adjoint: construct \emph{some} left adjoint $\iota^*\ul\CMon^P_T\to\ul\CMon^P_S$ and then show that it is compatible with the forgetful functors.	Indeed, this is precisely what we will do now, using the restriction functor from Construction~\ref{constr:no-pun-intended}:

	\begin{proposition}\label{prop:iota*-semiadd-RA}
		The composite
		\begin{equation}\label{eq:composite-iota*}
			\iota^*\ul\Fun_T(\ul{\mathbb F}^P_{T,*},\ul\Spc_T)\xrightarrow{\;\iota^*\;}\ul\Fun_S(\ul{\mathbb F}^P_{S,*},\iota^*\ul\Spc_T)\xrightarrow{\smash{\ul\Fun_S(\ul{\mathbb F}^P_{S,*},\iota^*)}\rule[-2pt]{0pt}{0pt}}\ul\Fun_S(\ul{\mathbb F}^P_{S,*},\ul\Spc_S)
		\end{equation}
		restricts to the functor $\iota^*\ul\CMon^P_T\to\ul\CMon^P_S$ right adjoint to $\iota_!$. Moreover, $(\ref{eq:composite-iota*})$ admits a right adjoint $\iota_*\colon\ul\Fun_S(\ul{\mathbb F}^P_{S,*},\ul\Spc_S)\to\iota^*\ul\Fun_T(\ul{\mathbb F}^P_{T,*},\ul\Spc_T)$, which again restricts to $\ul\CMon^P_S\to\iota^*\ul\CMon^P_T$.
		\begin{proof}
			For $(\ref{eq:composite-iota*})$ to restrict as claimed, it will be enough to show that its adjunct $\ul\Fun_T(\ul{\mathbb F}^P_{T,*},\ul\Spc_T)\to\iota_*\ul\Fun_S(\ul{\mathbb F}^P_{S,*},\ul\Spc_S)$ restricts to $\ul\CMon^P_T\to\iota_*\ul\CMon^P_S$. However, unravelling the definitions, the adjunct is precisely given by
			\begin{equation}\label{eq:iota*-adj-semiadd}
				\Fun_T(\ul{\mathbb F}^P_{T,*},\ul\Spc_T)\xrightarrow{\ul\Fun_T(\ul{\mathbb F}^P_{T,*},\tilde\iota^*)}\ul\Fun_T(\ul{\mathbb F}^P_{T,*},\iota_*\ul\Spc_S)\mathrel{\smash{\xrightarrow[\raise11pt\hbox{$\scriptstyle\sim$}]{\,\;\Phi\;\,}}}\iota_*\ul\Fun_S(\ul{\mathbb F}^P_{S,*},\ul\Spc_S),
			\end{equation}
			where $\tilde\iota^*$ is the adjunct of $\iota^*$, as in \Cref{prop:right-adjoints-adjunct}. The first functor restricts to semiadditive functors as $\tilde\iota^*$ is $S$-continuous by Corollary~\ref{cor:S-cocomplete-adjunction}$^\op$, and so does the second functor by Corollary~\ref{cor:Phi-semiadd}.

			Proposition~\ref{prop:right-adjoints-adjunct} then shows that $(\ref{eq:iota*-adj-semiadd})$ has a right adjoint $\tilde\iota_*$ given by the composite
			\begin{equation*}
				\iota_*\ul\Fun_S(\ul{\mathbb F}^P_{S,*},\ul\Spc_S)\mathrel{\smash{\xrightarrow[\raise11pt\hbox{$\scriptstyle\sim$}]{\,\Phi^{-1}\,}}}\ul\Fun_T(\ul{\mathbb F}^P_{T,*},\iota_*\ul\Spc_S)\xrightarrow{\ul\Fun_T(\ul{\mathbb F}^P_{T,*},\tilde\iota_*)}\Fun_T(\ul{\mathbb F}^P_{T,*},\ul\Spc_T)
			\end{equation*}
			which restricts to $\iota_*\ul\CMon^P_S\to\ul\CMon^P_T$ by the same argument as before.

			We can now show that $(\ref{eq:composite-iota*})$ has a right adjoint $\iota_*$: namely, as it is adjunct to $(\ref{eq:iota*-adj-semiadd})$, it factors as
			\begin{equation*}
				\iota^*\ul\Fun_T(\ul{\mathbb F}^P_{T,*},\ul\Spc_T)\xrightarrow{\iota^*(\ref{eq:iota*-adj-semiadd})}\iota^*\iota_*\ul\Fun_S(\ul{\mathbb F}^P_{S,*},\ul\Spc_S)\xrightarrow{\,\;\epsilon\;\,}\ul\Fun_S(\ul{\mathbb F}^P_{S,*},\ul\Spc_S),
			\end{equation*}
			and the first map has a right adjoint given by $\iota^*(\tilde\iota_*)$ as $\iota^*$ obviously preserves adjunctions, while the second one has a right adjoint as it $S$-cocontinuous by (the proof of) Corollary~\ref{cor:S-cocomplete-adjunction}.

			Next, let us show that $\iota_*$ restricts to $\ul\CMon^P_S\to\iota^*\ul\CMon_T^P$. By the above, $\iota^*(\tilde\iota_*)$ restricts to $\iota^*\iota_*\ul\CMon^P_S\to\iota^*\ul\CMon^P_T$, so it only remains to show that the right adjoint of the counit $\epsilon\colon \iota^*\iota_* \to \id$ at $\ul\Fun_S(\ul{\mathbb F}^P_{S,*},\ul\Spc_S)$ restricts to $\ul\CMon^P_S\to\iota^*\iota_*\ul\CMon^P_S$. But $\epsilon$ is simply given by restricting along the unit of $\iota_!\colon\PSh(S)\rightleftarrows\PSh(T)\noloc\iota^*$, so the claim follows as $\ul\CMon^P_S\subset\ul\Fun_S(\ul{\mathbb F}^P_{S,*},\ul\Spc_S)$ is closed under all $S$-limits.

			It now only remains to show that the restriction of $(\ref{eq:composite-iota*})$ to $\iota^*\ul\CMon^P_T\to\ul\CMon^P_S$ indeed agrees with the functor $\iota^*$ considered before. For this we consider the diagram
			\begin{equation*}
				\begin{tikzcd}
					\iota^*\ul\Fun_T(\ul{\mathbb F}^P_{T,*},\ul\Spc_T)\arrow[r, "\iota^*"]\arrow[d, "\iota^*\ev_{S^0}"'] & \ul\Fun_S(\ul{\mathbb F}^P_{S,*},\iota^*\ul\Spc_T)\arrow[r, "{\ul\Fun_S(\ul{\mathbb F}^P_{S,*},\iota^*)}"]\arrow[d, "\ev_{S^0}"] &[4em] \ul\Fun_S(\ul{\mathbb F}^P_{S,*},\ul\Spc_S)\arrow[d, "\ev_{S^0}"]\\
					\iota^*\ul\Spc_T\arrow[r, equal] & \iota^*\ul\Spc_T\arrow[r, "\iota^*"'] & \ul\Spc_S
				\end{tikzcd}
			\end{equation*}
			with top row $(\ref{eq:composite-iota*})$. The right hand square commutes by naturality, as does the left hand square by a straightforward mate argument. Thus, the restricted functor $\iota^*\ul\CMon^P_T\to\ul\CMon^P_S$ lies over $\iota^*\colon\iota^*\ul\Spc_T\to\ul\Spc_S$. But it is also $P$-semiadditive (it is even an $S$-left adjoint by the above), so the claim follows from Lemma~\ref{lemma:iota*-over-iota*} and the universal property of $\ul\CMon^P_S$.
		\end{proof}
	\end{proposition}

	\begin{proof}[Proof of Theorem~\ref{thm:CMon-ST-adj}]
		Combine Theorem~\ref{thm:CMon-technical} with the previous proposition.
	\end{proof}

	\section{The universal property of equivariant special $\Gamma$-spaces}\label{sec:gamma-orbi}

	In this section we will identify the universal equivariantly semiadditive equivariantly presentable global category in terms of Shimakawa's \emph{special $\Gamma$-$G$-spaces} \cite{shimakawa,shimakawa-simplify}.

	\subsection{Model categories of equivariant $\bm\Gamma$-spaces.}\label{subsec:eq-gamma-spaces}
	We begin by introducing the main players, which will require a bit more model categorical sophistication than the unstable case.

	\begin{definition}
		We write $\Gamma$ for the category of finite pointed sets and based maps. For any $n\ge0$ we let $n^+\coloneqq\{0,\dots,n\}$ with basepoint $0$.
	\end{definition}

	\begin{definition}
		Let $G$ be a finite group. A \emph{$\Gamma$-$G$-space} is a functor $\Gamma\to\cat{$\bm G$-SSet}$ that sends the singleton set $0^+$ to the $1$-point space. We write $\cat{$\bm\Gamma$-$\bm G$-SSet}_*$ for the category of $\Gamma$-$G$-spaces.
	\end{definition}

	By \cite[Lemma~1.17]{mmo} we can equivalently think of a $\Gamma$-$G$-space as an $\cat{Set}_*$-enriched functor $\Gamma\to\cat{$\bm G$-SSet}_*$ into the category of \emph{pointed} $G$-spaces, with the equivalence given by forgetting the basepoints and the enrichment.

	\subsubsection{Level model structures} Next, we will equip $\cat{$\bm\Gamma$-$\bm G$-SSet}_*$ with a suitable \emph{level model structure}. To put this into context, we recall the standard equivariant model structures on $\cat{$\bm G$-SSet}$:

	\begin{proposition}
		Let $G$ be a finite group and let $\mathcal F$ be a \emph{family of subgroups} of $G$, i.e.~a non-empty collection of subgroups that is closed under taking subconjugates. Then $\cat{$\bm G$-SSet}$ carries a model structure with
		\begin{enumerate}
			\item weak equivalences the \emph{$\mathcal F$-weak equivalences}, i.e.~those maps $f$ such that $f^H$ is a weak homotopy equivalence for every $H\in\mathcal F$;
			\item fibrations the \emph{$\mathcal F$-fibrations}, i.e.~those maps $f$ such that $f^H$ is a Kan fibration for every $H\in\mathcal F$;
			\item cofibrations the \emph{$\mathcal F$-cofibrations}: those maps $f$ that are levelwise injective and such that the isotropy of any simplex outside the image of $f$ belongs to $\mathcal F$.
		\end{enumerate}
		We call this the $\mathcal F$-model structure. It is combinatorial, simplicial, and proper.
		\begin{proof}
			See e.g.~\cite[Proposition~1.1.2]{g-global} and \cite[Proposition~2.16]{cellular}.
		\end{proof}
	\end{proposition}

	\begin{example}
		In the special case that $\mathcal F=\mathcal A\ell\ell$ consists of all subgroups, we call this the \emph{$G$-equivariant model structure}; its weak equivalences are the \emph{$G$-equivariant weak equivalences} considered before.
	\end{example}

	\begin{proposition}\label{prop:equiv-level-model-struct}
		The category $\cat{$\bm\Gamma$-$\bm G$-SSet}_*$ admits a unique model structure with
		\begin{enumerate}
			\item weak equivalences those maps $f\colon X\to Y$ such that $f(S_+)\colon X(S_+)\to Y(S_+)$ is a $G$-weak equivalence for any finite \emph{$G$-set} $S$; here we equip both sides with the \emph{diagonal} $G$-action induced from the actions on $X$, $Y$, and $S$;
			\item fibrations those $f$ such that $f(S_+)$ is a $G$-fibration for any finite $G$-set $S$.
		\end{enumerate}
		We call this the \emph{$G$-equivariant level model structure} and its weak equivalences the \emph{$G$-equivariant level weak equivalences}. It is simplicial, proper, and combinatorial with generating cofibrations the maps
		\begin{equation*}
			\Gamma(S_+,\blank)\smashp G/H_+\smashp(\partial\Delta^n\hookrightarrow\Delta^n)_+
		\end{equation*}
		for all $n\ge0$, all $H\subset G$, and all finite $G$-sets $S$, while its generating acyclic cofibrations are similarly given by the maps $\Gamma(S_+,\blank)\smashp G/H_+\smashp(\Lambda_k^n\hookrightarrow\Delta^n)_+$.
	\end{proposition}

	\begin{remark}
		By \cite[Remark~4.11]{ostermayr}, we could equivalently ask for $f(S_+)$ to be an $H$-weak equivalence or $H$-fibration for any $H\subset G$ and any finite $H$-set $S$. Put differently, if $\mathcal G_{G,\Sigma_S}$ denotes the family of \emph{graph subgroups} of $G\times\Sigma_S$ (i.e.~the subgroups of the form $\Gamma_{H,\phi}\coloneqq\{(h,\phi(h)):h\in H\}$ for $H\subset G$ and $\phi\colon H\to\Sigma_S$, or equivalently those subgroups intersecting $1\times\Sigma_S$ trivially), then a map $f$ is a weak equivalence or fibration in the above model structure if and only if $f(S_+)$ is a $\mathcal G_{G,\Sigma_S}$-weak equivalence or fibration, respectively, for any finite set $S$.
	\end{remark}

	\begin{proof}[Proof of Proposition~\ref{prop:equiv-level-model-struct}]
		The model structure appears without proof as \cite[Theorem~4.7]{ostermayr}; see \cite[Proposition~2.2.36]{g-global} for a complete argument.
	\end{proof}

	\begin{lemma}\label{lemma:gamma-level-adj}
		Let $\alpha\colon G\to G'$ be a homomorphism of finite groups. Then the restriction $\alpha^*\colon\cat{$\bm\Gamma$-$\bm{G'}$-SSet}_*\to\cat{$\bm\Gamma$-$\bm G$-SSet}_*$ is left Quillen for the level model structures.
		\begin{proof}
			It suffices that the right adjoint $\alpha_*$ preserves (acyclic) fibrations. As the latter are defined levelwise, this amounts to saying that
			\begin{equation*}
				(\alpha\times\Sigma_S)^*\colon\cat{$\bm{(G'\times\Sigma_S)}$-SSet}_{\mathcal G_{G',\Sigma_S}}\rightleftarrows\cat{$\bm{(G\times\Sigma_S)}$-SSet}_{\mathcal G_{G,\Sigma_S}}\noloc(\alpha\times\Sigma_S)_*
			\end{equation*}
			is a Quillen adjunction for every finite set $S$. But clearly $(\alpha\times\Sigma_S)^*$ preserves cofibrations and sends generating acyclic cofibrations to weak equivalences.
		\end{proof}
	\end{lemma}

	\begin{remark}\label{rk:alpha-RQ-equiv-level}
		If $\alpha\colon G\to G'$ is injective, then $\alpha^*\colon\cat{$\bm\Gamma$-$\bm{G'}$-SSet}_*\to\cat{$\bm\Gamma$-$\bm{G}$-SSet}$ is easily seen to preserve weak equivalences and fibrations; in particular, it is also right Quillen.
	\end{remark}

	Beware that the previous remark does \emph{not} hold for non-injective $\alpha$, see e.g.~\cite[Example~2.2.15]{g-global}, and accordingly the composition
	\begin{equation*}
		\cat{$\bm\Gamma$-$\bm{\bullet}$-SSet}_*\colon\Glo^\op\hookrightarrow\Grpd^\op\xrightarrow{\Fun(\blank,\cat{$\bm\Gamma$-SSet}_*)}\Cat_1
	\end{equation*}
	does not lift to $\RelCat$ via the above weak equivalences. However, by Ken Brown's Lemma we can fix this by restricting to the subcategories of \emph{cofibrant} objects:

	\begin{definition}
		We write $\cat{$\bm\Gamma$-$\bm\bullet$-SSet}_*^\text{cof}$ for the resulting functor $\Glo^\op\to\RelCat$ and $\ul{\GammaS}_*$ for the global category obtained by pointwise localization.
	\end{definition}

	We can (and will) equivalently think of $\ul{\GammaS}_*$ as sending a finite group $G$ to the localization of $\cat{$\bm\Gamma$-$\bm G$-SSet}_*$ and a homomorphism $\alpha$ to the left derived functor $\cat{L}\alpha^*$.

	\begin{lemma}
		Let $G$ be a finite group. Then we have a Quillen adjunction
		\begin{equation*}
			\Gamma(1^+,\blank)\smashp(\blank)_+\colon\cat{$\bm G$-SSet}\rightleftarrows\cat{$\bm\Gamma$-$\bm G$-SSet}_*\noloc\ev_{1^+}
		\end{equation*}
		in which both adjoints are homotopical.
		\begin{proof}
			It is clear that both adjoints are homotopical, and that the right adjoint moreover preserves fibrations, so that it is in particular right Quillen.
		\end{proof}
	\end{lemma}

	Thus, $\Gamma(1^+,\blank)\smashp(\blank)_+$ induces a natural transformation $\cat{$\bullet\bm$-SSet}\to\cat{$\bm\Gamma$-$\bm{\bullet}$-SSet}_*^\text{cof}$, and hence a global functor $\ul\S\to\ul{\GammaS}_*$, which we denote by ${\mathbb P}$. It is not hard to check that ${\mathbb P}$ admits a global right adjoint (induced by $\ev_{1^+}$); as we will not need this below, we leave the details to the interested reader.

	\subsubsection{Specialness} In order to study equivariant commutative monoids, we have to Bousfield localize the above level model structures. For this we recall:

	\begin{definition}
		A $\Gamma$-$G$-space is called \emph{special} if for every finite $G$-set $S$ the \emph{Segal map} $X(S_+)\to X(1^+)^{\times S}$ induced by the characteristic maps $\chi_s\colon S_+\to 1^+$ for varying $s\in S$, is a $G$-weak equivalence.
	\end{definition}

	Similarly to the different characterizations of the $G$-equivariant level weak equivalences, specialness is equivalent to asking more generally for the Segal maps to be $H$-equivariant weak equivalences for all $H\subset G$ and all finite $H$-sets $S$, or for them to be $\mathcal G_{G,\Sigma_S}$-weak equivalences for every finite set $S$, see \cite[Lemma~2.2.10]{g-global}.

	\begin{proposition}[See {\cite[Proposition~2.2.60]{g-global}}]
		The $G$-equivariant level model structure on $\cat{$\bm\Gamma$-$\bm G$-SSet}_*$ admits a Bousfield localization with fibrant objects precisely the level fibrant \emph{special} $\Gamma$-$G$-spaces. We call this the \emph{$G$-equivariant model structure} and its weak equivalences the \emph{$G$-equivariant weak equivalences}. It is combinatorial, simplicial, and left proper.\qed
	\end{proposition}

	\begin{remark}\label{rk:special-generating-we}
		The above model structure is obtained from the level model structure by localizing with respect to the maps $S_+\smashp \Gamma(1^+,\blank)\smashp G/H_+\to\Gamma(S_+,\blank)\smashp G/H_+$ induced by the map $S_+\to\Gamma(S_+,1^+)$ sending $s\in S$ to its characteristic map $\chi_S\colon S_+\to 1^+$ for all finite $G$-sets $S$. In particular, all of these maps are $G$-equivariant weak equivalences.
	\end{remark}

	\begin{lemma}\label{lemma:gamma-spec-adj}
		Let $\alpha\colon G\to G'$ be a homomorphism. Then
		\begin{equation*}
			\alpha^*\colon\cat{$\bm\Gamma$-$\bm{G'}$-SSet}_*\rightleftarrows\cat{$\bm\Gamma$-$\bm{G}$-SSet}_*
		\end{equation*}
		is left Quillen with respect to the above model structures. If $\alpha$ is injective, then $\alpha^*$ is also right Quillen.
		\begin{proof}
			For the first statement, it will suffice by \cite[Corollary~A.3.7.2]{HTT} that $\alpha^*$ preserves cofibrations and $\alpha_*$ preserves fibrant objects. The first statement is clear from Lemma~\ref{lemma:gamma-level-adj}, while for the second statement it is enough by adjunction to show that $\cat{L}\alpha^*$ sends the maps from the previous remark to weak equivalences. As these are maps between cofibrant objects, it is enough to prove the same for $\alpha^*$. However, decomposing $\alpha^*(G'/H)$ into $G$-orbits expresses $\alpha^*(S_+\smashp\Gamma(1^+,\blank)\smashp G'/H_+\to\Gamma(S_+,\blank)\smashp G'/H_+)$ as a coproduct of weak equivalences between cofibrant objects, so the claim follows.

			The second statement follows similarly from Remark~\ref{rk:alpha-RQ-equiv-level} as $\alpha^*$ clearly preserves specialness for injective $\alpha$.
		\end{proof}
	\end{lemma}

	In particular, we get a functor $\cat{$\bm\Gamma$-$\bm\bullet$-SSet}_*^\text{cof, spc}\colon\Glo^\op\to\RelCat$ that sends $G$ to $\cat{$\bm\Gamma$-$\bm G$-SSet}_*^\text{cof}$ with the above weak equivalences. The identity of underlying categories $\cat{$\bm\Gamma$-$\bm\bullet$-SSet}_*^\text{cof}\to\cat{$\bm\Gamma$-$\bm\bullet$-SSet}_*^\text{cof, spc}$ then induces a localization $L\colon\ul{\GammaS}_*\to\ul{\GammaS}_*^{\text{spc}}$. We will write ${\mathbb P}\colon\ul\S\to\ul\GammaS_*^{\text{spc}}$ for $L\circ{\mathbb P}$; note that this is again induced by the homotopical left Quillen functors $\Gamma(1^+,\blank)\smashp(\blank)_+$.

	\begin{warning}\label{warning:L_not_global_left_adj}
		The functors $\cat{L}\alpha^*$ do \emph{not} preserve specialness for non-injective $\alpha$, i.e.~the pointwise right adjoints of $L$ do not assemble into a global right adjoint. This is hard to see directly (as we know so few cofibrant objects in the above model structure, making it hard to compute $\cat L\alpha^*$), so we use a trick and a bit of equivariant infinite loop space theory instead:

		Let $\Gamma(1^+,\blank)\to\mathcal S$ be an acyclic cofibration to a special $\Gamma$-space. In particular, $\mathcal S$ is cofibrant, so if $\cat{L}\alpha^*$ preserved specialness, then $\mathcal S$ with the trivial $G$-action would be a special $\Gamma$-$G$-space for any finite $G$. On the other hand, as restrictions are left Quillen by the above, it would be equivalent to $\Gamma(1^+,\blank)$ with trivial $G$-action. We show that already for $G=\mathbb Z/2$ this is impossible: no special $\Gamma$-$\mathbb Z/2$-space equivalent to $\Gamma(1^+,\blank)$ can have trivial action.

		For this we use that the delooping of $\Gamma(1^+,\blank)$ (and hence of any $\Gamma$-$\mathbb Z/2$-space equivalent to it) is the equivariant sphere spectrum. Now the zeroth stable homotopy groups of the latter are given by the Burnside ring, and hence in particular $\pi_0^1(\mathbb S)\cong\mathbb Z\not\cong \mathbb Z^2\cong\cramped{\pi_0^{\smash{\mathbb Z/2}}}(\mathbb S)$. However, for a \emph{special} $\Gamma$-$\mathbb Z/2$-space the homotopy groups of its delooping are simply given as the group completions of the original homotopy monoids. In the case of a trivial $G$-action, the restriction homomorphisms between these homotopy monoids are clearly isomorphisms, and in particular their group completions are isomorphic, yielding the desired contradiction.

		Note that the same argument shows that also the underived functors $\alpha^*$ do not preserve specialness, although there are much more concrete counterexamples available in this case.
	\end{warning}

	We can now finally state the main results of this section.

	\begin{theorem}\label{thm:gamma-comparison}
		The global category $\ul{\GammaS}_*^\textup{spc}$ is equivariantly presentable and equivariantly semiadditive. Moreover, the unique equivariantly cocontinuous global functor $\ul\CMon^{\Orb}_{\Orb\triangleright\Glo}\to\ul{\GammaS}_*^\textup{spc}$ sending ${\mathbb P}(*)$ to ${\mathbb P}(*) $ is an equivalence.
	\end{theorem}

	\begin{theorem}\label{thm:gamma-univ-prop}
		The global category $\ul{\GammaS}_*^\textup{spc}$ is the free equivariantly cocomplete equivariantly semiadditive global category on one generator in the following sense: for every other such $\mathcal D$ evaluation at ${\mathbb P}(*)$ provides an equivalence
		\begin{equation*}
			\ul\Fun^{\textup{eq-cc}}_{\Glo}(\ul{\GammaS}_*^\textup{spc},\mathcal D)\iso\mathcal D.
		\end{equation*}
	\end{theorem}

	The proofs for these two theorems will be given in Subsection~\ref{subsec:Gamma-semiadd-proofs} after some further preparations. Similarly to the unstable situation, the key idea will be to deduce these from a statement about the universal \emph{globally cocomplete} equivariantly semiadditive category via the special case $T=\Glo,P=S=\Orb$ of Theorem~\ref{thm:CMon-ST-adj}.

	\subsection{$\bm G$-global vs.~$\bm G$-equivariant $\bm\Gamma$-spaces} In this subsection we will recall the pointset model of the free globally cocomplete equivariantly semiadditive category from \cite{CLL_Global} and compare it to $\ul{\GammaS}_*^\textup{spc}$.

	\subsubsection{Model categories of $G$-global $\Gamma$-spaces} We begin by introducing the relevant model categories.

	\begin{definition}
		A \emph{$G$-global $\Gamma$-space} is a functor $X\colon\Gamma\to\cat{$\bm G$-$\bm{\mathcal I}$-SSet}$ such that $X(0^+)=*$.
	\end{definition}

	\begin{proposition}[See {\cite[Theorem~2.2.31]{g-global}}]
		The category $\cat{$\bm\Gamma$-$\bm G$-$\bm{\mathcal I}$-SSet}_*$ of $G$-global $\Gamma$-spaces carries a unique model structure with
		\begin{enumerate}
			\item weak equivalences those maps $f$ such that $f(S_+)$ is a $(G\times\Sigma_S)$-global weak equivalence for every finite set $S$;
			\item acyclic fibrations those maps $f$ such that $f(S_+)(A)$ is a $\mathcal G_{\Sigma_A,G\times\Sigma_S}$-acyclic fibration for all finite sets $S$ and $A$.
		\end{enumerate}
		We call this the \emph{$G$-global level model structure} and its weak equivalences the \emph{$G$-global level weak equivalences}. It is combinatorial, simplicial, and proper. Moreover, pushouts along injective cofibrations (i.e.~levelwise injections) are homotopy pushouts in this model structure; in particular, they preserve weak equivalences.\qed
	\end{proposition}

	For any $\alpha\colon G\to G'$ the functor $\alpha^*\colon \cat{$\bm\Gamma$-$\bm{G'}$-$\bm{\mathcal I}$-SSet}_*\to\cat{$\bm\Gamma$-$\bm G$-$\bm{\mathcal I}$-SSet}_*$ preserves weak equivalences, so the above yields a global category $\ul\GammaS_*^\text{gl}$ by the usual procedure. Note that \cite{CLL_Global} uses the notation `$\ul\GammaS_{\mathcal I,*}^\text{gl}$' instead; however, the above is equivalent to the category denoted by the same symbols in \emph{op.~cit.} by \cite[Theorem~2.2.33]{g-global}.

	\begin{construction}\label{constr:Gamma-free-forgetful}
		For any $G$, we have a homotopical adjunction
		\begin{equation*}
			\Gamma(1^+,\blank)\smashp(\blank)_+\colon\cat{$\bm G$-$\bm{\mathcal I}$-SSet}\rightleftarrows\cat{$\bm\Gamma$-$\bm G$-$\bm{\mathcal I}$-SSet}_* \noloc\ev_{1^+}.
		\end{equation*}
		As both adjoints are moreover strictly compatible with restriction, we obtain an induced adjunction $\mathbb P\colon\ul\S^\text{gl}\rightleftarrows\ul\GammaS^\text{gl}_*\noloc\mathbb U$. We will refer to $\mathbb U$ as the \emph{forgetful functor}.
	\end{construction}

	\begin{remark}\label{rk:pted-unptd-inclusion}
		We can also consider the category $\cat{$\bm\Gamma$-$\bm G$-$\bm{\mathcal I}$-SSet}$ of all functors $\Gamma\to\cat{$\bm G$-$\bm{\mathcal I}$-SSet}$ and equip this with the analogue of the $G$-global level weak equivalence. For varying $G$, these again assemble into a global category, which we denote by $\ul\GammaS^\text{gl}$.

		The inclusion $\cat{$\bm\Gamma$-$\bm G$-$\bm{\mathcal I}$-SSet}_*\hookrightarrow\cat{$\bm\Gamma$-$\bm G$-$\bm{\mathcal I}$-SSet}$ admits a left adjoint $\Lambda$ given by taking the cofibers of the maps $X(0^+)\to X(S_+)$ induced by the unique maps $i\colon 0^+\to S_+$ in $\Gamma$. As each $X(i)$ is an injective cofibration (as $i$ admits a retraction), this is actually a homotopy cofiber and $\Lambda$ is homotopical. It follows easily that the map $\ul\GammaS_*^\text{gl}\to\ul\GammaS^\text{gl}$ induced by the inclusions is fully faithful with essential image given in degree $G$ by those $X$ with $X(0^+)\simeq *$ in $\S^\text{gl}_G$.
	\end{remark}

	\begin{definition}
		A $G$-global $\Gamma$-space $X$ is called \emph{special} if the Segal map $X(S_+)\to X(1^+)^{\times S}$ is a $(G\times\Sigma_S)$-global weak equivalence for every finite set $S$.
	\end{definition}

	Note that unlike their equivariant counterparts, these are stable under arbitrary restrictions, so they form a global subcategory $\ul\GammaS_*^\text{gl, spc}$.

	\begin{theorem}[See \cite{CLL_Global}*{Corollary~5.3.6}]\label{thm:Gamma-gl-comp}
		There exists an equivalence of global categories $\Xi\colon\ul\GammaS^\textup{gl}\simeq\ul\Fun_{\Glo}(\ul{\mathbb F}^{\Orb}_{\Glo,*},\ul\Spc_{\Glo})$ compatible with the forgetful functors and restricting to an equivalence $\ul\GammaS^\textup{gl, spc}_*\simeq\ul\CMon^{\Orb}_{\Glo}$.\qed
	\end{theorem}

	\begin{corollary}
		The inclusions $\ul\GammaS^\textup{gl, spc}_*\hookrightarrow\ul\GammaS^\textup{gl}_*$ and $\ul\GammaS^\textup{gl, spc}_*\hookrightarrow\ul\GammaS^\textup{gl}$ admit global left adjoints.
		\begin{proof}
			The second statement follows from the previous theorem as $\ul\CMon^{\Orb}_{\Glo}\hookrightarrow\ul\Fun_{\Glo}(\ul{\mathbb F}^{\Orb}_{\Glo,*},\ul\Spc_{\Glo})$ admits a left adjoint. The first one then follows from this together with Remark~\ref{rk:pted-unptd-inclusion}.
		\end{proof}
	\end{corollary}

	\begin{remark}
		One can also prove the corollary via purely model categorical arguments: by \cite[Proposition~2.2.61]{g-global}, the $G$-global level model structure admits a Bousfield localization with fibrant objects the level fibrant \emph{special} $G$-global $\Gamma$-spaces. In particular, we get a pointwise left adjoint, and the Beck--Chevalley condition then translates to demanding that each $\alpha^*\colon\ul\GammaS_*^\text{gl}(G')\to \ul\GammaS_*^\text{gl}(G)$ preserve the weak equivalences of these Bousfield localizations, or equivalently that the restriction functors \begin{equation}\label{eq:pointset-restr}
			\alpha^*\colon\cat{$\bm\Gamma$-$\bm{G'}$-$\bm{\mathcal I}$-SSet}_*\to\cat{$\bm\Gamma$-$\bm{G}$-$\bm{\mathcal I}$-SSet}_*
		\end{equation}
		be homotopical for the localized model structures. While this is doable by careful inspection, it is actually more work than in the equivariant case (as the maps we localize at are more complicated), and hence deliberately avoided in \cite{g-global}, which is why we went via the above route instead.

		Note however that conversely the above corollary now shows that the functor $\alpha^*\colon\ul\GammaS_*^\text{gl}(G')\to \ul\GammaS_*^\text{gl}(G)$ and hence also $(\ref{eq:pointset-restr})$ is homotopical for any $\alpha$, yielding an $\infty$-categorical proof of a model categorical statement.
	\end{remark}

	Composing the above with the adjunction from Construction~\ref{constr:Gamma-free-forgetful}, we get an adjunction $\ul\S^\text{gl}\rightleftarrows\ul\GammaS^\text{gl, spc}_*$ that we again denote by $\mathbb P\dashv\mathbb U$. The (inverse) equivalence $\ul\CMon^{\Orb}_{\Glo}\simeq\ul\GammaS^\text{gl}_*$ from Theorem~\ref{thm:Gamma-gl-comp} can then be described (by some easy mate yoga) as the unique left adjoint that sends $\mathbb P(*)$ to $\mathbb P(*)$.

	\subsubsection{The comparison} Finally, let us relate $G$-global and $G$-equivariant $\Gamma$-spaces to each other:

	\begin{proposition}\label{prop:Lconst-Gamma}
		There is a global functor $\cat{L}\const\colon\ul\GammaS_*^\textup{spc}\to\ul\GammaS^\textup{gl, spc}_*$ with the following properties:
		\begin{enumerate}
			\item It is fully faithful and sends ${\mathbb P}(*)$ to $\mathbb P(*)$.
			\item It admits an $\Orb$-right adjoint.
		\end{enumerate}
	\end{proposition}

	Once again, after the universal property of $\ul\GammaS_*$ is established, we will see a posteriori that the above adjunction is actually unique.

	For the proof of the proposition we will need another model structure:

	\begin{lemma}[See {\cite[Corollary~2.2.40 and proof of Proposition~2.2.42]{g-global}}]\label{lemma:injective-model-Gamma}
		The category $\cat{$\bm\Gamma$-$\bm G$-$\bm{\mathcal I}$-SSet}_*$ admits a model structure with
		\begin{enumerate}
			\item weak equivalences the $G$-global level weak equivalences
			\item cofibrations the injective cofibrations.
		\end{enumerate}
		We call this the \emph{injective $G$-global level model structure}. It is combinatorial, simplicial, and proper. Moreover, if $\alpha\colon G\to G'$ is an \emph{injective} homomorphism, then $\alpha^*\colon(\cat{$\bm\Gamma$-$\bm{G'}$-$\bm{\mathcal I}$-SSet}_*)_\textup{injective}\to(
		\cat{$\bm\Gamma$-$\bm G$-$\bm{\mathcal I}$-SSet}_*)_\textup{injective}$ is right Quillen.\qed
	\end{lemma}

	\begin{proof}[Proof of Proposition~\ref{prop:Lconst-Gamma}]
		For every $G$, we have a Quillen adjunction
		\begin{equation}\label{eq:const-Gamma}
			\const\colon\cat{$\bm\Gamma$-$\bm G$-SSet}_*\rightleftarrows(\cat{$\bm\Gamma$-$\bm G$-$\bm{\mathcal I}$-SSet}_*)_\text{injective}\noloc\ev_{\emptyset},
		\end{equation}
		see \cite[Proposition~2.2.25]{g-global}. By Ken Brown's Lemma, we in particular see that $\const$ sends $G$-equivariant weak equivalences between cofibrant objects to $G$-global level weak equivalences, so we get an induced global functor $\cat{L}\const\colon\ul\GammaS_*\to\ul\GammaS^\text{gl}_*$, which we can postcompose with the localization to $\ul\GammaS^\text{gl, spc}_*$. Note that this sends $\Gamma(1^+,\blank)$ to $\mathbb P(*)$ by direct inspection.

		We now claim that this descends to $\ul\GammaS^\text{spc}_*$, which amounts to saying that the left adjoint in $(\ref{eq:const-Gamma})$ sends $G$-equivariant weak equivalences of cofibrant objects to $G$-global weak equivalences, for which it is in turn enough that the right derived functor $\cat{R}\ev_\emptyset$ preserve specialness. However, by \emph{loc.~cit.} this right adjoint is equivalent to $\ev_{\mathcal U}$ for our favourite complete $G$-set universe $\mathcal U$, and it is clear that the latter has the required property (also see Lemma~2.2.51 of \emph{op.~cit.}). Altogether, we therefore get a functor $\cat{L}\const \colon\ul\GammaS^\text{spc}_*\to \ul\GammaS^\text{gl, spc}_*$ sending ${\mathbb P}(*)$ to $\mathbb P(*)$; moreover, this is fully faithful as the right adjoint $\mathcal R$ of the right adjoint $\cat{R}\ev_\emptyset$ is fully faithful by Theorem~2.2.59 of \emph{op.~cit.}

		It only remains to show that the pointwise right adjoints $\cat{R}\ev_\emptyset$ assemble into an $\Orb$-right adjoint, i.e.~that for any \emph{injective} homomorphism $\alpha\colon G\to G'$ the Beck--Chevalley transformation $\cat{L}\alpha^*\circ\cat{R}\ev_\emptyset\Rightarrow\cat{R}\ev_\emptyset\circ\alpha^*$ is an equivalence.

		However, the \emph{pointset level} Beck--Chevalley map $\alpha^*\circ\ev_\emptyset\Rightarrow\ev_\emptyset\circ\alpha^*$ is clearly an isomorphism, and all functors in question are right Quillen by the above together with Lemmas~\ref{lemma:gamma-spec-adj} and~\ref{lemma:injective-model-Gamma}, so this already models the derived Beck--Chevalley map when restricted to injectively fibrant objects.
	\end{proof}

	\subsection{Proof of Theorems~\ref{thm:gamma-comparison} and~\ref{thm:gamma-univ-prop}}\label{subsec:Gamma-semiadd-proofs} Finally, we turn to the universal property of $\ul\GammaS^\text{spc}_*$.

	\begin{lemma}\label{lemma:GammaS-generated}
		The category $\ul\GammaS_*^\textup{spc}(G)$ is generated under (non-parametrized) co\-limits by the $G$-equivariant $\Gamma$-spaces $\Gamma(1^+,\blank)\smashp G/H_+$ for $H\subset G$.
		\begin{proof}
			Inspecting the generating cofibrations from Proposition~\ref{prop:equiv-level-model-struct} we see that $\ul\GammaS_*(G)$ is generated under colimits by the $\Gamma(S_+,\blank)\smashp G/H_+$ for finite $G$-sets $S$ and subgroups $H\subset G$. Thus, these objects also generate $\ul\GammaS_*^\text{spc}(G)$. However, in the latter $\Gamma(S_+,\blank)\smashp G/H_+\simeq S_+\smashp\Gamma(1^+,\blank)\smashp G/H_+$ by Remark~\ref{rk:special-generating-we}. The claim follows by decomposing the $G$-set $G/H\times S$ into its orbits.
		\end{proof}
	\end{lemma}

	Note that $\ul\GammaS_*^\text{spc}(G)\ni\Gamma(1^+,\blank)\smashp G/H_+\simeq i_!p^*\Gamma(1^+,\blank)$ where $i\colon H\hookrightarrow G$ denotes the inclusion and $p\colon H\to 1$ the unique map. Thus, once we know that $\ul\GammaS_*^\textup{spc}$ is equivariantly cocomplete, the lemma will tell us that it is generated under equivariant colimits by $\mathbb P(*)=\Gamma(1^+,\blank)\in\ul\GammaS^\text{spc}_*(1)$.

	\begin{proof}[Proof of Theorems~\ref{thm:gamma-comparison} and~\ref{thm:gamma-univ-prop}]
		The fully faithful functor $\cat{L}\const$ from Proposition~\ref{prop:Lconst-Gamma} identifies $\ul\GammaS^\text{spc}_*$ with a full subcategory of $\ul\GammaS^\text{gl, spc}_*$, and the latter is globally presentable by Theorem~\ref{thm:Gamma-gl-comp}. However, the essential image of $\cat{L}\const$ is closed under all equivariant colimits as $\cat{L}\const$ has an $\Orb$-right adjoint, so $\ul\GammaS^\text{spc}_*$ is equivariantly cocomplete.

		In particular, there is a unique equivariantly cocontinuous functor $\ul\CMon^{\Orb}_{\Orb\triangleright\Glo}\to\ul\GammaS^\textup{spc}_*$ sending ${\mathbb P}(*)$ to ${\mathbb P}(*)$. We claim that this is an equivalence, for which it will be enough to construct \emph{some} equivalence preserving ${\mathbb P}(*)$. To this end, we will show that the composite
		\begin{equation}\label{eq:cmon-orb-Gamma-gl}
			\ul\CMon^{\Orb}_{\Orb\triangleright\Glo}\xrightarrow{\,\iota_!\,}\ul\CMon^{\Orb}_{\Glo}\xrightarrow{\;\Xi\;}\ul\GammaS^\text{gl, spc}_*
		\end{equation}
		of the fully faithful functor from Theorem~\ref{thm:CMon-ST-adj} and the equivalence from Theorem~\ref{thm:Gamma-gl-comp} (which sends ${\mathbb P}(*)$ to $\mathbb P(*)$ by construction) restricts to an equivalence $\ul\CMon^{\Orb}_{\Orb\triangleright\Glo}\simeq\essim(\cat L\const)\eqqcolon\mathcal E$. On the one hand, the source of $(\ref{eq:cmon-orb-Gamma-gl})$ is generated under equivariant colimits by ${\mathbb P}(*)$, so that $(\ref{eq:cmon-orb-Gamma-gl})$ factors through $\mathcal E$ as both functors are in particular $\Orb$-left adjoints. On the other hand, Lemma~\ref{lemma:GammaS-generated} shows that $\ul\GammaS^\text{spc}_*$ and hence also $\mathcal E$ is generated by $\mathbb P(*)$, so this restriction is also essentially surjective, hence an equivalence.

		Finally, the universal property of $\ul\GammaS^\text{spc}_*$ follows immediately from this equivalence and the universal property of $\ul\CMon^{\Orb}_{\Orb\triangleright\Glo}$ established in Theorem~\ref{thm:CMon-ST-univ-prop}.
	\end{proof}

	\section{The stable story}\label{sec:stable}
	As in the previous sections, we fix a cleft category $\iota\colon S \hookrightarrow T$. The goal of this section is to establish the stable analogues of the results from \Cref{sec:semiadd}. We begin with the fiberwise (or na\"ive) version of stability:

	\begin{definition}\label{defi:fib-stable}
		A $T$-category $\Cc$ is called \emph{fiberwise stable} if factors through the non-full subcategory $\Cat^\text{st}\subset\Cat$ of stable categories and exact functors.
	\end{definition}

	\begin{construction}
		Recall \cite{HA}*{Proposition 1.4.4.4, Example 4.8.1.23} that the inclusion $\Pr^\text{L,$\,$st}\hookrightarrow\Pr^\text{L}$ of presentable stable categories and left adjoints into all presentable categories admits a left adjoint, given by tensoring with the category $\Sp$ of spectra.

		If now $\mathcal C$ is a fiberwise presentable $T$-category, then we write $\Sp\otimes\mathcal C$ for the composite
		\begin{equation*}
			T^\op\xrightarrow{\;\mathcal C\;}\Pr\nolimits^\text{L}\xrightarrow{\Sp\otimes\blank}\Pr\nolimits^\text{L,$\,$st}\subset\Cat
		\end{equation*}
		and call it the \emph{left fiberwise stabilization} of $\mathcal C$. It comes with a functor $\Sigma^\infty\colon \mathcal C\to \Sp\otimes\mathcal C$ induced by the unit of the adjunction $\Pr^\text{L}\rightleftarrows\Pr^\text{L,$\,$st}\!$.
	\end{construction}

	\begin{remark}
		There is another way to fiberwise stabilize suitable $T$-categories, which we will refer to as \emph{right fiberwise stabilization} below: if $\mathcal C$ factors through the non-full subcategory $\Cat^\text{lex}$ of pointed categories with finite limits and left exact functors, then we can define $\ul\Sp^{\fib}(\mathcal C)$ by composing with the \emph{right} adjoint to the inclusion $\Cat^\text{st}\hookrightarrow\Cat^\text{lex}$ of stable categories. This is the perspective taken in \cite{CLL_Global}*{Subsection~6.1}.

		For the $T$-categories which we would like to stabilize, such as $\ul\CMon^P_{S\triangleright T}$, it is not clear whether the restriction functors preserve finite limits (as a consequence of the example in Warning~\ref{warn:sp-no-fiberwise-limits} below, they cannot preserve general limits). Therefore $\ul\Sp^{\fib}(\mathcal C)$ is not well-defined, and we cannot sensibly ask for $\Sp\otimes\mathcal C$ to agree with $\ul\Sp^{\fib}(\mathcal C)$.

		However, on the category $\Pr^\text{L,$\,$lex}$ of pointed presentable categories and left exact left adjoints, the two stabilization constructions agree \cite{HA}*{Example~4.8.1.23}. Thus, whenever we are given some subcategory $T'\subset T$ such that $\Cc|_{T'}$ is pointed and restrictions in $\Cc$ along maps in $T'$ \emph{are} left exact, then $(\Sp\otimes\mathcal C)|_{T'}$ agrees with $\ul\Sp^{\fib}(\mathcal C|_{T'})$. This will allow us below to still apply the results from \cite{CLL_Global}*{Section~6} to the present situation.
	\end{remark}

	\begin{lemma}
		Let $\Cc$ be a fiberwise presentable $T$-category. Then $\Sp\otimes\mathcal C$ is fiberwise presentable and fiberwise stable. Moreover, for every fiberwise cocomplete and fiberwise stable $\mathcal D$, restriction along $\Sigma^\infty$ defines an equivalence
		\begin{equation}\label{eq:spectr-fstab-cc}
			\ul\Fun_T^\textup{fib-cc}(\Sp\otimes\mathcal C,\mathcal D)\to\ul\Fun_T^\textup{fib-cc}(\mathcal C,\mathcal D)
		\end{equation}
		of $T$-categories of fiberwise cocontinuous functors.
		\begin{proof}
			It is clear that $\Sp\otimes\mathcal C$ is fiberwise presentable and fiberwise stable. Replacing $\mathcal D$ by $\ul\Fun_T(\mathcal T,\mathcal D)$ for small $\mathcal T\in\Cat_T$, it will suffice for the universal property to show that the induced map
			\begin{equation*}
				\maps_{\Cat_T^\text{fib-cc}}(\Sp\otimes\mathcal C,\mathcal D)\to \maps_{\Cat_T^\text{fib-cc}}(\mathcal C,\mathcal D)
			\end{equation*}
			of mapping spaces in the category $\Cat_T^\text{fib-cc}\coloneqq\Fun(T^\op,\Cat^\text{cc})$ of fiberwise cocomplete $T$-categories and fiberwise colimit-preserving functors is an equivalence.

			Writing both sides as the ends of the mapping spaces in $\Cat^\text{cc}$, it then suffices to consider the case $T=1$, i.e.~that for any cocomplete stable $\Dd$ restriction along $\Cc\to\Sp\otimes\Cc$ defines an equivalence $\maps^{\text{cc}}(\Sp\otimes\Cc,\Dd)\simeq\maps^{\text{cc}}(\Cc,\Dd)$. Using that the tensor product of presentable categories agrees with the tensor product of cocomplete categories \cite{HA}*{Proposition~4.8.1.15}, the tensor-hom adjunction reduces to the case $\Cc=\Spc$, i.e.~we want to show that evaluation at the sphere defines an equivalence $\maps^{\text{cc}}(\Sp,\Dd)\simeq\iota\Dd$, where $\iota\Dd$ again denotes the groupoid core of $\Dd$. This however follows at once by exhibiting $\Sp$ as the $\mathop\text{Ind}$-completion of the Spanier--Whitehead category \cite{SAG}*{Remark~C.1.1.6} and noting that right exact functors out of the latter classify objects by \cite{SAG}*{Proposition~C.1.1.7} together with \cite{HA}*{Proposition~1.4.2.21}.
		\end{proof}
	\end{lemma}

	\begin{lemma}
		Assume $\mathcal C$ is $S$-presentable. Then $\Sp\otimes\mathcal C$ is again $S$-presentable, hence in particular $S$-cocomplete. Moreover, if also $\mathcal D$ is $S$-cocomplete, then $(\ref{eq:spectr-fstab-cc})$ restricts to an equivalence
		\begin{equation*}
			\ul\Fun_T^\textup{$S$-cc}(\Sp\otimes\mathcal C,\mathcal D)\iso\ul\Fun_T^\textup{$S$-cc}(\mathcal C,\mathcal D).
		\end{equation*}
		\begin{proof}
			From the previous lemma we see that $\Sp\otimes \Cc$ is fiberwise presentable. If $s\colon A\to B$ is in $S$, then the adjunction $s_!\colon\mathcal C(A)\rightleftarrows\mathcal C(B)\noloc s^*$ is an internal adjunction in $\Pr^\text{L}$ (as $s^*$ is itself a left adjoint by fiberwise presentability), so we get an induced adjunction $\Sp\otimes s_!\dashv\Sp\otimes s^*$ by $2$-functoriality of the tensor product. Moreover, the Beck--Chevalley conditions for $\Sp\otimes\mathcal C$ follow immediately from the ones for $\mathcal C$.

			Finally, for the universal property it suffices by the previous lemma and replacing $\Dd$ by $\Dd^A$ to show that $\Sigma^\infty$ preserves $S$-colimits and that for any $S$-cocontinuous $F\colon\mathcal C\to\mathcal D$ also the lift $\tilde F\colon\Sp\otimes\mathcal C\to\mathcal D$ is $S$-cocontinuous.

			For the first statement, we observe that $\Sigma^\infty$ is clearly fiberwise cocontinuous, and that for any admissible $f\colon X\to Y$ in $\PSh(T)$ the Beck--Chevalley maps are equivalences by the explicit description of the adjoints $f_!\colon(\Sp\otimes\Cc)(X)\to(\Sp\otimes\Cc)(Y)$ given above.

			For the second statement, we first observe that $\tilde F$ is fiberwise cocontinuous by definition. Given now any admissible $f\colon X\to Y$, the mate of the total square in
			\begin{equation*}
				\begin{tikzcd}
					\Cc(Y)\arrow[d, "\Sigma^\infty"'] &[1em] \arrow[l, "f^*"'] \Cc(X)\arrow[d, "\Sigma^\infty"]\\
					\Sp\otimes\Cc(Y)\arrow[d, "\tilde F"'] & \arrow[l, "f^*"{description}] \Sp\otimes\Cc(X)\arrow[d, "\tilde F"]\\
					\Dd(Y)&\arrow[l,"f^*"] \Dd(X)
				\end{tikzcd}
			\end{equation*}
			is the Beck--Chevalley map $f_!F\Rightarrow Ff_!$, hence an equivalence by $S$-cocontinuity of $F$, and similarly the mate of the top square is an equivalence by the above. By the compatibility of mates with pastings, we conclude that the Beck--Chevalley map $f_!\tilde F\Rightarrow \tilde Ff_!$ becomes an equivalence after precomposition with $\Sigma^\infty\colon\Cc(X)\to\Sp(\Cc(X))$. However, both $f_!\tilde F$ and $\tilde Ff_!$ are cocontinuous functors, so the claim follows from the universal property of $\Sp\otimes\blank$ (cf.~the previous lemma).
		\end{proof}
	\end{lemma}

	Let us restate the key step in the above proof separately for easy reference:

	\begin{corollary}\label{cor:check-cocontinuity-on-Sigma}
		Let $\mathcal C,\mathcal D$ be as above. Then a fiberwise cocontinuous functor $F\colon\Sp\otimes\mathcal C\to\mathcal D$ is $S$-cocontinuous if and only if $F\circ\Sigma^\infty\colon\mathcal C\to\mathcal D$ is so.\qed
	\end{corollary}

	Finally, let us move to the setting of \emph{genuine} stability:

	\begin{definition}\label{defi:P-stable}
		Let $P\subset T$ be atomic orbital. A $T$-category $\Cc$ is called \emph{$P$-stable} if it is $P$-semiadditive (Definition~\ref{defi:P-semiadd}) and fiberwise stable (Definition~\ref{defi:fib-stable}).
	\end{definition}

	\begin{lemma}
		Let $P\subset S$ be atomic orbital in $T$ and let $\mathcal C$ be a $P$-semiadditive $S$-cocomplete $T$-category. Then $\Sp\otimes\mathcal C$ is $P$-stable.
		\begin{proof}
			We already know that $\Sp\otimes\mathcal C$ is $S$-cocomplete and fiberwise stable. Moreover, its underlying $S$-category is $P$-semiadditive by \cite[Lemma~6.2.6]{CLL_Global}, so $\Sp\otimes\mathcal C$ is also $P$-semiadditive as a $T$-category by Remark~\ref{rk:P-semiadd-underlying}.
		\end{proof}
	\end{lemma}

	\begin{definition}
		We define $\ul\Sp^P_{S\triangleright T}\coloneqq\Sp\otimes\ul\CMon^P_{S\triangleright T}$, and we write $\Sigma^\infty_+$ for the composite
		\begin{equation*}
			\ul\Spc_{S\triangleright T}\xrightarrow{\;{\mathbb P}\;} \ul\CMon^P_{S\triangleright T}\xrightarrow{\Sigma^\infty}\Sp\otimes\ul\CMon^P_{S\triangleright T}=\ul\Sp^P_{S\triangleright T}.
		\end{equation*}
	\end{definition}

	\begin{remark}
		Note that $\Sigma^\infty_+$ is by construction an extension of the $S$-functor $\Sigma^\infty_+\coloneqq\Sigma^\infty\circ\mathbb P\colon\ul\Spc_S\to\Sp\otimes\ul\CMon^P_S=\ul\Sp^P_S$ from \cite[Definition~6.2.12]{CLL_Global}.
	\end{remark}

	Combining the above fiberwise results with the universal property of $\ul\CMon^P_{S\triangleright T}$ from Theorem~\ref{thm:CMon-ST-univ-prop} we get:

	\begin{theorem}\label{thm:sp-univ-property}
		The $T$-category $\ul\Sp^P_{S\triangleright T}$ is $S$-presentable and $P$-stable. For any $S$-cocomplete $P$-stable $T$-category $\mathcal D$ evaluation at $\mathbb S\coloneqq\Sigma^\infty_+(*)$ induces an equivalence $\ul\Fun_T^\textup{$S$-cc}(\ul\Sp^P_{S\triangleright T},\mathcal D)\simeq\mathcal D$.\qed
	\end{theorem}

	We can also compare this to $\ul\Sp^P_T$:

	\begin{theorem}\label{thm:sp-adjunction}
		The essentially unique $S$-cocontinuous functor $\iota_!\colon\ul\Sp^P_{S\triangleright T}\to\ul\Sp^P_{T}$ preserving ${\mathbb S}$ is fully faithful. Moreover, it admits an $S$-right adjoint $\iota^*$, which in turn admits a further $S$-right adjoint $\iota_*$ (again fully faithful for formal reasons).
		\begin{proof}
			The functor $\Sp\otimes\iota_!\colon\ul\Sp^P_{S\triangleright T}=\Sp\otimes\ul\CMon^P_{S\triangleright T}\to\Sp\otimes\ul\CMon^P_T=\ul\Sp^P_T$ admits an $S$-right adjoint given by $\Sp\otimes\iota^*$ (as $\iota^*$ is itself an $S$-left adjoint). For each $A\in T$ the unit $\id\to\Sp\otimes(\iota^*\iota_!)$ is then induced by the unit of $\iota_!\dashv\iota^*$, so it is an equivalence as $\iota_!$ is fully faithful (Theorem~\ref{thm:CMon-ST-adj}). Thus, also $\Sp\otimes\iota_!$ is fully faithful. Moreover, it sends $\Sigma^\infty_+(*)$ to $\Sigma^\infty_+(*)$ simply by naturality, so this is the functor $\ul\Sp^P_{S\triangleright T}\to\ul\Sp^P_T$ in question.

			It only remains to show that also $\Sp\otimes\iota^*$ admits an $S$-right adjoint. However, by construction it admits a pointwise right adjoint, and it is moreover $S$-cocontinuous as a consequence of Corollary~\ref{cor:check-cocontinuity-on-Sigma} (for $T=S$), so the claim follows.
		\end{proof}
	\end{theorem}

	\section{The universal property of equivariant spectra}\label{sec:spectra}
	In this section, we will describe the universal equivariantly presentable equivariantly stable (i.e.~$\Orb$-stable) global category in terms of classical equivariant stable homotopy theory.

	\subsection{$\bm G$-equivariant spectra}\label{subsec:spectra}
	We start by introducing the global category of equivariant spectra, and state our main results.
	\begin{definition}
		We write $\cat{Spectra}$ for the $1$-category of \emph{symmetric spectra} \cite{hss} in simplicial sets. For any finite $G$, we write $\cat{$\bm G$-Spectra}$ for the category of $G$-objects; by slight abuse of language, we will refer to its objects simply as $G$-spectra.
	\end{definition}

	We refer the reader to \cite[Definition~2.35]{hausmann-equivariant} for the definition of the \emph{$G$-stable weak equivalences} of $G$-spectra. Below, we will simply refer to these as \emph{$G$-equivariant weak equivalences}.

	\begin{proposition}[See {\cite[Theorem~4.8 and Proposition~4.9]{hausmann-equivariant}}]
		The category $\cat{$\bm G$-Spectra}$ carries a model structure with
		\begin{enumerate}
			\item weak equivalences the $G$-equivariant weak equivalences
			\item acyclic fibrations those maps $f\colon X\to Y$ such that the level $n$-map $f_n\colon X_n\to Y_n$ is a $\mathcal G_{G,\Sigma_n}$-acyclic fibration for every $n\ge0$.
		\end{enumerate}
		We call this the \emph{$G$-equivariant projective model structure}. It is combinatorial and stable.\qed
	\end{proposition}

	All that we will need to know about this model structure below is that the sphere spectrum is cofibrant, which follows from \cite[discussion after Corollary~2.26]{hausmann-equivariant} or by simply observing that the above acyclic fibrations are surjective in degree $0$ and hence have the right lifting property against $0\to\mathbb S$.

	\begin{lemma}\label{lemma:restr-spectra-left-Quillen}
		Let $\alpha\colon G\to G'$ be any homomorphism. Then $\alpha^*\colon\cat{$\bm{G'}$-Spectra}\to\cat{$\bm G$-Spectra}$ is left Quillen with respect to the above model structures.
		\begin{proof}
			Factoring $\alpha$, we may assume that it is either injective or surjective. In the first case, the claim is an instance of \cite[5.2]{hausmann-equivariant}, while in the latter case it follows by combining 5.3 and 5.1 of \emph{op.~cit.}
		\end{proof}
	\end{lemma}

	As before, we therefore get a global category $\ul\mySp$ with $\mySp_G\coloneqq\ul\mySp(G)$ the localization of (projectively cofibrant) $G$-spectra at the $G$-weak equivalences, and with structure maps given by the left derived functors $\cat{L}\alpha^*$. We will refer to this as the \emph{global category of equivariant spectra}. It has a natural section ${\mathbb S}$ given by the equivariant sphere spectra (determined by the usual sphere in $\mySp_1$).

	Using this, we can now state our main results:

	\begin{theorem}\label{thm:mysp-univ-prop}
		The global category $\ul\mySp$ is equivariantly presentable and equivariantly stable. For any other equivariantly cocomplete equivariantly stable $\mathcal D$ evaluation at ${\mathbb S}$ defines an equivalence $\ul\Fun^\textup{eq-cc}_{\Glo}(\ul\mySp,\mathcal D)\simeq\mathcal D$.
	\end{theorem}

	\begin{theorem}\label{thm:mysp-comp}
		The essentially unique equivariantly cocontinuous global functor $\ul\Sp^{\Orb}_{\Orb\triangleright\Glo}\to\ul\mySp$ sending ${\mathbb S}$ to ${\mathbb S}$ is an equivalence.
	\end{theorem}

	The proof will be given at the end of this section. For now let us stop to observe that some pleasant properties one might have hoped for $\ul\Sp^{P}_{S\triangleright T}$ to satisfy do not hold even for $\ul\Sp^{\Orb}_{\Orb\triangleright\Glo}\simeq \ul\mySp$:

	\begin{warning}\label{warn:sp-no-finite-limits}
		For any $f\colon G\to G'$ the functor $\cat{L}f^*\colon\mySp_{G'}\to\mySp_G$ admits a right adjoint $\cat{R}f_*$ by Lemma~\ref{lemma:restr-spectra-left-Quillen}. However, these do \emph{not} satisfy the Beck--Chevalley condition in general (i.e.~$\ul\mySp$ does not have finite \emph{global} products). To see this, consider the pullback
		\begin{equation*}
			\begin{tikzcd}
				\mathbb Z/2\times\mathbb Z/2\arrow[dr,phantom,"\lrcorner"{very near start,xshift=-5pt}]\arrow[d, "\pr_1"']\arrow[r, "\pr_2"] & \mathbb Z/2\arrow[d, "q"]\\
				\mathbb Z/2\arrow[r, "q"'] & 1
			\end{tikzcd}
		\end{equation*}
		in $\Glo$, giving rise to a map $\cat{L}q^*\cat{R}q_*X\to\cat{R}{\pr_{2^*}}\cat{L}{\pr_1^*}X$ for any $X\in\mySp_{\mathbb Z/2}$; we will now show that this cannot be an equivalence for $X=\mathbb S$ by computing the result of applying $\cat{R}q_*$ to both sides:

		The functor $\cat{R}q_*$ is given by taking categorical $\mathbb Z/2$-fixed points, so the tom Dieck-splitting \cite{tomDieck} tells us that
		\begin{equation*}
			\cat{R}q_*\mathbb S\simeq\bigvee_{G\subset\mathbb Z/2}\Sigma^\infty_+ B\big((\mathbb Z/2)/G\big).
		\end{equation*}
		The right hand side is actually cofibrant, so $\cat{L}q^*\cat{R}q_*\mathbb S$ is simply given by equipping this with the trivial $\mathbb Z/2$-action. Accordingly, another application of the tom Dieck splitting shows
		\begin{equation*}
			\cat{R}q_*\cat{L}q^*\cat{R}q_*\mathbb S\simeq\bigvee_{G\subset\mathbb Z/2}\bigvee_{H\subset\mathbb Z/2}\Sigma^\infty_+\big(B\big((\mathbb Z/2)/G\big)\times B\big((\mathbb Z/2)/H\big)\big).
		\end{equation*}
		If we take $\pi_0$, then each wedge summand contributes a summand of $\mathbb Z$ (being the unreduced suspension of a connected space), so $\pi_0(\cat{R}q_*\cat{L}q^*\cat{R}q_*\mathbb S)$ is free abelian of rank $4$.

		On the other hand, by uniqueness of adjoints $\cat{R}q_*\cat{R}{\pr_{2*}}$ agrees with $\cat{R}r_*$ for $r\colon\mathbb Z/2\times\mathbb Z/2\to 1$ the unique map, so $\cat{R}q_*\cat R{\pr_{2*}}\cat{L}{\pr_1^*}\mathbb S$ is given by the categorical $(\mathbb Z/2\times\mathbb Z/2)$-fixed points of $\mathbb S$. By another application of the tom Dieck splitting (or using the classical computation of the zeroth equivariant homotopy groups of $\mathbb S$ as the Burnside ring), we therefore see that $\pi_0(\cat{R}q_*\cat{R}{\pr_{2*}}\cat{L}{\pr_1^*}\mathbb S)\cong\cramped{\pi_0^{\smash{\mathbb Z/2\times\mathbb Z/2}}(\mathbb S)}$ is free abelian of rank the number of subgroups of $\mathbb Z/2\times\mathbb Z/2$, which is $5$ instead of $4$.
	\end{warning}

	\begin{remark}
		The extra $\mathbb Z$-summand in $\pi_0(\cat{R}q_*\cat{R}{\pr_{2*}}\cat{L}{\pr_1^*}\mathbb S)$ can be attributed to the fact that $\mathbb Z/2\times\mathbb Z/2$ has a subgroup that is not given as a product of subgroups of its factors, namely the diagonal subgroup. A similar phenomenon appears for general $G$, and as observed in \cite{nickel} this is what prevents the tom Dieck map
		\begin{equation}\label{eq:tom-Dieck}
			\bigvee_{(H\subset G)/\text{conj.}}\Sigma^\infty\big(E(W_GH)\smashp_{W_GH}X^H\big)\to F^G\Sigma^\infty X
		\end{equation}
		for a pointed $G$-simplicial set $X$ from being a \emph{global weak equivalence} instead of just a non-equivariant weak equivalence: after taking categorical $K$-fixed points on both sides, the left hand side only contains the wedge summands of the tom Dieck splitting of $F^{K\times G}\Sigma^\infty X$ corresponding to subgroups of the form $L\times H\subset K\times G$ for $L\subset K,H\subset G$. In fact, this is the only obstruction to $(\ref{eq:tom-Dieck})$ being a global weak equivalence, see \emph{op.~cit.} for details.
	\end{remark}

	\begin{warning}\label{warn:sp-no-fiberwise-limits}
		$\ul\mySp$ is neither globally cocomplete nor fiberwise complete, and hence neither is $\ul\Sp^{\Orb}_{\Orb\triangleright\Glo}$ by \Cref{thm:mysp-comp}. In fact, already the restriction functor $\cat{L}q^*\colon\mySp_{1}\to\mySp_{\mathbb Z/2}$ induced by the unique map $q\colon\mathbb Z/2\to1$ does not preserve all products, and in particular it does not admit a left adjoint. The third author learned the following argument for this fact from Denis Nardin: By \cite{GrothendieckNeeman16}*{Theorem 3.3}, $\cat{L}q^*$ preserves all products if and only if $\cat{R}q_*$ preserves compact objects. However, as observed above $\cat{R}q_*\mathbb S$ contains $\Sigma^\infty_+ B(\Z/2)$ as a wedge summand. As the latter is not compact, neither is $\cat{R}q_*\mathbb S$, yielding the desired contradiction. A similar argument shows that $\cat{L}q^*$ does not have a left adjoint whenever $q$ has a non-trivial kernel.
	\end{warning}

	\subsection{$\bm G$-global spectra} As before, the proof of Theorems~\ref{thm:mysp-univ-prop} and~\ref{thm:mysp-comp} will proceed via comparison with a model of the universal \emph{globally} presentable equivariantly stable category.

	\begin{definition}
		A map $f\colon X\to Y$ in $\cat{$\bm G$-Spectra}$ is called a \emph{$G$-global weak equivalence} if $\alpha^*f$ is an $H$-equivariant weak equivalence for every finite group $H$ and every homomorphism $\alpha\colon H\to G$.
	\end{definition}

	We emphasize that we are \emph{not} deriving $\alpha^*$ here with respect to the equivariant model structures (as otherwise this would of course simply recover the $G$-weak equivalences again).

	\begin{proposition}[See {\cite[Corollary~3.1.46--Proposition~3.1.48]{g-global}}]
		The category $\cat{$\bm G$-Spectra}$ admits a model structure with
		\begin{enumerate}
			\item weak equivalences the $G$-global weak equivalences
			\item cofibrations the injective cofibrations.
		\end{enumerate}
		We call this the \emph{injective $G$-global model structure}. It is combinatorial, simplicial, proper, and stable.\qed
	\end{proposition}

	Basically by definition, the restriction functors $\alpha^*\colon\cat{$\bm{G'}$-Spectra}\to\cat{$\bm G$-Spectra}$ are homotopical and left Quillen. In particular, we again obtain a global category $\ul\mySp^\text{gl}$.

	\begin{theorem}[See {\cite[Corollary~7.3.3]{CLL_Global}}]\label{thm:mysp-global-equivalence}
		$\ul\mySp^\textup{gl}$ is globally presentable and equivariantly stable. The essentially unique globally cocontinuous functor $\ul\Sp^{\Orb}_{\Glo}\to\ul\mySp^\textup{gl}$ sending $\Sigma^\infty_+(*)$ to the global sphere spectrum $\mathbb S$ is an equivalence.\qed
	\end{theorem}

	\subsection{Proof of Theorems~\ref{thm:mysp-univ-prop} and~\ref{thm:mysp-comp}} Let us begin with a comparison of the above models complementing Theorem~\ref{thm:sp-adjunction}:

	\begin{lemma}
		There is a global functor $\cat{L}\id\colon\ul\mySp\to\ul\mySp^\textup{gl}$ with the following properties:
		\begin{enumerate}
			\item It is fully faithful and sends ${\mathbb S}$ to $\mathbb S$.
			\item It admits an $\Orb$-right adjoint.
		\end{enumerate}
		\begin{proof}
			For any $G$, \cite[Proposition~3.3.1]{g-global} provides a Quillen adjunction
			\begin{equation}\label{eq:spectra-vs-spectra}
				\id\colon\cat{$\bm G$-Spectra}_\text{$G$-equiv.~proj.}\rightleftarrows\cat{$\bm G$-Spectra}_\text{$G$-gl.~inj.} \noloc\id.
			\end{equation}
			In particular, $G$-equivariant weak equivalences between projectively cofibrant spectra are $G$-global weak equivalences (also see Lemma~\ref{lemma:restr-spectra-left-Quillen}), so the inclusion of projectively cofibrant objects yields a functor $\cat{L}\id\colon\ul\mySp\to\ul\mySp^\text{gl}$ sending ${\mathbb S}$ to $\mathbb S$. Moreover, the right adjoint in $(\ref{eq:spectra-vs-spectra})$ evidently induces a localization, so that $\cat{L}\id$ is fully faithful.

			It only remains that the right adjoints assemble into an $\Orb$-right adjoint. However, the pointset level Beck--Chevalley maps $\alpha^*\circ{\id}\Rightarrow{\id}\circ\alpha^*$ are isomorphisms for trivial reasons, and for \emph{injective} $\alpha$, $\alpha^*$ is also homotopical in the equivariant world \cite{hausmann-equivariant}*{5.2}, so that this already models the derived Beck--Chevalley map.
		\end{proof}
	\end{lemma}

	\begin{proof}[Proof of Theorems~\ref{thm:mysp-univ-prop} and~\ref{thm:mysp-comp}]
		By Theorem~\ref{thm:sp-univ-property} it is enough to prove that $\ul\mySp$ is equivariantly stable and equivariantly cocomplete, and that the preferred map $\ul\Sp^{\Orb}_{\Orb\triangleright\Glo}\to\ul\mySp$ is an equivalence.

		For this, let us write $\mathcal E$ for the essential image of $\cat{L}\id\colon\ul\mySp\to\ul\mySp^\text{gl}$; this is then closed under equivariant colimits as $\cat{L}\id$ admits an $\Orb$-right adjoint, and it is closed under desuspension as each $\mySp_G$ is stable. It follows that $\mathcal E$ and hence also $\ul\mySp$ is indeed equivariantly cocomplete and equivariantly stable.

		Now let $F\colon\ul\Sp^{\Orb}_{\Orb\triangleright\Glo}\to\ul\mySp$ be the unique equivariantly cocontinuous functor preserving ${\mathbb S}$. Then ${\cat{L}\id}\circ F\colon\ul\Sp^{\Orb}_{\Orb\triangleright\Glo}\to\ul\mySp^\text{gl}$ is an equivariantly cocontinuous functor sending ${\mathbb S}$ to $\mathbb S$. The same holds for the composite
		\begin{equation*}
			\ul\Sp^{\Orb}_{\Orb\triangleright\Glo}\xrightarrow{\;\iota_!\;}\ul\Sp^{\Orb}_{\Glo}\iso\ul\mySp^\text{gl}
		\end{equation*}
		of the fully faithful functor from Theorem~\ref{thm:gamma-comparison} with the equivalence from Theorem~\ref{thm:mysp-global-equivalence}, so they have to agree by the universal property of $\ul\Sp^{\Orb}_{\Orb\triangleright\Glo}$. In particular, $F$ is fully faithful. To see that it is also essentially surjective, it is by \cite[Proposition~4.9]{hausmann-equivariant} enough to see that it hits the suspension spectra $\Sigma^\infty_+(G/H)$ for all $H\subset G$. However, as before we have $i_!\mathbb S\simeq\Sigma^\infty_+(G/H)$ for $i\colon H\hookrightarrow G$ the inclusion, so the claim follows from the defining properties of $F$.
	\end{proof}

	\section{The $G$-equivariant universal properties}
	Let $G$ be a finite group, taken to be fixed throughout this section. The universal properties of the global categories of equivariant spaces, equivariant $\Gamma$-spaces, and equivariant spectra proved in this article immediately imply analogous universal properties for the $G$-categories of $G$-spaces, $\Gamma$-$G$-spaces, and $G$-spectra. The goal of this section is to explain this reduction step. For ease of reference, we will briefly recall the relevant definitions regarding $G$-equivariant category theory, originally developed in \cite{nardin2016exposeIV}.

	Recall from \Cref{ex:G-cat} the notion of a \textit{$G$-category}
	\[
		\Cc\colon \Orb_G^{\op}\rightarrow \Cat, \qquad G/H \mapsto \Cc(H),
	\]
 	informally given by assigning a category $\Cc(H)$ to every subgroup $H$ of $G$ together with restriction functors $\res^K_H \colon \Cc(K)\rightarrow \Cc(H)$ for $H \subset K\subset G$ and conjugation functors $c_g^*\colon \Cc(H) \rightarrow \Cc(H)$ for $g\in N_G H$ that satisfy various higher coherences.

	\begin{definition}
		A $G$-category $\Cc$ is called \textit{$G$-cocomplete} if it is $\Orb_G$-cocomplete in the sense of \Cref{def:T-cc}. More concretely:
		\begin{itemize}
			\item The category $\Cc(H)$ is cocomplete for every subgroup $H \subset G$ ;
			\item Each restriction functor $\res^K_H\colon \Cc(K)\rightarrow \Cc(H)$ preserves colimits and admits a left adjoint $\ind^K_H$;
			\item The functors $\ind^K_H$ satisfy the \emph{double coset formula}, i.e.\ for $K,K' \subset H$ a certain map
			\[
			\coprod_{[h] \in K' \backslash H / K} \ind_{K'\cap \,{}^h\! K}^H \res_{K'\cap \,{}^h\! K}^{{}^h\! K} c_h^* \to \res_{K'}^H \ind_K^H
			\]
			is an equivalence.
		\end{itemize}
		A $G$-category $\Cc$ is \textit{$G$-presentable} if it is $G$-cocomplete and each $\Cc(H)$ is presentable.
	\end{definition}
	\begin{definition}
		A $G$-category $\Cc$ is called \textit{$G$-semiadditive} if it is $\Orb_G$-semiadditive in the sense of \Cref{defi:P-semiadd}. Explicitly:
		\begin{itemize}
			\item The category $\Cc(H)$ is semiadditive for every subgroup $H \subset G$;
			\item Each restriction functor $\res^K_H$ admits both a left adjoint $\ind^K_H$ and a right adjoint $\coind_H^K$ satisfying the double coset formula;
			\item A canonically defined norm map
			\[
				\Nm_H^K\colon \ind_H^K\rightarrow \coind_H^K
			\]
			is an equivalence for all inclusions $H\subset K \subset G$.
		\end{itemize}
	A $G$-category $\Cc$ is \textit{$G$-stable} if it is $G$-semiadditive, each $\Cc(H)$ is stable.
	\end{definition}

	As announced, we will show below that the universal $G$-categories with these properties are given by the $G$-categories of $G$-spaces, $\Gamma$-$G$-spaces, and $G$-spectra. We may define these $G$-categories by appropriately restricting the previously defined global categories of equivariant spaces, equivariant $\Gamma$-spaces, and equivariant spectra. This restriction happens via a functor $\Orb_G \to \Glo$ informally given by sending the orbit $G/H$ to the group $H$. To make this precise, let us temporarily view $\Orb_G$ as a subcategory of $\Fun(BG,\Spc)$ via the inclusions
	\[
		\Orb_G \subset \Fun(BG,\Set)\subset \Fun(BG,\Spc).
	\]
	Similarly we will view $\Orb$ and $\Glo$ as subcategories of $\Spc$ via the inclusions
	\[
		\Orb \subset \Glo \overset{B}{\hookrightarrow} \Grpd \subset \Spc.
	\]

	\begin{lemma}
		The unstraightening equivalence $(-)_{hG}\colon \Fun(BG,\Spc) \iso \Spc_{/BG}$ restricts to an equivalence $\Orb_G\simeq \Orb_{/G}$.
	\end{lemma}

	In particular, $\Orb_{/G}$ (which a priori is a $(2,1)$-category) happens to be a $1$-category.

	\begin{proof}
		For this proof we will identify $\Glo$ with its image in $\Grpd$, i.e.\ with the category of finite connected groupoids. Under this idenitification, a morphism $f\colon BH \rightarrow BK$ in $\Glo$ lies in $\Orb$ if and only if the map $\pi_1(f)$ is injective. This implies, by left cancellation of injective maps, that $\Orb_{/G}$ is a full subcategory of $\Spc_{/BG}$. Since $\Orb_G$ is similarly a full subcategory of $\Fun(BG,\Spc)$, it thus suffices to show that for a space $X$ with a $G$-action the unstraightening of $X$ is in $\Orb_{/BG}$ if and only if $X$ is equivalent to a transitive $G$-set.

		Recall that the unstraightening of $X$ is given by the map $X_{hG} \to BG$, where $X_{hG}$ denotes the homotopy quotient. Since the underlying space of $X$ sits in a fiber sequence $X\rightarrow X_{hG}\rightarrow BG$, the long exact sequence on homotopy groups shows that $X$ is discrete if and only if $X_{hG}$ is 1-truncated and the maps $\pi_1(X_{hG},[x]) \to \pi_1(BG)$ are injective for every basepoint $[x] \in X_{hG}$. Because of the bijection $\pi_0(X_{hG}) \cong \pi_0(X)/G$ we also see that $X_{hG}$ is connected if and only if the $G$-action on $\pi_0(X)$ is transitive. Combining these two statements, we may conclude.
	\end{proof}

	\begin{construction}
		\label{cons:Underlying-G-category}
		We write $\upsilon_G\colon \Orb_G \to \Glo$ for the composite
		\[
			\Orb_G\simeq\Orb_{/G}\mathrel{\smash{\xrightarrow{\pi_G}}}\Orb\hookrightarrow\Glo,
		\]
		i.e.~this is the homotopy quotient functor $(-)_{hG}\colon\Orb_G\to\Spc$, viewed as a functor into the full subcategory $\Glo\subset\Spc$. Restricting along $\upsilon_G$ then yields a functor $\upsilon_G^*\colon\Cat_{\Glo}\to\Cat_{\Orb_G}$ sending a global category to its `underlying $G$-category.'
	\end{construction}

	\begin{definition}
	We define the $G$-categories
	\begin{enumerate}
		\item $\ul\S_G\coloneqq \upsilon_G^* \ul\S$ of \emph{$G$-spaces}, which sends $G/H$ to the localization of the category \cat{$\bm H$-SSet} of $H$-simplicial sets at the $H$-equivariant weak equivalences;
		\item $\ul\GammaS^\textup{spc}_{G,*} \coloneqq \upsilon_G^*\ul\GammaS^\textup{spc}_*$ of \emph{special $\Gamma$-$G$-spaces}, which sends $G/H$ to the localization of the category $\cat{$\bm \Gamma$-$\bm H$-SSet}_*$ of $\Gamma$-$H$-simplicial sets at the $H$-equivariant weak equivalences;
		\item $\ul \mySp_G \coloneqq \upsilon_G^*\ul\mySp$ of \emph{$G$-spectra}, which sends $G/H$ to the localization of the category \cat{$\bm H$-Spectra} of $H$-symmetric spectra at the $H$-equivariant weak equivalences.
	\end{enumerate}
	\end{definition}

	\begin{remark}
	The reader may consult Sections \ref{sec:spaces}, \ref{subsec:eq-gamma-spaces}, \ref{subsec:spectra} for the more details on the global categories $\ul\S$, $\ul\GammaS^\textup{spc}_*$, and $\ul\mySp$, respectively. In each case the restriction and conjugation functoriality is given before localization by simply restricting and conjugating group actions, respectively.
	\end{remark}

	We will now provide the universal properties of these three $G$-categories. To this end, recall that a $G$-functor $F\colon \Cc \to \Dd$ between two $G$-cocomplete $G$-categories is called \textit{$G$-cocontinuous} if each functor $F(H)\colon \Cc(H) \to \Cc(G)$ is cocontinuous and if $F$ commutes with the induction functors (in the sense that the associated Beck--Chevalley transformations are equivalences). The $G$-cocontinuous functors assemble into a $G$-category $\ul\Fun_{\Orb_G}^{G\textup{-cc}}(\Cc,\Dd)$, see \Cref{def:cocont-functor-cat}.

	\begin{theorem}\label{thm:equivariant-spaces}$ $
		\begin{enumerate}
			\item The $G$-category $\ul\S_G$ is $G$-presentable. For any $G$-cocomplete $\Dd$, evaluation at the terminal object defines an equivalence $\ul\Fun^{G\textup{-cc}}_{\Orb_G}(\ul\S_G,\Dd)\simeq\Dd$.
			\item The $G$-category $\ul\GammaS^\textup{spc}_{G,*}$ is $G$-presentable and $G$-semiadditive. For any $G$-cocomplete $G$-semiadditive $\Dd$, evaluation at $\mathbb P(*)$ induces an equivalence $\ul\Fun_{\Orb_G}^\textup{$G$-cc}(\ul\GammaS^\textup{spc}_{G,*},\Dd)\simeq\Dd$.
			\item The $G$-category $\ul\mySp_G$ is $G$-presentable and $G$-stable. For any $G$-cocomplete $G$-stable $\Dd$, evaluation at $\mathbb S$ defines an equivalence $\ul\Fun_{\Orb_G}^{G\textup{-cc}}(\ul\mySp_G,\Dd)\simeq\Dd$.
		\end{enumerate}
	\end{theorem}

	\begin{proof}
	We begin with (1). It suffices to construct an equivalence $\upsilon_G^* \ul\S \simeq\ul\Spc_{\Orb_G}$; the theorem will then follow from the universal property of the right hand side (Theorem~\ref{thm:univ-prop-T-spc}). By Theorem~\ref{thm:orbi-S-equivalence} we have an equivalence of global categories $\ul\S\simeq\ul\Spc_{\Orb\triangleright\Glo}$, and hence in particular an equivalence $\ul\S|_{\Orb}\simeq\ul\Spc_{\Orb}$ of $\Orb$-categories. To finish the proof it suffices now to observe that for any small $T$ and any $A\in T$, there is an equivalence $\pi_A^*\ul\Spc_T\simeq\ul\Spc_{T_{/A}}$ by \cite{martini2021yoneda}*{Remark~3.7.2}.

	Similarly, to prove (2) it will be enough to construct an equivalence $\ul\CMon_{\Orb_G}\simeq\upsilon_G^*\ul\GammaS^\text{spc}_*$ preserving $\mathbb P(*)$, for which it in turn suffices to combine Proposition~\ref{prop:cmon-slice} with Theorem~\ref{thm:gamma-comparison}.

	Arguing as in the semiadditive case, (3) follows from \Cref{thm:mysp-comp} and the stable analog of Proposition~\ref{prop:cmon-slice}, which follows immediately from it by applying $\Sp\otimes\blank$ to both sides.
	\end{proof}

	\begin{remark}
	Evaluating at $H\subset G$, part (1) in particular gives an equivalence  $\S_H\simeq \PSh(\Orb_H)$. In this sense, it can be viewed as a `coherent' version of the classical Elmendorf Theorem \cite{elmendorf}, additionally taking into account the restriction functors as well as all higher structure between them.

	A proof of (3) has previously been sketched by Nardin as \cite{nardin2016exposeIV}*{Theorem~A.4}.
	\end{remark}

	\appendix

	\section{The calculus of mates}
	\label{sec:Calculus_Of_Mates}

	In this appendix\footnote{This appendix is based on \cite[Appendix F]{Cnossen2023PhD}.} we recall some general identities involving Beck--Chevalley transformations (or \emph{mates}) that are used throughout this work. A standard reference for the calculus of mates is \cite[Section~2.2]{KellyStreet1974Mates}. Some aspects of our treatment are inspired by \cite[Section~2.2]{CSY20}. Throughout, we denote counits generically by $\epsilon$ and units by $\eta$.

	\begin{definition}[Beck--Chevalley transformation]
		\label{def:BeckChevalley}
		Let $\alpha\colon HF \Rightarrow KG$ be a natural transformation of functors as displayed in the following diagram:
		\[\begin{tikzcd}
			\Cc \rar{F} \dar[swap]{G} & \Dd \dar{H} \dlar[Rightarrow, shorten <=4pt, shorten >=4pt, "\alpha"{description}] \\
			\Cc' \rar[swap]{K} & \Dd'\llap.
		\end{tikzcd}\]
		\begin{enumerate}
			\item If the functors $F$ and $K$ have left adjoints $F_!\colon \Dd \to \Cc$ and $K_!\colon \Dd' \to \Cc'$, we define the \textit{left Beck--Chevalley transformation} $\BC_!\colon K_! H \Rightarrow G F_!$ as the composite
			\begin{align*}
				\BC_!(\alpha)\colon K_! H \xRightarrow{\eta_{F}} K_! H F F_! \xRightarrow{\alpha} K_! K G F_! \xRightarrow{\epsilon_{K}} G F_!.
			\end{align*}
			\item If the functors $G$ and $H$ have right adjoints $G_*\colon \Cc' \to \Cc$ and $H_*\colon \Dd' \to \Dd$, we define the \textit{right Beck--Chevalley transformation} $\BC_*\colon F G_* \Rightarrow H_* K$ as the composite
			\begin{align*}
				\BC_*(\alpha)\colon F G_*  \xRightarrow{\eta_{H}} H_* H F G_* \xRightarrow{\alpha} H_* K G G_* \xRightarrow{\epsilon_{G}} H_* K.
			\end{align*}
		\end{enumerate}
		We write $\BC_!$ and $\BC_*$ if the transformation $\alpha$ is left implicit (e.g.\ in case of some canonical equivalence $HF \simeq KG$).
	\end{definition}

	Diagrammatically, the two Beck--Chevalley transformations may be displayed as follows:
	\[\BC_!(\alpha) = \begin{tikzcd}
		\Dd \\
		\Cc & \Dd \\
		{\Cc'} & {\Dd'} \\
		& {\Cc'}
		\arrow["{F_!}"{description}, from=1-1, to=2-1]
		\arrow["=",""{name=0, anchor=center, inner sep=0}, from=1-1, to=2-2, bend left]
		\arrow["F"{description}, from=2-1, to=2-2]
		\arrow["G"{description}, from=2-1, to=3-1]
		\arrow["H"{description}, from=2-2, to=3-2]
		\arrow["{K}"{description}, from=3-1, to=3-2]
		\arrow["{K_!}"{description}, from=3-2, to=4-2]
		\arrow[""{name=1, anchor=center, inner sep=0}, "="', from=3-1, to=4-2, bend right]
		\arrow["\alpha"{description}, shorten <=4pt, shorten >=4pt, Rightarrow, from=2-2, to=3-1]
		\arrow["{\eta}"{description}, shorten <=2pt, Rightarrow, from=0, to=2-1]
		\arrow["{\epsilon}"{description}, shorten >=2pt, Rightarrow, from=3-2, to=1]
	\end{tikzcd}
	\qquadtext{ and }
	\BC_*(\alpha) = \begin{tikzcd}
		{\Cc'} & \Cc & \Dd \\
		& {\Cc'} & {\Dd'} & \Dd\rlap{.}
		\arrow["F"{description}, from=1-2, to=1-3]
		\arrow["G"{description}, from=1-2, to=2-2]
		\arrow["H"{description}, from=1-3, to=2-3]
		\arrow["{K}"{description}, from=2-2, to=2-3]
		\arrow["\alpha"{description}, shorten <=4pt, shorten >=4pt, Rightarrow, from=1-3, to=2-2]
		\arrow[""{name=0, anchor=center, inner sep=0}, "="', from=1-1, to=2-2, bend right]
		\arrow["{G_*}"{description}, from=1-1, to=1-2]
		\arrow["{H_*}"{description}, from=2-3, to=2-4]
		\arrow[""{name=1, anchor=center, inner sep=0}, "=", from=1-3, to=2-4, bend left]
		\arrow["\epsilon"{description}, shorten >=2pt, Rightarrow, from=1-2, to=0]
		\arrow["\eta"{description}, shorten <=2pt, Rightarrow, from=1, to=2-3]
	\end{tikzcd}\]
	We will frequently employ these diagrammatic ways of displaying Beck--Chevalley transformations in proofs, as it is often easier to parse than the more symbolic way of writing these transformations.

	\begin{warning}
		If $\alpha$ is an equivalence, then the notation of \Cref{def:BeckChevalley} is somewhat abusive: if $G$ and $H$ have left adjoints $G_!$ and $H_!$ we may flip the diagram and obtain another transformation $\BC_!\colon H_!K\Rightarrow FG_!$. Analogously if $F$ and $K$ have right adjoints $F_*$ and $K_*$ one obtains a transformation $\BC_*\colon GF_* \Rightarrow K_*H$. It will always be clear from the source and target of the maps $\BC_!$ and $\BC_*$ which version we mean to use.
	\end{warning}

	\begin{lemma}
		\label{lem:MateEquivalence}
		The assignment $\alpha \mapsto \BC_!(\alpha)$ defines a bijection between homotopy classes of natural transformations
		\[
		\BC_!\colon \pi_0\,\mathrm{Nat}(HF,KG) \xrightarrow{\;\cong\;} \pi_0\,\mathrm{Nat}(K_!H,GF_!),
		\]
		whose inverse is given by $\BC_*\colon \pi_0\,\mathrm{Nat}(K_!H,GF_!) \to \pi_0\,\mathrm{Nat}(HF,KG)$.
	\end{lemma}
	\begin{proof}
		It follows directly from the triangle identities that one may recover $\alpha$ from $\BC_!(\alpha)$ by forming the right Beck--Chevalley transformation:
		\[\BC_*(\BC_!(\alpha))\hskip.5em = \begin{tikzcd}
			\Cc & \Dd \\
			& \Cc & \Dd \\
			& {\Cc'} & {\Dd'} \\
			&& {\Cc'} & {\Dd'}
			\arrow["{F_!}"{description}, from=1-2, to=2-2]
			\arrow[""{name=0, anchor=center, inner sep=0}, "=", from=1-2, to=2-3, bend left]
			\arrow["F"{description}, from=2-2, to=2-3]
			\arrow["G"{description}, from=2-2, to=3-2]
			\arrow["H"{description}, from=2-3, to=3-3]
			\arrow["{K}"{description}, from=3-2, to=3-3]
			\arrow["{K_!}"{description}, from=3-3, to=4-3]
			\arrow[""{name=1, anchor=center, inner sep=0}, "="', from=3-2, to=4-3, bend right]
			\arrow["\alpha"{description}, shorten <=4pt, shorten >=4pt, Rightarrow, from=2-3, to=3-2]
			\arrow["F"{description}, from=1-1, to=1-2]
			\arrow[""{name=2, anchor=center, inner sep=0}, "="', from=1-1, to=2-2, bend right]
			\arrow["{K}"{description}, from=4-3, to=4-4]
			\arrow[""{name=3, anchor=center, inner sep=0}, "=", from=3-3, to=4-4, bend left]
			\arrow["{\eta}"{description}, shorten <=2pt, Rightarrow, from=0, to=2-2]
			\arrow["{\epsilon}"{description}, shorten >=2pt, Rightarrow, from=3-3, to=1]
			\arrow["{\epsilon}"{description}, shorten >=2pt, Rightarrow, from=1-2, to=2]
			\arrow["{\eta}"{description}, shorten <=2pt, Rightarrow, from=3, to=4-3]
		\end{tikzcd}
		\hskip-5pt
		\simeq
		\hskip.5em
		\begin{tikzcd}
			\Cc \rar{F} \dar[swap]{G} & \Dd \dar{H} \dlar[Rightarrow, shorten <=4pt, shorten >=4pt, "\alpha"{description}] \\
			\Cc' \rar[swap]{K} & \Dd'.
		\end{tikzcd} \hskip.5em = \hskip.5em \alpha.
		\]
		One may similarly prove that $\BC_!(\BC_*(\beta)) \simeq \beta$ for a transformation $\beta\colon K_!H \Rightarrow GF_!$.
	\end{proof}

	\begin{remark}
		A more elaborate argument shows that the above can be lifted to an equivalence of \emph{spaces} of natural transformations, see \cite{HHLNa}*{Theorem~5.3.5}.
	\end{remark}

	\begin{lemma}
		\label{lem:BeckChevalleyCompatibleWithCoUnits}
		Consider the situation of \Cref{def:BeckChevalley}.
		\begin{enumerate}
			\item If the functors $F$ and $K$ have left adjoints $F_!\colon \Dd \to \Cc$ and $K_!\colon \Dd' \to \Cc'$, then the following diagrams commute:
			\[
			\begin{tikzcd}
				{K_! HF} & {K_!KG} \\
				{GF_!F} & G
				\arrow["{\BC_!(\alpha)}"', from=1-1, to=2-1]
				\arrow["{\epsilon_F}", from=2-1, to=2-2]
				\arrow["{\epsilon_{K}}", from=1-2, to=2-2]
				\arrow["\alpha", from=1-1, to=1-2]
			\end{tikzcd}
			\qquadtext{ and }
			\begin{tikzcd}
				H & {KK_! H} \\
				{H F F_!} & {KGF_!.}
				\arrow["{\BC_!(\alpha)}", from=1-2, to=2-2]
				\arrow["\alpha", from=2-1, to=2-2]
				\arrow["{\eta_{K}}", from=1-1, to=1-2]
				\arrow["{\eta_{F}}"', from=1-1, to=2-1]
			\end{tikzcd}
			\]
			\item If the functors $G$ and $H$ have right adjoints $G_*\colon \Cc' \to \Cc$ and $H_*\colon \Dd' \to \Dd$, then the following diagrams commute:
			\[
			\begin{tikzcd}
				{HFG_*} & {KGG_*} \\
				{HH_*K} & {K}
				\arrow["{\BC_*(\alpha)}"', from=1-1, to=2-1]
				\arrow["{\epsilon_H}", from=2-1, to=2-2]
				\arrow["{\epsilon_{G}}", from=1-2, to=2-2]
				\arrow["\alpha", from=1-1, to=1-2]
			\end{tikzcd}
			\qquadtext{ and }
			\begin{tikzcd}
				F & {FG_*G} \\
				{H_*H F} & {H_*KG.}
				\arrow["{\BC_*(\alpha)}", from=1-2, to=2-2]
				\arrow["\alpha", from=2-1, to=2-2]
				\arrow["{\eta_{G}}", from=1-1, to=1-2]
				\arrow["{\eta_{H}}"', from=1-1, to=2-1]
			\end{tikzcd}
			\]
		\end{enumerate}
	\end{lemma}
	\begin{proof}
		Just like the proof of \Cref{lem:MateEquivalence}, this follows directly from the triangle identities.
	\end{proof}

	\begin{lemma}[Pasting laws for Beck--Chevalley transformations]
		\label{lem:PastingLemmaBeckChevalley}
		Consider natural transformations $\alpha$, $\beta$ and $\gamma$ as in the following diagram:
		\[
		\begin{tikzcd}
			\Cc \rar{F} \dar[swap]{G} & \Dd \dar["H"{description}] \rar{E} \dlar[Rightarrow, shorten <=4pt, shorten >=4pt, "\alpha"{description}] & \Ee \dar{K} \dlar[Rightarrow, shorten <=4pt, shorten >=4pt, "\beta"{description}] \\
			\Cc' \rar["K"{description}] \dar[swap]{G'} & \Dd' \rar[swap]{E'} \dar{H'} \dlar[Rightarrow, shorten <=4pt, shorten >=4pt, "\gamma"{description}] & \Ee' \\
			\Cc' \rar[swap]{F^{\prime\prime}} & \Dd''.
		\end{tikzcd}
		\]
		\begin{enumerate}
			\item If the functors $F$, $K$, $E$, and $E'$ have left adjoints $F_!$, $K_!$, $E_!$, and $E'_!$, then the composite
			\[
			(E'K)_!K \simeq K_!E'_!K \xrightarrow{\BC_!(\beta)} K_!HE_! \xrightarrow{\BC_!(\alpha)} GF_!E_! \simeq G(EF)_!
			\]
			is homotopic to $\BC_!(\alpha\beta)$.
			\item If the functors $G$, $H$, and $K$ admit right adjoints $G_*$, $H_*$, and $K_*$, then the composite
			\[
			EFG_* \xrightarrow{\BC_*(\beta)} EH_*K \xrightarrow{\BC_*(\alpha)} K_*E'K
			\]
			is homotopic to $\BC_*(\alpha\beta)$.
			\item If the functors $F$, $K$, and $F^{\prime\prime}$ have left adjoints $F_!$, $K_!$, and $F^{\prime\prime}_!$, then the composite
			\[
			F^{\prime\prime}_!H'H \xrightarrow{\BC_!(\gamma)} G'K_!H \xrightarrow{\BC_!(\alpha)} G'GF_!
			\]
			is homotopic to $\BC_!(\alpha\gamma)$.
			\item If the functors $G$, $G'$, $H$, and $H'$ admit right adjoints $G_*$, $G'_*$, $H_*$, and $H'_*$, then the composite
			\[
			F(G'G)_* = FG_*G'_* \xrightarrow{\BC_*(\gamma)} H_*KG'_* \xrightarrow{\BC_*(\alpha)} H_*H'_*F^{\prime\prime} = (H'H)_*F^{\prime\prime}
			\]
			is homotopic to the $\BC_*(\alpha\gamma)$.
		\end{enumerate}
	\end{lemma}
	\begin{proof}
		This follows directly from the definitions and the triangle identities; see \cite[Proposition~2.2]{KellyStreet1974Mates} for details.
	\end{proof}

	We may specialize the definition of Beck--Chevalley transformation to the case where some of the functors are the identity. In this case, the resulting transformation is known as the \textit{total mate}:

	\begin{definition}[Total mate transformation]
		\label{def:MateTransformations}
		Consider two functors $L,L'\colon \Cc \to \Dd$ that admit right adjoints $R,R' \colon \Dd \to \Cc$.
		\begin{enumerate}
			\item Given a transformation $\alpha\colon L \to L'$, we obtain a \textit{(right) total mate} $\overline{\alpha} = \BC_*(\alpha)\colon R' \to R$;
			\item Given a transformation $\beta\colon R' \to R$, we obtain a \textit{(left) total mate} $\overline{\beta} = \BC_!(\beta) \colon L \to L'$.
		\end{enumerate}
		It follows from \Cref{lem:MateEquivalence} that $\overline{\alpha} \simeq \beta$ if and only if $\overline{\beta} \simeq \alpha$, in which case we will say that $\alpha$ and $\beta$ are \textit{total mates} of each other.
	\end{definition}

	\begin{lemma}[Total mates of composites]
		\label{lem:TotalMateOfComposite}
		Consider three functors $L,L',L''\colon \Cc \to \Dd$ with right adjoints $R$, $R'$, and $R''$. Then for all transformations $\alpha\colon L \to L'$ and $\alpha'\colon L' \to L''$, there is an equivalence
		\[
		\overline{\alpha' \circ \alpha} \simeq \overline{\alpha'} \circ \overline{\alpha^{\vphantom{\prime}}} \quad \in \quad \pi_0\,\mathrm{Nat}(R'',R).
		\]
	\end{lemma}
	\begin{proof}
		This is a special case of \Cref{lem:PastingLemmaBeckChevalley}.
	\end{proof}

	\begin{corollary}[Total mates of inverses]
		\label{cor:TotalMateOfInverse}
		Consider functors $L,L' \colon \Cc \to \Dd$ with right adjoints $R$ and $R'$. Then a transformation $\alpha\colon L \to L'$ is an equivalence if and only if its total mate $\overline{\alpha}\colon R' \to R$ is an equivalence. In this case, the inverse of $\overline{\alpha}\colon R' \to R$ is the total mate of $\alpha^{-1}\colon L' \to L$.
	\end{corollary}
	\begin{proof}
		This is immediate from \Cref{lem:TotalMateOfComposite}.
	\end{proof}

	Given a transformation $\alpha$ as in \Cref{def:BeckChevalley}, if both $\BC_!(\alpha)$ and $\BC_*(\alpha)$ are defined, then these transformations are the total mates of each other:

	\begin{lemma}
		\label{lem:MateOfBeckChevalley}
		Consider again a natural transformation $\alpha\colon HF \to KG$ as in \Cref{def:BeckChevalley}, and assume that the following conditions are satisfied:
		\begin{enumerate}
			\item The functors $F$ and $K$ have left adjoints $F_!\colon \Dd \to \Cc$ and $K_!\colon \Dd' \to \Cc'$;
			\item The functors $G$ and $H$ have right adjoints $G_*\colon \Cc' \to \Cc$ and $H_*\colon \Dd' \to \Dd$.
		\end{enumerate}
		Then the left Beck--Chevalley transformation $\BC_!(\alpha)\colon K_! H \Rightarrow G F_!$ is the total mate of the right Beck--Chevalley transformation $\BC_*(\alpha)\colon F G_* \Rightarrow H_* K$. In particular the former transformation is an equivalence if and only if the latter is.
	\end{lemma}
	\begin{proof}
		The total mate of $\BC_!(\alpha)\colon K_! H \Rightarrow G F_!$ is defined as the image of $\BC_!(\alpha)$ under the bijection
		\[
		\BC_*\colon \pi_0\,\mathrm{Nat}(K_! H,GF_!) \xrightarrow{\;\cong\;} \pi_0\,\mathrm{Nat}(FG_*,H_*K).
		\]
		Since $K_! H$ and $GF_!$ are composites of adjoints, this map factors as
		\[
		\pi_0\,\mathrm{Nat}(K_! H,GF_!) \xrightarrow[\cong]{\,\BC_*\,} \pi_0\,\mathrm{Nat}(HF,KG) \xrightarrow[\cong]{\,\BC_*\,} \pi_0\,\mathrm{Nat}(FG_*,H_*K).
		\]
		Since the first map sends $\BC_!(\alpha)$ to $\alpha$ by \Cref{lem:MateEquivalence}, the claim follows.
	\end{proof}
	\bibliographystyle{amsalpha}
	\bibliography{reference}

\newcommand{\etalchar}[1]{$^{#1}$}
\providecommand{\bysame}{\leavevmode\hbox to3em{\hrulefill}\thinspace}
\providecommand{\MR}{\relax\ifhmode\unskip\space\fi MR }
\providecommand{\MRhref}[2]{%
  \href{http://www.ams.org/mathscinet-getitem?mr=#1}{#2}
}
\providecommand{\href}[2]{#2}
\begin{thebibliography}{BDG{\etalchar{+}}16}

\bibitem[ABFJ22]{abfj-left-exact}
Mathieu Anel, Georg Biedermann, Eric Finster, and Andr{\'e} Joyal,
  \emph{Left-exact localizations of {{\(\infty\)}}-topoi. {I}: {Higher}
  sheaves}, Adv. Math. \textbf{400} (2022), 64pp., Id/No 108268.

\bibitem[BD20]{balmerAmbrogio_Mackey}
Paul Balmer and Ivo Dell'Ambrogio, \emph{Mackey 2-functors and {Mackey}
  2-motives}, EMS Monogr. Math., European Mathematical Society (EMS), 2020.

\bibitem[BDG{\etalchar{+}}16]{exposeI}
Clark Barwick, Emanuele Dotto, Saul Glasman, Denis Nardin, and Jay Shah,
  \emph{{Parametrized higher category theory and higher algebra: Exposé I --
  Elements of parametrized higher category theory}}, arXiv:1608.03657 (2016).

\bibitem[BDS16]{GrothendieckNeeman16}
Paul Balmer, Ivo Dell'Ambrogio, and Beren Sanders, \emph{Grothendieck-{Neeman}
  duality and the {Wirthm{\"u}ller} isomorphism}, Compos. Math. \textbf{152}
  (2016), no.~8, 1740--1776.

\bibitem[CLL23]{CLL_Global}
Bastiaan Cnossen, Tobias Lenz, and Sil Linskens, \emph{{Parametrized stability
  and the universal property of global spectra}}, arXiv:2301.08240 (2023).

\bibitem[Cno23]{Cnossen2023PhD}
Bastiaan Cnossen, \emph{Twisted ambidexterity in equivariant homotopy theory:
  Two approaches}, Ph.D. thesis, University of Bonn, 2023.

\bibitem[CSY22]{CSY20}
Shachar Carmeli, Tomer~M. Schlank, and Lior Yanovski, \emph{Ambidexterity in
  chromatic homotopy theory}, Invent. Math. \textbf{228} (2022), no.~3,
  1145--1254.

\bibitem[Elm83]{elmendorf}
Anthony~D. Elmendorf, \emph{Systems of fixed point sets}, Trans. Amer. Math.
  Soc. \textbf{277} (1983), 275--284.

\bibitem[Hau17]{hausmann-equivariant}
Markus Hausmann, \emph{{{\(G\)}}-symmetric spectra, semistability and the
  multiplicative norm}, J. Pure Appl. Algebra \textbf{221} (2017), no.~10,
  2582--2632.

\bibitem[Hau19]{hausmann-global}
\bysame, \emph{{Symmetric spectra model global homotopy theory of finite
  groups}}, {Algebr. Geom. Topol.} \textbf{19} (2019), no.~3, 1413--1452.

\bibitem[HHLN23]{HHLNa}
Rune Haugseng, Fabian Hebestreit, Sil Linskens, and Joost Nuiten,
  \emph{Two-variable fibrations, factorisation systems and
  {{\(\infty\)}}-categories of spans}, Forum Math. Sigma \textbf{11} (2023),
  70, Id/No e111.

\bibitem[Hil24]{hilman2022parametrised}
Kaif Hilman, \emph{Parametrised presentability over orbital categories}, Appl.
  Categ. Struct. \textbf{32} (2024), no.~3, 53pp., Id/No 15.

\bibitem[HSS00]{hss}
Mark {Hovey}, Brooke {Shipley}, and Jeff {Smith}, \emph{{Symmetric spectra}},
  {J. Amer. Math. Soc.} \textbf{13} (2000), no.~1, 149--208.

\bibitem[HY17]{descent-lim}
Asaf Horev and Lior Yanovski, \emph{On conjugates and adjoint descent},
  Topology Appl. \textbf{232} (2017), 140--154.

\bibitem[KS74]{KellyStreet1974Mates}
G.~Maxwell Kelly and Ross Street, \emph{Review of the elements of
  2-categories}, Category {Sem}., {Proc}., {Sydney} 1972/1973, {Lect}. {Notes}
  {Math}. 420, 75-103 (1974), 1974.

\bibitem[Len25]{g-global}
Tobias Lenz, \emph{{{\(G\)}-Global Homotopy Theory and Algebraic
  {{\(K\)}}-Theory}}, Mem. Amer. Math. Soc. \textbf{306} (2025), no.~1545.

\bibitem[Lin24]{Linskens2023globalization}
Sil Linskens, \emph{Globalizing and stabilizing global $\infty$-categories},
  arXiv:2401.02264 (2024).

\bibitem[LNP22]{LNP}
Sil Linskens, Denis Nardin, and Luca Pol, \emph{Global homotopy theory via
  partially lax limits}, to appear in \textit{Geom.~Topol.}, arXiv:2206.01556
  (2022).

\bibitem[Lur09]{HTT}
Jacob Lurie, \emph{Higher topos theory}, Ann.~Math.~Stud., vol. 170, Princeton
  University Press, Princeton, NJ, 2009, updated version available at
  \url{https://people.math.harvard.edu/~lurie/papers/highertopoi.pdf}.

\bibitem[Lur17]{HA}
\bysame, \emph{Higher algebra},
  \url{https://www.math.ias.edu/~lurie/papers/HA.pdf} (2017).

\bibitem[Lur18]{SAG}
\bysame, \emph{{Spectral Algebraic Geometry}}, {under construction (version
  dated February 2018), \url{www.math.ias.edu/~lurie/papers/SAG-rootfile.pdf}}
  (2018).

\bibitem[Mar21]{martini2021yoneda}
Louis Martini, \emph{Yoneda's lemma for internal higher categories},
  arXiv:2103.17141 (2021).

\bibitem[MMO17]{mmo}
J.~Peter {May}, Mona {Merling}, and Ang{\'e}lica~M. {Osorno},
  \emph{{Equivariant Infinite Loop Space Theory, the Space Level Story}}, {to
  appear in \textit{Mem.~Amer.~Math.~Soc.}, {arXiv:1704.03413}} (2017).

\bibitem[MW22]{martiniwolf2022presentable}
Louis Martini and Sebastian Wolf, \emph{Presentable categories internal to an
  $\infty$-topos}, arXiv:2209.05103 (2022).

\bibitem[MW24]{martiniwolf2021limits}
\bysame, \emph{Colimits and cocompletions in internal higher category theory},
  High. Struct. \textbf{8} (2024), no.~1, 97--192.

\bibitem[Nar16]{nardin2016exposeIV}
Denis Nardin, \emph{{Parametrized higher category theory and higher algebra:
  Expos\'e IV -- Stability with respect to an orbital $\infty$-category}},
  arXiv:1608.07704 (2016).

\bibitem[Nar17]{nardin2017thesis}
\bysame, \emph{Stability and distributivity over orbital $\infty$-categories},
  Ph.D. thesis, Massachusetts Institute of Technology, 2017.

\bibitem[{Nic}22]{nickel}
Jana {Nickel}, \emph{{The tom Dieck splitting in global equivariant stable
  homotopy theory}}, Master's thesis, {Rheinische
  Friedrich-Wilhelms-Universität Bonn}, 2022.

\bibitem[Ost16]{ostermayr}
Dominik Ostermayr, \emph{Equivariant {{\(\Gamma\)}}-spaces}, Homology Homotopy
  Appl. \textbf{18} (2016), no.~1, 295--324.

\bibitem[Sch18]{schwede2018global}
Stefan Schwede, \emph{Global homotopy theory}, {New Math. Monogr.}, vol.~34,
  Cambridge University Press, 2018.

\bibitem[Sha23]{shah2021parametrized}
Jay Shah, \emph{Parametrized higher category theory}, Algebr. Geom. Topol.
  \textbf{23} (2023), no.~2, 509--644.

\bibitem[{Shi}89]{shimakawa}
Kazuhisa {Shimakawa}, \emph{{Infinite loop $G$-spaces associated to monoidal
  $G$-graded categories}}, {Publ. Res. Inst. Math. Sci.} \textbf{25} (1989),
  no.~2, 239--262.

\bibitem[{Shi}91]{shimakawa-simplify}
\bysame, \emph{{A note on $\Gamma_G$-spaces}}, Osaka J. Math. \textbf{28}
  (1991), no.~2, 223--228.

\bibitem[{Ste}16]{cellular}
Marc {Stephan}, \emph{{On equivariant homotopy theory for model categories}},
  {Homology Homotopy Appl.} \textbf{18} (2016), no.~2, 183--208.

\bibitem[tD75]{tomDieck}
Tammo tom Dieck, \emph{Orbittypen und {\"a}quivariante {Homologie}. {II}},
  Arch. Math. \textbf{26} (1975), 650--662.

\end{thebibliography}
\end{document}